%% file: main_article.tex
\documentclass[hidelinks,onefignum,onetabnum]{siamart220329}

\input{ex_shared}

\ifpdf
\hypersetup{
  pdftitle={The Boosted Double-Proximal Subgradient Algorithm for Nonconvex Optimization},
  pdfauthor={F.J. Arag\'on-Artacho, P. P\'erez-Aros, and D. Torregrosa-Bel\'en}
}
\fi

\begin{document}

\maketitle

\begin{abstract}
In this paper we introduce the Boosted Double-proximal Subgradient Algorithm (BDSA), a novel splitting algorithm designed to address general structured nonsmooth and nonconvex mathematical programs expressed as sums and differences of composite functions. BDSA exploits the combined nature of subgradients from the data and proximal steps, and integrates a line-search procedure to enhance its performance. While BDSA encompasses existing schemes proposed in the literature, it extends its applicability to more diverse problem domains. We establish the convergence of BDSA under the Kurdyka--{\L}ojasiewicz property and provide an analysis of its convergence rate.
To evaluate the effectiveness of BDSA, we introduce a novel family of challenging test functions with an abundance of critical points. We conduct comparative evaluations demonstrating its ability to effectively escape non-optimal critical points. Additionally, we present two practical applications of BDSA for testing its efficacy, namely, a constrained minimum-sum-of-squares clustering problem and a nonconvex generalization of Heron's problem.
\end{abstract}



\section{Introduction}

In this paper we are concerned with the structured nonconvex optimization problem
\begin{equation}\label{eq:P1}
    \min_{x\in\mathbb{R}^n} f(x) +g(x)-\sum_{i =1 }^p h_i(\Psi_i (x)), \tag{P}
\end{equation}
where the function $f:\R^n\to\R$ is locally Lipschitz and satisfies the descent lemma, $g:\R^n\to{]-\infty,+\infty]}$ is lower-semicontinuous and prox-bounded, $h_i:\R^{m_i}\to\R$ are convex continuous functions and $\Psi_i:\R^n\to\R^{m_i}$ are differentiable functions with Lipschitz continuous gradients, for $i=1,\ldots,p$ (see Assumption~\ref{assump:1} for more specific details).

Problems in this form appear in a broad variety of fields such as machine learning, image recovery or signal processing~\cite{wen2018survey, lou2015weighted,zhang2009some}.
Numerous algorithms have been developed to handle simpler  instances of~\eqref{eq:P1} (see, e.g., \cite{an2017convergence, bot2019doubleprox, doi:10.1137/060655183, sun2003proximal}), but as far as we are aware of, not for the more general problem that we address in this work.

We present a new splitting algorithm, named \emph{Boosted Double-proximal Subgradient Algorithm} (BDSA), that makes use of the inherent structure of problem~\eqref{eq:P1}. More specifically, the method employs the subdifferential of $f$, the gradients of $\Psi_i$ and the \emph{proximal point operators} of the functions $g$ and $h_i$, for $i=1,\ldots,p$. Subsequently, an optional linesearch can be performed to obtain the final update, which intends to steer (or \emph{``boost''}) the iteration to a point with a reduced value of the objective function, in a similar manner than the linesearch introduced in~\cite{aragon2018accelerating, boostedDCA} permits to accelerate the \emph{Difference of Convex functions Algorithm} (DCA)~\cite{Oliveira2020,TAO1986249, pham2014recent}.

In our numerical tests, we show that the addition of the linesearch provides  major improvements in the performance of the method. On the one hand, it \emph{accelerates} the algorithm,  significantly reducing both the number of iterations and the time that it needs to converge. On the other hand, it may help the sequence to converge to \emph{better solutions}. Indeed, note that the algorithms employed for tackling this class of nonconvex problems usually converge to critical points (see Section~\ref{sect:BDSA}). Being a critical point is a necessary condition for local optimality, but not sufficient. Therefore, the algorithms often converge to critical points which are not local minima. We show that the linesearch introduced in our scheme helps  the method to converge to local minima by avoiding some of the non-desirable critical points.

The paper is structured as follows. We begin by introducing the notation and some basic notions of variational analysis in Section~\ref{sect:preliminaries}. In Section~\ref{sect:BDSA}, we present our algorithm and study its convergence properties in Subsection~\ref{subsec:convergence}. In Subsection~\ref{subsec:KL} we make use of the Kurdyka--{\L}ojasiewicz property to prove its global convergence and deduce some convergence rates.
Section~\ref{sect:numerical} contains multiple numerical experiments demonstrating the good performance of BDSA. The benefit of the linesearch for escaping non-minimal critical points is demonstrated with the introduction of a new family of challenging test functions in Subsection~\ref{subsect:avoidingcriticalpoints}. The reduction in time and iterations is illustrated with an  application of the \emph{minimum sum-of-squares clustering problem} and a generalization of the classical \emph{Heron problem} in Subsections~\ref{sect:numerical2} and~\ref{sect:numerical3}, respectively. We finish with some concluding remarks in Section~\ref{sec:conc}.

\section{Preliminaries and Notational Conventions}\label{sect:preliminaries}
 Throughout this paper, we use $\mathbb{R}^n$ to denote the Euclidean space of dimension $n$, while we denote the extended-real-valued line by $\Rex := \mathbb{R}\cup\{-\infty,+\infty\}$, and set $1/{0}=+\infty$ by convention. The notations $\|\cdot\|$ and $\langle \cdot, \cdot\rangle$ represent the Euclidean norm and inner product in $\mathbb{R}^n$, respectively. We use $\ball_\epsilon(\bar{x})$ to denote the closed ball of radius $\epsilon>0$ centered at $\bar{x}$.

Given $p$ positive integers $m_1,\ldots, m_p$, the inner product of the product space $ \mathbb{R}^{m_1}\times\cdots\times\mathbb{R}^{m_p}$ is defined by
\begin{equation*}
 \langle (x_1,\ldots,x_p),(y_1,\ldots,y_p)\rangle := \sum_{i=1}^p \langle x_i,y_i\rangle  \quad \text{for all } (x_1,\ldots,x_p),(y_1,\ldots,y_p)\in\Rnm,
\end{equation*}
with $m := \sum_{i=1}^p m_i$, and its induced norm is denoted by $\|(x_1,\ldots,x_p)\|$.
Vectors in product spaces will be marked with bold, e.g., $\mathbf{x} = (x_1,\ldots,x_p)\in\mathbb{R}^m$.

\subsection{Notions of Variational Analysis}

 Given some constant $L\geq 0$, a vector-valued function $F: C\subseteq \mathbb{R}^n \to \R^m$ is said to be  $L$-\emph{Lipschitz continuous} on $C$ if
 \begin{equation*}
 	\|F(x)-F(y)\| \leq L \|x-y\|\quad \text{for all } x,y\in C,
 \end{equation*}
 and \emph{locally Lipschitz continuous} around $\bar x\in C$ if it is Lipschitz continuous in some neighborhood of $\bar x$.

The  \emph{domain} of a function $f:\mathbb{R}^n\to\Rex$ is defined as $\dom f := \{ x \in\mathbb{R}^n : f(x) < +\infty\}$, and we say that $f$ is \emph{proper}  if it does not attain the value $-\infty$ and $ \dom f\neq \emptyset$. The function $f$  is \emph{lower-semicontinuous} (\emph{l.s.c.}) at some point $\bar{x}\in\mathbb{R}^n$ if $\liminf_{x\to\bar{x}} f(x) \geq f(\bar{x})$. The \emph{Dini upper directional derivative} of $f$ at some point $\bar{x}\in\dom f$ in the direction $d\in\R^n$ is defined as
$$d^{+} f(\bar{x} ; d):=\limsup_{t \downarrow 0} \frac{f(\bar{x}+t d)-f(\bar{x})}{t}.$$

Given $\bar{x}\in  f^{-1}(\R)$, the \emph{regular subdifferential} of the function $f$ at $\bar{x}\in\mathbb{R}^n$ is the closed and convex set of \emph{regular subgradients}
 \begin{equation*}
  \hat{\partial} f(\bar{x}) := \left\{ v\in\mathbb{R}^n : f(x) \geq f(\bar{x}) + \langle v, x-\bar{x}\rangle + o(\|x-\bar{x}\|) \right\}.
\end{equation*}
We say that $v\in\mathbb{R}^n$ is a (\emph{basic}) \emph{subgradient} of $f$ at $\bar{x}$ if there exists sequences $(x^k)_{k\in\mathbb{N}}$ and $(v^k)_{k\in\mathbb{N}}$, with $v^k\in\hat{\partial}f(x^k)$ for all $k\in\mathbb{N}$, such that $x^k\to\bar{x}$, $f(x^k)\to f(\bar{x})$ and $v^k\to v$, as $k \to \infty$. The (\emph{basic}) \emph{subdifferential} of $f$ at $\bar{x}$, denoted by  $\partial f(\bar{x})$, is the set of all subgradients of $f$ at $\bar{x}$. We use the convention $ \hat{\partial} f(\bar{x})={\partial} f(\bar{x}):=\emptyset$ if $|f(\bar x) |= +\infty$. We say that $f$ is \emph{lower regular} at $\bar{x}\in\dom f$ if $\partial f(\bar{x})=\hat{\partial} f(\bar{x})$.

Let us recall that the regular and basic subdifferentials coincide for convex functions and are equal to the convex subdifferential, which in an abuse of notation we denote by
\begin{equation*}
 \partial f(\bar{x}) := \left\{ v \in\mathbb{R}^n : f(x) \geq f(\bar{x}) + \langle v, x-\bar{x} \rangle \right\}.
\end{equation*}

If a function $f:\mathbb{R}^n\to\mathbb{R}$ is strictly differentiable at a point $\bar{x}\in\R^n$ then all its subdifferentials are singletons and coincide with its gradient, which is denoted by $\nabla f(\bar{x})$. We also use the same symbol to denote the transpose of the Jacobian matrix of a multivariable function $F=(F_1,\ldots,F_m):\mathbb{R}^n\to\mathbb{R}^m$, i.e., the matrix of gradients
$ \nabla F(x) = \left( \nabla F_1(x), \ldots, \nabla F_m(x)\right)$.

Given a function $f:\mathbb{R}^n\to\Rex$, its convex conjugate is the function $f^*:\mathbb{R}^n\to\Rex$ given by
\begin{equation*}
 f^*(x)  := \sup_{u\in\mathbb{R}^n} \left\{ \langle u,x\rangle -f(u) \right\}.
\end{equation*}
If $f$ is proper, the \emph{Fenchel--Moreau} theorem states that $f$ is convex and l.s.c. if and only if $f=f^{**}$, where $f^{**}:=(f^*)^*$ denotes the biconjugate of $f$.

Given a proper l.s.c. function $f:\mathbb{R}^n\to\Rex$ and some constant $\gamma>0$, the \emph{proximal mapping} (also known as the \emph{proximity operator} or \emph{proximal point operator}) is the multifunction $\prox_{\gamma f}:\R^n\rightrightarrows\R^n$ defined as the solution set of the optimization problem
\begin{equation*}\label{eq:argmin}
	\prox_{\gamma f}(x) := \argmin_{u\in\mathbb{R}^n} \left\{ f(u) + \frac{1}{2 \gamma} \|u-x\|^2\right\}.
\end{equation*}

A function $f:\mathbb{R}^n \to \overline{\mathbb{R}}$ is said to be \emph{prox-bounded} if there exists some ${\gamma} >0$ such that $f(\cdot)+\frac{1}{2{\gamma}} \|\cdot-x\|^2$ is bounded from below for all $x\in\mathbb{R}^n$. The supremum of the set of all such constants $\gamma$ is called the \emph{prox-boundedness threshold} of $f$ and is denoted by~${\gamma}^f$. For a proper, l.s.c. and prox-bounded function, the proximal mapping $\prox_{\gamma f}$ has full domain for any $\gamma\in{]0,{\gamma}^f[}$, but it may not be single-valued (see, e.g., \cite[Theorem~1.25]{MR1491362}). It is worth noting that all proper, convex and l.s.c. functions are prox-bounded, but the class is much larger. Any proper l.s.c. function $f:\R^n\to \Rex$ that is bounded from below by an affine function has a threshold of prox-boundedness $\gamma^f=+\infty$ (see~\cite[Example 3.28]{MR1491362}). For example,  the indicator function $\iota_C$ of a nonempty and closed set $C\subseteq\R^n$ is prox-bounded with threshold ${\gamma}^{\iota_C} = +\infty$.

\subsection{The Family of Upper-$\cC^2$ Functions}\label{subsec:upperC2}
The class of functions with a Lipschitz continuous gradient is very important in optimization. Its relevance relies on the fact that these functions verify the so-called \emph{descent lemma} (see, e.g.,~\cite[Lemma A.11]{MR3289054}). Namely, given a differentiable function $f:\mathbb{R}^n\to\mathbb{R}$ with $L_f$-Lipschitz continuous gradient, then the following inequality holds
\begin{equation}\label{eq:descentlemmaineq}
 f(y) \leq f(x) + \langle \nabla f(x),y-x\rangle + \frac{L_f}{2}\|y-x\|^2\quad \text{for all } x,y\in\mathbb{R}^n.
\end{equation}

Different authors have identified a larger family of functions that also satisfies an inequality similar to~\eqref{eq:descentlemmaineq}, preserving thus the same nice properties for optimization (see, e.g. \cite[Theorem~5.1]{MR1363364} and~\cite[Definition~10.29]{MR1491362}). We extend our analysis to this broader class of functions, which we present next.
\begin{definition}
 Let  $V \subseteq \mathbb{R}^n$ be an open, convex and bounded set and let $f :\R^n \to \Rex$ be Lipschitz continuous on $V$. We say that $f$ is $\kappa$-\emph{upper-$\cC^2$} on $V$ for some $\kappa\geq0$ if there exist a compact set $S$ (in some topological space) and some continuous functions $b: S \to \R^n$ and $c: S \to \R$ such that
 \begin{equation}\label{eq:upperC2}
 f(x)= \min\limits_{ s\in S} \left\{ \kappa\| x\|^2 - \langle b(s),x\rangle - c(s) \right\}\quad  \text{for all } x\in V.
 \end{equation}
\end{definition}
From~\eqref{eq:upperC2} it directly follows that $\kappa$-upper-$\cC^2$ functions are DC (difference of convex) functions with the specific DC decomposition
$$f(x)=\kappa \|x\|^2 - \max_{s\in S} \left\{\langle b(s), x\rangle+c(s)\right\},$$
since $x\mapsto \max_{s\in S} \left\{\langle b(s), x\rangle+c(s)\right\}$ is a convex continuous function.

The following result, based on~\cite[Theorem~5.1]{MR1363364}, establishes the relationship between upper-$\cC^2$ functions and the descent lemma. The convex hull of a set is denoted by ``$\conv$''.
\begin{proposition}\label{p:deslemma}
	Let $U$ be an open and convex set such that $f:\R^n\to\Rex$ is locally Lipschitz on $U$.  Then, the following assertions are equivalent for a parameter $\kappa\geq 0$:
	\begin{enumerate}[(i)]
	\item\label{it:p25-i} $f$ is $\kappa$-upper-$\cC^2$ on every bounded subset  of $U$;
    \item\label{it:p25-iii} for all $x\in U$ and $\xi \in \conv{\partial f(x)}$, it holds
	\begin{align}\label{equpperdesc}
		f(y) \leq f(x) + \langle \xi , y-x\rangle + \kappa \| y -x\|^2\quad\text{for all } y \in U;
	\end{align}
	\item\label{it:p25-ii} for each $x\in U$, there exits $\xi \in \mathbb{R}^n$ such that \eqref{equpperdesc} holds.
	\end{enumerate}
\end{proposition}

\begin{proof}
The case of $\kappa=0$ can be straightforwardly established. Therefore, for the remainder of the proof, we assume $\kappa>0$.
 
(\ref{it:p25-i})$\Rightarrow$(\ref{it:p25-iii}) Let $x,y\in U$ and $\xi\in\conv{\partial f(x)}$. Let $V$ be an open, convex and bounded subset of $U$ that contains both $x$ and $y$. By~\cite[Theorem~5.1(a)$\Rightarrow$(c) and Theorem~5.2]{MR1363364} applied to $-f$ and~$V$, it holds
\begin{equation}\label{eq:Psub}
f(w) \leq f(x) + \langle -\zeta, w-x\rangle + \kappa \|w-x\|^2,\quad\forall w\in V
\end{equation}
for all $\zeta \in\partial(-f)(x)=\hat{\partial}( -f)(x)$. By convexity of the regular subdifferential,
\begin{equation}\label{eq:subinclusions}
\partial (-f)(x)=\conv{ \partial (-f)(x)} =-\conv{\partial f(x)}.
\end{equation}
Since $-\xi\in-\conv{\partial f(x)}=  \partial (-f)(x)$, we conclude that~\eqref{eq:Psub} holds for $\zeta:=-\xi$ and $w:=y\in V$, which implies~\eqref{equpperdesc}.

(\ref{it:p25-iii})$\Rightarrow$(\ref{it:p25-ii}) This is straightforward: $\partial f(x)\neq\emptyset$ since $f$ is locally Lipschitz around $x$ (see, e.g., \cite[Theorem 1.22]{MR3823783}).

(\ref{it:p25-ii})$\Rightarrow$(\ref{it:p25-i}) If $V$ is an open, convex and bounded subset of $U$, then~\eqref{equpperdesc} holds on $V$, so \cite[Theorem~5.1 (b)$\Rightarrow$(a)]{MR1363364} implies that $f$ is $\kappa$-upper-$\cC^2$ on V. This completes the proof.
\end{proof}

We conclude this section with some examples of upper-$\cC^2$ functions, the first of which was considered in \cite[Lemma~5.2]{ferreira2021boosted}. A local variation of this result can be found in~\cite[Lemma~3.6]{aragonartacho2023coderivativebased}.
\begin{example}[Difference of an $L$-smooth function and a convex function]\label{example01}
 Let $f_1:\mathbb{R}^n\to\mathbb{R}$ be a differentiable function whose gradient is $L_{f_1}$-Lipschitz continuous and let $f_2:\mathbb{R}^n\to\mathbb{R}$ be a convex and l.s.c. function. Then, the function $f:= f_1-f_2$ is $L_{f_1}/2$-upper-$\cC^2$ on every bounded set of $\mathbb{R}^n$. Indeed, observe first that~\eqref{eq:descentlemmaineq} holds for $f_1$. Moreover, for any $v\in\partial f_2(x)$, by definition of the convex subdifferential, we have
 \begin{equation*}
  -f_2(y) \leq -f_2(x) + \langle -v,y-x\rangle, \text{ for all } y\in\mathbb{R}^n.
 \end{equation*}
 Adding together~\eqref{eq:descentlemmaineq} for $f_1$ and the above inequality, we get that
 \begin{equation*}
  f(y) \leq f(x) + \langle \nabla f_1(x)-v,y-x\rangle + \frac{L_{f_1}}{2}\|y-x\|^2.
 \end{equation*}
 Hence, the assertion follows by Proposition~\ref{p:deslemma}.
\end{example}

\begin{example}[Squared distance to a nonconvex set]\label{prop:Asplund}
Consider a (possibly nonconvex) nonempty closed set $C \subseteq \mathbb{R}^s$ and a    matrix $Q  \in   \mathbb{R}^{s\times n}$. Then the mapping $x\to  \frac{1}{2}d^2(Qx,C)$ is upper $C^2$ on $\mathbb{R}^n$. Indeed, let us notice that
\begin{align}\label{rep_01}
\frac{1}{2}	d^2(Qx,C) =\frac{1}{2} \| Qx\|^2 - \Asp{C}(Qx),
\end{align}
where $\Asp{C}$ is the \emph{Asplund function} associated to the set $C$ given by
\begin{align*}
	\Asp{C}(w):=\sup_{ y \in C} \left\{\langle y, w\rangle - \frac{1}{2} \| y\|^2\right\}=\left(\iota_C+\frac{1}{2}\|\cdot\|^2\right)^*(w).
\end{align*}
By representation \eqref{rep_01} and Example \ref{example01}, we have that $x\mapsto \frac{1}{2}	d^2(Qx,C)$ is $\frac{\|Q\|^2}{2}$-upper-$\cC^2$ on $\mathbb{R}^n$.
\end{example}

\section{The Boosted Double-Proximal Subgradient Algorithm}\label{sect:BDSA}
In this section, we design an algorithm for solving the nonconvex optimization problem~\eqref{eq:P1}. The algorithm utilizes subgradients of $f$, the gradients of~$\Psi_i$, and proximal steps of $g$ and $h_i^*$, for $i=1,\ldots,p$. Additionally, it incorporates a linesearch step, leading us to name it the \emph{Boosted Double-proximal Subgradient Algorithm} (in short BDSA). 

Le us define the function $\varphi: \mathbb{R}^n \to \Rex$ as follows:
\begin{equation}\label{Def_varphi}
	\varphi(x) := f(x) +g(x)-\sum_{i =1 }^p h_i(\Psi_i (x)).
\end{equation}

 The following assumptions are made throughout the paper.

\begin{assumption}\label{assump:1}
Let $U\subseteq \mathbb{R}^n$ be an open convex set such that $\dom g \subseteq U$. Suppose that  $\inf_{x\in\mathbb{R}^n} \varphi(x) > -\infty$ and that the functions in~\eqref{eq:P1} satisfy:
 \begin{enumerate}[(i)]
  \item $f:\R^n\to\Rex$ is locally Lipschitz on $U$ and $\kappa$-upper-$\cC^2$ on every bounded subset of $U$;
  \item $g:\mathbb{R}^n\to\Rex$ is proper, l.s.c. and prox-bounded for some ${\gamma}^g >0$;
  \item $h_i:\mathbb{R}^{m_i}\to\mathbb{R}$ are  convex and continuous functions for all $i=1,\ldots,p$; 
  \item $\Psi_i:\mathbb{R}^n\to\mathbb{R}^{m_i}$ are differentiable functions with $L_i$-Lipschitz continuous gradients on $U$ for all $i=1,\ldots,p$.
 \end{enumerate}
\end{assumption}

\begin{remark} Despite the fact that the class of upper-$\cC^2$ functions is large, as demonstrated by the examples presented in Subsection~\ref{subsec:upperC2}, it is worth noting that, in general, one cannot subsume the function $-\sum_{i =1 }^p h_i(\Psi_i (x))$ into the function $f$ in~\eqref{Def_varphi}, since $-h_i(\Psi_i(x))$ may not belong to the class of upper-$\cC^2$ functions (e.g., if $h_i(y)=y^2$ and $\Psi_i(x)=x^2$, then $-h_i(\Psi_i(x))=-x^4$ does not satisfy~\eqref{equpperdesc} on $U=\R$ for any fixed $\kappa\geq0$).
\end{remark}

Instead of directly addressing problem~\eqref{eq:P1}, which consists in the minimization of the function $\varphi$ in~\eqref{Def_varphi}, we consider the primal-dual formulation
\begin{equation}\label{eq:P2}
    \min_{(x,\mathbf{y}) \in\mathbb{R}^n \times \Rnm}\Phi(x,\mathbf{y}), \tag{PD}
\end{equation}
where  $\Phi:\mathbb{R}^n\times\Rnm\to\overline{\mathbb{R}}$ is given by
\begin{equation}\label{Def_PHI}
	\Phi(x,\mathbf{y}) := f(x)+g(x)+\sum_{i=1}^p \left(h_i^*(y_i)-\langle \Psi_i(x),y_i\rangle\right),
\end{equation}
with $\mathbf{y}=(y_1,\ldots,y_p)\in\mathbb{R}^{m_1}\times\cdots\times\mathbb{R}^{m_p}=\Rnm$. It is easy to check that the optimal values of both problems coincide, that is,
\begin{align}\label{Eq_Optvalues}
	\inf_{x\in\mathbb{R}^n} \varphi(x) = \inf_{(x,\mathbf{y}) \in\mathbb{R}^n \times \Rnm}\Phi(x,\mathbf{y}).
\end{align}
 Indeed, by the Fenchel--Moreau theorem,
\begin{equation*}
\begin{aligned}
    \inf_{x\in\mathbb{R}^n} \varphi(x)
    &= \inf_{x\in\mathbb{R}^n} \left\{ f(x)+g(x)-\sum_{i=1}^p h_i(\Psi_i(x)) \right\}\\
    & = \inf_{x\in\mathbb{R}^n} \left\{f(x)+g(x)-\sum_{i=1}^p \sup_{y_i\in\mathbb{R}^{m_i}} \left\{ \langle \Psi_i (x),y_i\rangle -h_i^*(y_i)\right\} \right\}\\
    & = \inf_{x\in\mathbb{R}^n}\inf_{\mathbf{y}\in\Rnm}\left\{f(x)+g(x)+\sum_{i=1}^p  \left( h_i^*(y_i)-\langle \Psi_i(x),y_i\rangle\right)\right\}\\
    &= \inf_{(x,\mathbf{y}) \in\mathbb{R}^n \times \Rnm}\Phi(x,\mathbf{y}).
\end{aligned}
\end{equation*}

By the generalized Fermat rule, a necessary condition for a point $(\bar{x},\bar{\mathbf{y}})\in\R^n\times\R^m$ to be a local minimum of $\Phi$ is  that $0\in\partial \Phi(\bar{x},\bar{\mathbf{y}})$. Observe that we can express $\Phi=\Phi_1+\Phi_2$, with $\Phi_1(x,\mathbf{y}):=f(x)+g(x)+\sum_{i=1}^p h_i^*(y_i)$ and $\Phi_2(x,\mathbf{y}):=\sum_{i=1}^p \langle \Psi_i(x),y_i\rangle$. Since $\Phi_2$ is~$\cC^1$, by the sum rule in~\cite[Proposition~1.30]{MR3823783},
\begin{equation}\label{eq:subdiff_Phi}
\begin{aligned}
\partial \Phi(x,\mathbf{y})&=\partial\Phi_1(x,\mathbf{y})-\nabla \Phi_2(x,\mathbf{y})\\
&=\partial(f+g)(x)\times\left(\bigtimes_{i=1}^p \partial h_i^*(y_i)\right)-\left(\sum_{i=1}^p\nabla \Psi_i(x)y_i,\Psi_1(x),\ldots,\Psi_p(x)\right),
\end{aligned}
\end{equation}
where the second equality holds because $\Phi_1$ has separate variables. Therefore, the necessary condition $0\in\partial \Phi(\bar{x},\bar{\mathbf{y}})$ is equivalent to
\begin{equation}\label{eq:estpoint}
\left\{
    \begin{aligned}
        &\sum_{i=1}^p \nabla\Psi_i(\bar{x}) \bar{y}_i    \in  \partial  (f+ g)(\bar{x}) ,\\
        &\Psi_i( \bar{x}) \in \partial h_i^* (\bar{y}_i), \quad \forall i=1,\ldots,p.
    \end{aligned}
\right.
\end{equation}
By~\cite[Corollary~2.20]{MR3823783}, we have $\partial  (f+ g)(\bar{x}) \subseteq \partial f(\bar x)+\partial g(\bar x)$. Thus, the inclusions~\eqref{eq:estpoint} imply
\begin{equation}\label{eq:critpoint}
\left\{
    \begin{aligned}
        &\sum_{i=1}^p \nabla\Psi_i(\bar{x}) \bar{y}_i    \in  \partial  f(\bar x )+\partial g(\bar{x}) ,\\
        &\bar{y}_i \in \partial h_i (\Psi_i( \bar{x})), \quad \forall i=1,\ldots,p.
    \end{aligned}
\right.
\end{equation}
We refer to a point $(\bar{x},\bar{\mathbf{y}})\in\R^n\times\R^m$ that satisfy~\eqref{eq:critpoint} as a \emph{critical point of~\eqref{eq:P2}}.

On the other hand, given a point $\bar{x}\in\dom{\varphi}$, it is well-known (see, e.g., \cite{MR4228320,correa2023optimality,MR2453095}) that the subdifferential inclusion
\begin{equation}\label{eq:statpoint}
    \partial\left( \sum_{i=1}^p h_i \circ \Psi_i\right)(\bar{x}) \subset \partial (f+g)(\bar{x}).
\end{equation}
constitutes a necessary condition for local optimality of problem~\eqref{eq:P1}. A point verifying~\eqref{eq:statpoint} is called a \emph{stationary point} of~\eqref{eq:P1}. In general, finding stationary points is highly challenging, so it is useful to consider relaxed notions. We say that a point $\bar{x}\in\dom{\varphi}$ is a \emph{critical point} of~\eqref{eq:P1} if there exists $\barbf{y}\in\Rnm$ such that~\eqref{eq:critpoint} holds. By~\cite[Theorem~10.6]{MR1491362} and~\cite[Corollary~2.21]{MR3823783}, it easily follows that every stationary point of~\eqref{eq:P1} is a critical point, but the converse is not true in general.

Therefore, if  $(\bar{x},\bar{\mathbf{y}})\in\mathbb{R}^n\times\Rnm$ is a critical point of~\eqref{eq:P2}, then $\bar{x}$ is a critical point of~\eqref{eq:P1}. Conversely, if $\bar{x}\in\R^n$ is a critical point of~\eqref{eq:P1}, then there exists $ \bar{\mathbf{y}}\in\mathbb{R}^m$ such that $(\bar{x},\bar{\mathbf{y}})$ is a critical point of~\eqref{eq:P2}. The next result establishes further connections between critical points and solutions of the two minimization problems \eqref{eq:P1} and \eqref{eq:P2} presented above.
\begin{proposition}\label{p:solphiL}
Let $(\bar{x},\bar{\mathbf{y}})\in\mathbb{R}^n\times\Rnm$. Then, the following claims hold.
\begin{enumerate}[(i)]
 \item\label{it:psolphiL-3} If $(\bar{x},\bar{\mathbf{y}})$ is a critical point of~\eqref{eq:P2}, then $\Phi(\bar{x},\bar{\mathbf{y}})=\varphi(\bar{x})$.
 \item\label{it:psolphiL-4} If $(\bar{x},\bar{\mathbf{y}})$ is a solution of~\eqref{eq:P2}, then  $\bar{x}$ is a solution of~\eqref{eq:P1}.
  \item\label{it:psolphiL-5} If $\bar{x}$ is a solution of~\eqref{eq:P1},  then there exists $ \bar{\mathbf{y}}\in\mathbb{R}^m$ such that $(\bar{x},\bar{\mathbf{y}})$ is a solution of~\eqref{eq:P2}.
 \end{enumerate}
\end{proposition}
\begin{proof}
\eqref{it:psolphiL-3} Let  $(\bar{x},\bar{\mathbf{y}})=(\bar{x},\bar{y}_1,\ldots,\bar{y}_p)$ be a critical point of~\eqref{eq:P2}. In particular, $y_i\in\partial h_i(\Psi_i(\bar{x}))$  for all $i=1,\ldots,p$, so for these points the Fenchel--Young inequality becomes an equality (see, e.g.~\cite[Proposition~16.10]{bauschke2017}), and we obtain the following expressions
\begin{equation}\label{F_Y_Eq}
    \begin{aligned}
    h_i(\Psi_i(\bar{x})) +h_i^*(\bar{y}_i) = \langle \Psi_i(\bar{x}),\bar{y}_i\rangle, \quad \forall i=1,\ldots,p.
    \end{aligned}
\end{equation}
This yields the equality
\begin{equation*}\label{eq:phiL}
\begin{aligned}
\varphi(\bar{x}) & = f(\bar{x}) + g(\bar{x}) -\sum_{i=1}^p h_i(\Psi_i (\bar{x})) \\
& = f(\bar{x}) + g(\bar{x}) + \sum_{i=1}^p  \left( h_i^*(\bar{y}_i) - \langle \Psi_i(\bar{x}),\bar{y}_i\rangle\right) = \Phi(\bar{x},\bar{y}_1,\ldots,\bar{y}_p).
\end{aligned}
\end{equation*}

\eqref{it:psolphiL-4} If $(\bar{x},\bar{\textbf{y}})$ is a solution of~\eqref{eq:P2}, as argued above, it must be a critical point.
Then, by~\eqref{it:psolphiL-3}, we have that $
\Phi(\bar{x},\bar{\mathbf{y}})=\varphi(\bar{x})
$. Since the optimal values of problems~\eqref{eq:P1} and~\eqref{eq:P2} coincide (recall \eqref{Eq_Optvalues}), the above expression implies that $\bar{x}$ is a solution to~\eqref{eq:P1}.

\eqref{it:psolphiL-5} Finally, let us suppose that  $\bar{x}$ is a solution of~\eqref{eq:P1}. Then, $\bar{x}$ is necessarily a critical point of~\eqref{eq:P1}, and thus there exists $\bar{y}_i \in \partial h_i(\Psi_i(\bar x))$ such that \eqref{F_Y_Eq} holds. Once more, this implies that $\varphi(\bar{x})= \Phi(\bar x, \bar{\mathbf{y}}) $, where $\bar{\mathbf{y}}:=(\bar{y}_1,\ldots,\bar{y}_p )$.  Therefore, using again the fact that  optimal values of problems~\eqref{eq:P1} and~\eqref{eq:P2} coincide, we get that $(\bar x, \bar{\mathbf{y}}) $ is a solution of problem \eqref{eq:P2}.
\end{proof}
Now we present in Algorithm \ref{alg:1} the pseudo-code of the \emph{Boosted Double-proximal Subgradient Algorithm}.
\begin{algorithm}[hb!]
\caption{Boosted Double-proximal Subgradient Algorithm for problem~\eqref{eq:P1}}\label{alg:1}
\begin{algorithmic}[1]
\Require{$(x^0,\mathbf{y}^0) = (x^0,y^0_1,\ldots,y^0_p)\in\mathbb{R}^n\times\Rnm$, $R\geq 0$, $\rho\in{]0,1[}$ and $\alpha\geq 0$. Set $k:=0$.}

\State{Choose $v^{k}\in\partial f(x^k)$.}
\State{Take some positive $\gamma_k<\min\left\{{\gamma}^g,\left(2\kappa + \sum_{i=1 }^p L_i\left\| y_i^k\right\|\right)^{-1}\right\}$ and compute
    \begin{equation}\label{eq:algx}
    \begin{aligned}
        \hat{x}^{k} & \in \prox_{\gamma_k g} \left(x^k+\gamma_k \sum_{i=1}^p \nabla \Psi_i(x^k) y_i^k-\gamma_k v^k\right) .
    \end{aligned}
    \end{equation}
}
\State{For each $i=1,\ldots,p$, take $\mu_i^k>0$ and compute
\begin{equation}\label{eq:algy}
\begin{aligned}
        \hat{y}_i^{k} & = \prox_{\mu^k_i h_i^*}\left(y_i^k+\mu^k_i  \Psi_i(\hat{x}^{k})  \right).
\end{aligned}
\end{equation}
}
\State{Choose any $\overline{\lambda}_k\geq 0$. Set $\lambda_{k}:=\overline{\lambda}_{k}$, $r:=0$ and $(d^{k}, \mathbf{e}^{k}):=(\hat{x}^{k}, \hat{\mathbf{y}}^{k}) -(x^{k},\mathbf{y}^k)$.}
\State{\textbf{if} $(d^{k},\mathbf{e}^{k})=0$ \textbf{then} STOP and return $x^k$.}
\State{\textbf{while} $r<R$ and
		\begin{align}\label{eq:whilecond}
			\Phi\big((\hat{x}^{k},\hat{\mathbf{y}}^{k}) + \lambda_{k}(d^{k}, \mathbf{e}^{k}) \big)  >  		\Phi(\hat{x}^{k}, \hat{\mathbf{y}}^{k})- \alpha \lambda_k^2 \| (d^{k},\mathbf{e}^k)\|^2
		\end{align}%
\textbf{do }{$r:=r+1$ and $\lambda_k:=\rho^r\overline{\lambda}_k$.}}

\State{\textbf{if }$r=R$ \textbf{then} $\lambda_k:=0$.}

	\State{Set $(x^{k+1},\mathbf{y}^{k+1}) :=(\hat{x}^{k},\hat{\mathbf{y}}^k)+ \lambda_{k}(d^{k}, \mathbf{e}^{k})$, $k:=k+1$ and go to Step 1.}
\end{algorithmic}
\end{algorithm}

\begin{remark}\label{r:alg1}
From the definition of  the proximity operator,~\eqref{eq:algx} and~\eqref{eq:algy} imply that the sequences generated by Algorithm~\ref{alg:1} verify
\begin{equation}\label{eq:subdif1}
\begin{aligned}
    &\frac{x^k-\hat{x}^{k}}{\gamma_k} + \sum_{i=1}^p \nabla \Psi_i(x^k) y_i^k -  v^k  \in \partial g(\hat{x}^{k}),
    \\
    &\frac{y_i^k-\hat{y}_i^{k}}{\mu_i^k} + \Psi_i (  \hat{x}^{k})  \in \partial h_i^*(\hat{y}_i^{k}), \quad \forall i= 1,\ldots, p.
\end{aligned}
\end{equation}
In particular, if $(d^k,\mathbf{e}^k)=0$, the above inclusions become~\eqref{eq:critpoint} for $(\bar{x},\barbf{y}):=(x^k,\mathbf{y}^k)$ and hence $\bar{x}$ is a critical point of problem~\eqref{eq:P1}.

Let us also observe that in Algorithm~\ref{alg:1} it is possible to replace $v^k \in \partial f(x^k)$ by $v^k \in \conv{\partial f(x^k)}$. This is justified by Proposition~\ref{p:deslemma}, which shows that such subgradients satisfy the descent inequality~\eqref{equpperdesc}. This straightforward modification in Step~1 of Algorithm~\ref{alg:1} can be advantageous in numerical applications, particularly when the calculus rules for the basic subdifferential only provide an upper estimate rather than an equality. For example, if the objective function $\varphi(x)$ contains a term of the form $-\upsilon(x)$ and $\upsilon:\R^n\to\Rex$ is convex, it is natural to set $f:=-\upsilon$, which is upper-$\cC^2$ by Example~\ref{example01}. Algorithm~\ref{alg:1} requires choosing $v^k \in \partial f(x^k)=\partial\left(-\upsilon(x^k)\right)\subset -\partial\upsilon(x^k)$, while the modification $v^k \in \conv{\partial f(x^k)}=-\partial \upsilon(x^k)$ (see~\eqref{eq:subinclusions}) allows choosing $v^k$ from the possibly larger set $-\partial \upsilon(x^k)$.

On the other hand, according to~\eqref{eq:subdif1}, allowing $v^k \in \conv{\partial f(x^k)}$ in Step~1 entails the weaker notion of criticality
\begin{equation*}
\left\{
    \begin{aligned}
        &\sum_{i=1}^p \nabla\Psi_i(\bar{x}) \bar{y}_i    \in  \conv{\partial  f(\bar x )}+\partial g(\bar{x}) ,\\
        &\bar{y}_i \in \partial h_i (\Psi_i( \bar{x})), \quad \forall i=1,\ldots,p,
    \end{aligned}
\right.
\end{equation*}
when $(d^k,\mathbf{e}^k)=0$ and $(\bar{x},\barbf{y}):=(x^k,\mathbf{y}^k)$. Thus, the price to pay for having more freedom in the choice of $v^k$ is the possibility of having a larger set of non-optimal critical points to which the algorithm might converge. Nonetheless, a more subtle analysis can be performed to derive stronger notions of criticality, see Remark~\ref{Remark_sub} and Section~\ref{sect:numerical2} for a numerical example.

\end{remark}

\begin{remark}[Particular cases of Algorithm~\ref{alg:1}]\label{remark:particularcases}
Different known algorithms can be obtained as particular cases of Algorithm~\ref{alg:1}.
\begin{enumerate}[(i)]
		\item\label{it:particularcases1} Consider the problem of minimizing $\varphi(x):=g_1(x)-g_2(x)$, with $g_1$ and $g_2$ being convex. If we let $f:=-g_2$, $g:=g_1$ and $\Psi_i:=0=:h_i$ for all $i=1,\ldots,p$, then Step~1 of Algorithm~\ref{alg:1} becomes $v^k\in\partial (-g_2)(x^k)$. Since $\partial (-g_2)(x^k)\subseteq -\partial g_2(x^k)$, when $R=0$ one recovers a specific choice for the \emph{Proximal DC Algorithm} of~\cite{sun2003proximal} in which $v^k$ is chosen from a smaller set of subgradients. If $x^{k+1}=x^k=:\bar{x}$, one gets $-v^k\in \partial g_1(\bar{x})\cap\left(-\partial (-g_2)(\bar{x})\right)$. Observe that the more restrictive condition $\partial g_1(\bar{x})\cap\left(-\partial (-g_2)(x)\right)\neq\emptyset$ can serve to discard some points which are not local minima (e.g., if $g_1(x)=x^2$ and $g_2(x)=|x|$, then $\bar{x}:=0$ is a local maximum which satisfies $0\in\nabla g_1(\bar{x})\cap\partial g_2(\bar{x})$, but $0\not\in\nabla g_1(\bar{x})\cap\left(-\partial (-g_2)(\bar{x})\right)$). When $R=\infty$, one obtains the \emph{Boosted Proximal DC Algorithm} introduced in~\cite{alizadeh2022new}, which is a modification of the \emph{Boosted Difference of Convex functions Algorithm} from~\cite{aragon2018accelerating, boostedDCA} that adds a proximal term.
		\item Likewise, if $\varphi(x):=g_1(x)-g_2(x)+g_3(x)$ with $g_1$ being proper and l.s.c. with $\inf_{x\in\R^n}g_1>-\infty$, $g_2$ being convex and $g_3$ being $L$-smooth, letting $f:=-g_2+g_3$ and the rest of the functions as in the previous case with $R=0$, we obtain a specific choice for the \emph{Generalized Proximal Point Algorithm} presented in~\cite{an2017convergence}.
		\item The \emph{Double-Proximal Gradient Algorithm} proposed in~\cite{bot2019doubleprox} is also recovered by Algorithm~\ref{alg:1} in the case in which $f$ is convex and $L$-smooth, $g$ is convex, $R=0$, $p=1$ and $\Psi_1$ is a linear operator.
	\end{enumerate}
\end{remark}

Steps 6-7 of Algorithm~\ref{alg:1} correspond to an optional linesearch step in the direction $(d^{k}, \mathbf{e}^{k})$ with a fixed number $R$ of attempts. On the one hand, note that the computational burden of these steps can be avoided if the user either chooses $R=0$ or $\overline{\lambda}_k=0$, in which case we refer to the resulting algorithm as \emph{Double-proximal Subgradient Algorithm}  (abbr. \emph{DSA}). On the other hand, if $R>0$ and $\overline{\lambda}_k>0$, Step~6 allows to achieve a further decrease of the primal-dual objective function~$\Phi$. Step~7 sets $\lambda_k=0$ when the linesearch was not successful.

For general problems satisfying Assumption~\ref{assump:1} there is no guarantee that the vector $(d^{k},\mathbf{e}^{k})$  defined in Step $4$ of BDSA provides a descent direction for the function $\Phi$. The motivation for the linesearch step comes from the case in which the functions $g$ and $h_i^*$ are differentiable at the points $\hat{x}^k$ and $\hat{y}_i^k$, respectively. In this case, $(d^{k},\mathbf{e}^{k})$ is a descent direction for $\Phi$ at $(\hat{x}^{k},\hat{\mathbf{y}}^{k})$, since the upper Dini directional derivative of $\Phi$ at $(\hat{x}^{k},\hat{\mathbf{y}}^{k})$ in the direction $(d^{k},\mathbf{e}^{k})$ is negative. This fact is proved in the next result.

\begin{proposition}\label{l:descentdirection}
	Suppose that Assumption~\ref{assump:1} holds and consider the sequences generated by Algorithm~\ref{alg:1} for problem~\eqref{eq:P1}. Assume also that
	\begin{enumerate}[(i)]
		\item $g$ is differentiable at $\hat{x}^{k}$,
		\item $h_i^*$ is differentiable at $\hat{y_i}^{k}$ for all $i = 1,\ldots,p$,
		\item\label{it:descentdirection-4} $\gamma_k \in\left] 0, \left(2\kappa +\frac{p}{2}+\sum_{i=1}^p L_i\|y_i^k\|\right)^{-1}\right[$ and $\mu_i^{k} \in{\left]0,2\,\|\nabla\Psi_i(\hat{x}^k)\|^{-2}\right[}$ for all $i=1,\ldots,p$.
	\end{enumerate}
	Then, for all $k\geq0$,
	\begin{equation}\label{eq:direc_deriv}
		\begin{aligned}
			d^+\Phi\big(&(\hat{x}^{k},\hat{\mathbf{y}}^{k}); (d^{k},\mathbf{e}^{k})\big) \leq\\
			&\left(2\kappa +\frac{p}{2}+\sum_{i=1}^p L_i\|y_i^k\| -\frac{1}{\gamma_k}\right) \|d^{k}\|^2
			+ \sum_{i=1}^p\left(\frac{1}{2}\|\nabla\Psi_i(\hat{x}^k)\|^2-\frac{1}{\mu_i^k}\right) \|e_i^{k}\|^2.
		\end{aligned}
	\end{equation}
	Consequently, if $(d^{k},\mathbf{e}^{k})\neq 0$, then for every $\alpha>0$ there is some $\delta_k>0$ such that
	\begin{equation}\label{eq:linesearch_guarantee}
		\Phi\big((\hat{x}^{k},\hat{\mathbf{y}}^{k}) + \lambda(d^{k}, \mathbf{e}^{k}) \big)  \leq\Phi(\hat{x}^{k}, \hat{\mathbf{y}}^{k})- \alpha \lambda^2 \|(d^{k},\mathbf{e}^k)\|^2, \quad\text{for all }\lambda\in{[0,\delta_k]}.
	\end{equation}
\end{proposition}
\begin{proof}
	Indeed, for any $\hat{v}^{k}\in\partial f(\hat{x}^{k})$, we get
	{\small
		\begin{equation}\label{eq:direc_deriv0}
			\begin{aligned}
				d^+\Phi\big((\hat{x}^{k},\hat{\mathbf{y}}^{k}); (d^{k},\mathbf{e}^{k})\big) & \leq \limsup_{t\downarrow 0} \frac{f(\hat{x}^{k}+ td^{k}) -f(\hat{x}^{k}) }{t} +  \limsup_{t\downarrow 0} \frac{g(\hat{x}^{k}+ td^{k}) -g(\hat{x}^{k}) }{t} \\
				& \quad + \sum_{i=1}^p \limsup_{t\downarrow 0} \frac{h_i^*(\hat{y}_i^{k}+te_{i}^{k}) - h_i^*(\hat{y}_i^{k})}{t} \\
				& \quad -   \sum_{i=1}^p \liminf_{t\downarrow 0}\frac{ \langle\Psi_i(\hat{x}^{k}+td^{k})-\Psi_i(\hat{x}^{k}),\hat{y}_i^{k}\rangle +t\langle\Psi_i(\hat{x}^{k}+td^{k}),e_i^{k} \rangle}{t} \\
				& \leq \langle \hat{v}^{k}, d^{k}\rangle +  \langle\nabla g(\hat{x}^{k}),d^{k}\rangle + \sum_{i=1}^p \langle\nabla h_i^*(\hat{y}_i^{k}),e_{i}^{k}\rangle \\
				& \quad -   \sum_{i=1}^p \left(\langle\nabla\Psi_i(\hat{x}^k) \hat{y}_i^{k},d^{k}\rangle + \langle  \Psi_i(\hat{x}^{k}), e_i^{k}\rangle \right), \\
			\end{aligned}
	\end{equation} }
	where the second inequality is due to Proposition~\ref{p:deslemma}.
	Now, since $g$ and $h_i^*$ are  assumed to be differentiable at $\hat{x}^{k}$ and  $\hat{y}_i^{k}$, respectively,~\eqref{eq:subdif1} yields
	\begin{equation}\label{eq:diff}
		\begin{aligned}
			&\frac{x^k-\hat{x}^{k}}{\gamma_k} + \sum_{i=1}^p \nabla \Psi_i(x^k) y_i^k -  v^k  = \nabla g(\hat{x}^{k}),
			\\
			&\frac{y_i^k-\hat{y}_i^{k}}{\mu_i^k} + \Psi_i (  \hat{x}^{k})  =\nabla h_i^*(\hat{y}_i^{k}), \quad \forall i= 1,\ldots, p.
		\end{aligned}
	\end{equation}
	On the other hand, again making use of Proposition~\ref{p:deslemma}, we get the following inequality by setting $y:=\hat{x}^{k}$, $x:= x^{k}$ and $\xi:=v^k$ in equation~\eqref{equpperdesc}
	\begin{equation*}
		f(\hat{x}^{k})-f(x^k) - \langle  v^k, \hat{x}^{k} - x^k \rangle \leq \kappa \|\hat{x}^{k}-x^{k}\|^2.
	\end{equation*}
	Likewise, setting $y:=x^k$, $x:=\hat{x}^{k}$ and $\xi:=\hat{v}^{k}$  in~\eqref{equpperdesc} yields
	\begin{equation*}
		f(x^k) - f(\hat{x}^{k}) + \langle \hat{v}^{k},\hat{x}^{k}-x^k\rangle \leq \kappa \|\hat{x}^{k}-x^{k}\|^2.
	\end{equation*}
	Summing together these two equations, we get
	\begin{equation}\label{eq:subdiffbound}
		\langle \hat{v}^{k}-v^{k}, \hat{x}^{k}-x^k \rangle \leq 2 \kappa \|\hat{x}^{k}-x^k\|^2 \quad \text{ for all } v^k\in\partial f(x^k).
	\end{equation}
	Substituting~\eqref{eq:diff} and~\eqref{eq:subdiffbound} in~\eqref{eq:direc_deriv0}, it becomes
	\begin{equation}\label{eq:direc_deriv1}
		\begin{aligned}
			d^+\Phi\big((\hat{x}^{k},\hat{\mathbf{y}}^{k}); (d^{k},\mathbf{e}^{k})\big)  & \leq \langle  \hat{v}^{k}-v^k,d^{k}\rangle  -\frac{1}{\gamma_k} \langle \hat{x}^{k}-x^k,d^{k}\rangle  + \sum_{i=1}^p\langle \nabla \Psi_i(x^k)y_i^k,d^{k}\rangle  \\
			&\quad -\sum_{i=1}^p \frac{1}{\mu_i^k} \langle \hat{y}_i^{k}-y_i^k,e_i^{k}\rangle +\sum_{i=1}^p \langle \Psi_i(\hat{x}^{k}),e_i^{k}\rangle \\
			&\quad - \sum_{i=1}^p \left(\langle\nabla\Psi_i(\hat{x}^k) \hat{y}_i^{k},d^{k}\rangle + \langle  \Psi_i(\hat{x}^{k}), e_i^{k}\rangle \right) \\
			& \leq \left(2\kappa-\frac{1}{\gamma_k}\right)\|d^{k}\|^2 -\sum_{i=1}^p \frac{1}{\mu_i^k} \|e_i^{k}\|^2  \\
			&\quad- \sum_{i=1}^p \langle \nabla\Psi_i(\hat{x}^k) \hat{y}_i^{k}-\nabla \Psi_i(x^k)y_i^k ,d^{k} \rangle.
		\end{aligned}
	\end{equation}
	Finally, using the Cauchy--Schwartz inequality and Young's inequality, the terms in the last summation can be upper bounded as follows:
\begin{equation*}
 \begin{aligned}
		-\langle \nabla\Psi_i(\hat{x}^k) \hat{y}_i^{k}-\nabla \Psi_i(x^k)y_i^k ,d^{k} \rangle&=\langle \nabla\Psi_i(\hat{x}^k) e_i^k ,d^{k} \rangle+\langle (\nabla\Psi_i(\hat{x}^k)- \nabla \Psi_i(x^k))y_i^k ,d^{k} \rangle\\
		&\leq \|\nabla\Psi_i(\hat{x}^k)\| \|e_i^k\|\|d^{k}\| + L_i\|y_i^k\| \|d^k\|^2\\
		&\leq \frac{1}{2}\left(\|\nabla\Psi_i(\hat{x}^k)\|^2\|e_i^k\|^2+\|d^k\|^2\right)+ L_i\|y_i^k\| \|d^k\|^2\\
		&=\frac{1}{2}\|\nabla\Psi_i(\hat{x}^k)\|^2\|e_i^k\|^2+\left(\frac{1}{2}+L_i\|y_i^k\|\right)\|d^k\|^2,
\end{aligned}
\end{equation*}
	for all $i = 1,\ldots,p$. Putting this into~\eqref{eq:direc_deriv1}, we deduce~\eqref{eq:direc_deriv}.
	
	Thanks to assumption~(\ref{it:descentdirection-4}), we have
	$$K:=\min_{i=1,\ldots,p}\left\{\frac{1}{\gamma_k}-2\kappa +\frac{p}{2}-\sum_{i=1}^p L_i\|y_i^k\|, \frac{1}{\mu_i^k}-\frac{1}{2}\|\nabla\Psi_i(\hat{x}^k)\|^2\right\}>0.
	$$
	Thus, if $(d^{k},\mathbf{e}^k)\neq 0$, one has
	$$d^+\Phi\big((\hat{x}^{k},\hat{\mathbf{y}}^{k}); (d^{k},\mathbf{e}^{k})\big) \leq -K\| (d^{k},\mathbf{e}^k)\|^2<-\frac{K}{2}\| (d^{k},\mathbf{e}^k)\|^2,$$
	so there exist $\tau_k>0$ such that
	$$\Phi\big((\hat{x}^{k},\hat{\mathbf{y}}^{k}) + \lambda(d^{k}, \mathbf{e}^{k}) \big)  \leq \Phi(\hat{x}^{k}, \hat{\mathbf{y}}^{k})- \lambda\frac{K}{2} \|(d^{k},\mathbf{e}^k)\|^2, \quad\text{for all }\lambda\in{[0,\tau_k]}.$$
	Then, given any $\alpha>0$, letting $\delta_k:=\min\{K/(2\alpha),\tau_k\}>0$, we have that $-K/2\leq-\lambda\alpha$ for all $\lambda\in[0,\delta_k]$, so we obtain~\eqref{eq:linesearch_guarantee}.
\end{proof}

The differentiability of $h_i^*$ is guaranteed when $h_i$ is strictly convex. Actually, for proper l.s.c. convex functions, \emph{essential strict convexity} is equivalent to \emph{essential smoothness} of the conjugate function, cf.~\cite[Theorem~26.3]{Rockafellar1970}.

Before moving to the convergence analysis in the next subsection, let us explain the rationale behind Algorithm~\ref{alg:1} in the simplest case in which $h_i=\Psi_i=0$. Thanks to the $\kappa$-upper-$\cC^2$ assumption on $f$, since $v^k\in\partial f(x^k)$ (Step~1), it holds
$$f(z)\leq f(x^k)+\langle v^k,z-x^k\rangle+\kappa\|z-x^k\|^2,\quad\forall z\in\R^n.$$
Thus, if $\gamma_k\in{\left]0,\frac{1}{2\kappa}\right[}$ (as in Step~2), one gets
\begin{align*}
\varphi(z)&=f(z)+g(z)\\
&\leq g(z)+f(x^k)+\langle v^k,z-x^k\rangle+\frac{1}{2\gamma_k}\|z-x^k\|^2=:\widetilde{\varphi}_k(z).
\end{align*}%
for all $z\in\R^n$. Therefore, the function $\widetilde{\varphi}_k$ provides an upper bound to $\varphi$, so it makes sense to take
$$\hat{x}^k=\argmin_{z\in\R^n}\widetilde{\varphi}_k(z)= \prox_{\gamma_k g} \left(x^k-\gamma_k v^k\right), $$
which coincides with~\eqref{eq:algx}. Finally, when $g$ is differentiable at $\hat{x}^k$, the linesearch condition~\eqref{eq:whilecond} in Step~6 permits to further reduce the original function $\varphi$.

For illustration, consider the function from~\cite[Example~2.4]{boostedDCA} given by $\varphi(x):=x_1+x_2 -\|x\|_1+\|x\|^2$, for $x\in\R^2$. If we let $f(x):=x_1+x_2 -\|x\|_1$ and $g(x):=\|x\|^2$, then $f$ is $\kappa$-upper-$\cC^2$ for  $\kappa=0$, by Example~\ref{example01}. If we take $x^0:=(0,1)^T$ and $\gamma_0:=1$, we have $v^0:=(2,0)^T\in \partial f(x^0)=\left\{(2,0)^T,(0,0)^T\right\}$ and $\hat{x}^0=\prox_{\gamma_0 g}(x^0-\gamma_0v^0)=\frac{1}{3}(-2,1)^T$, which minimizes the function $\widetilde{\varphi}_0$. In Figure~\ref{fig:example} we represent the sections of $\varphi$ and $\widetilde{\varphi}$ at $\hat{x}^0$ in the direction $d^0=\hat{x}^0-x^0=-\frac{2}{3}(1,1)^T$. Taking for instance $\alpha=0.1$, we can observe how the linesearch step in the direction $d^0$ can help achieving an additional reduction of the objective function. 

\begin{figure}[H]\centering
	\includegraphics[height=.5\textwidth]{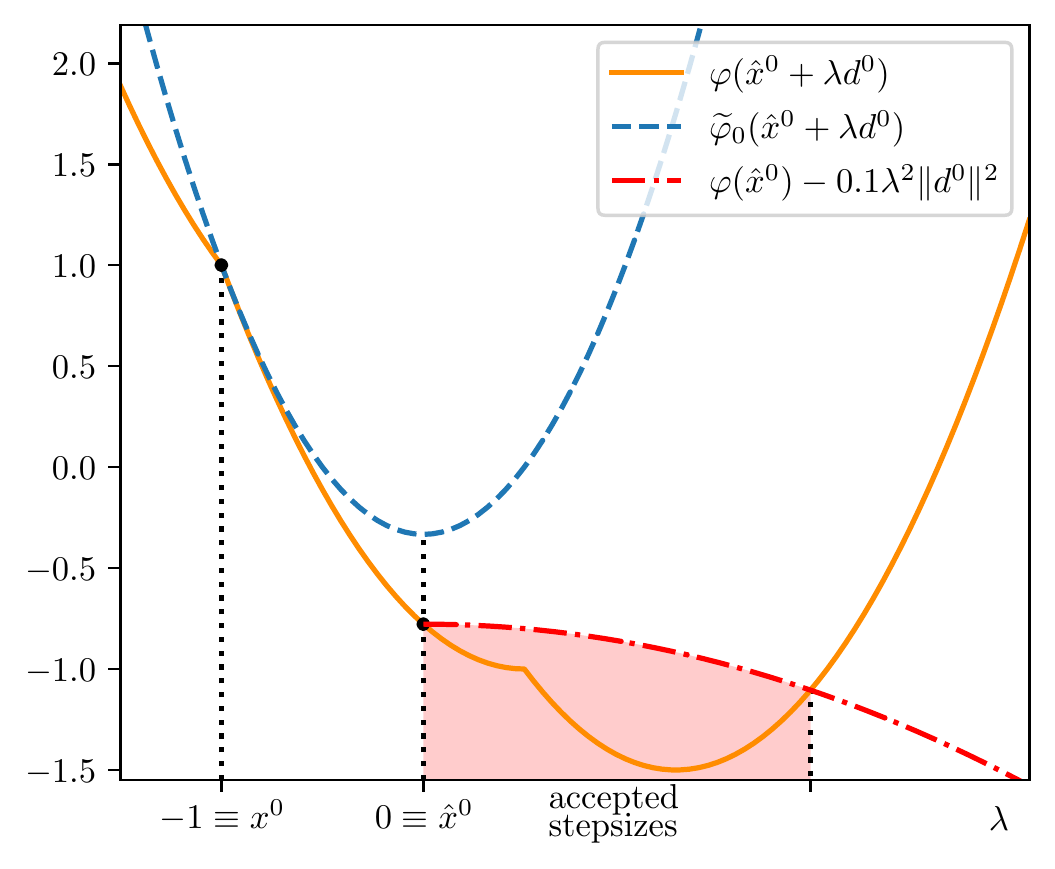}
	\caption{Sections of the functions $\varphi$ and $\widetilde{\varphi}_0$ at $\hat{x}^0$ in the direction $d^0$. The point $x^0$ corresponds with the stepsize $\lambda=-1$.}\label{fig:example}
\end{figure}

\subsection{Convergence Analysis}\label{subsec:convergence}

The following result shows that the primal-dual functional $\Phi$ of problem~\eqref{eq:P2} evaluated at the sequence $( x^k,\mathbf{y}^k)_{k\in\mathbb{N}}$ generated by BDSA decreases after every iteration of the algorithm.
\begin{proposition}\label{p:des}
Let $\Phi$ be the function defined in~\eqref{Def_PHI} and suppose that Assumption~\ref{assump:1} holds. Given a starting point $(x^0,\mathbf{y}^0)=(x^{0},y^{0}_1,\ldots,y^{0}_p)\in\mathbb{R}^n\times\Rnm$, consider the sequence $(x^k,\mathbf{y}^k)_{k\in\mathbb{N}}=(x^k,y^k_1,\ldots,y^k_p)_{k\in\mathbb{N}}$ generated by  Algorithm~\ref{alg:1}. Then, for all $k\geq1$,
\begin{equation}\label{eq:des}
\begin{aligned}
    \Phi(x^{k+1},&\mathbf{y}^{k+1}) - \Phi(x^{k},\mathbf{y}^k
    )
    & \leq -a_k \|x^{k+1}-x^k\|^2    -   \sum_{i=1}^p b_i^k \|y_i^{k+1}-y_i^k\|^2,
\end{aligned}
\end{equation}
where
\begin{equation*}\label{eq:akbk}
a_k:=\frac{2\alpha \lambda_k^2+\gamma_k^{-1}- 2\kappa -\sum_{i=1}^p L_i\| y_i^k\|}{2(1+\lambda_k)^2}>0\quad\text{and}\quad b^k_i:=\frac{1+\alpha \lambda_k^2\mu_i^k}{\mu_i^k(1+\lambda_k)^2}>0,
\end{equation*}
for  $i=1,\ldots,p$.
\end{proposition}
\begin{proof}
	First, note that for $k\geq1$ the vector  $(x^k,\mathbf{y}^k)$  belongs to $\dom \Phi$.	By the definition of the proximal point operator, equation~\eqref{eq:algx} yields the inequality
		{\small
	\begin{equation*}
		g(\hat{x}^{k}) + \frac{1}{2\gamma_k} \left\|\hat{x}^{k}-x^k-\gamma_k \left(\sum_{i=1}^p \nabla\Psi_i(x^k) y_i^k-v^k\right)\right\|^2  \leq g(x^k) + \frac{\gamma_k}{2}\left\|\sum_{i=1}^p \nabla \Psi_i(x^k) y_i^k -v^k\right\|^2.
	\end{equation*}}
	Rearranging this expression and remembering that $d^k=\hat{x}^k-x^k$ we get
	\begin{equation}\label{eq:subg_1}
	 g(\hat{x}^{k})-g(x^k) \leq \left\langle d^k,\sum_{i=1}^p\nabla\Psi_i(x^k) y_i^k-v^k\right\rangle-\frac{1}{2\gamma_k}\|d^k\|^2.
	\end{equation}
Now, let us notice that since the function $x\mapsto  -\langle \Psi_i(x),y_i^k\rangle$ is $\cC^{1}$ with $L_i\| y_i^k\|$-Lipschitz gradient, then  by \eqref{eq:descentlemmaineq} we have
\begin{align*}
	\langle   \Psi_i(x^k)-\Psi_i(\hat{x}^{k}),y_i^{k}\rangle \leq- \langle \nabla\Psi_i(x^k) y_i^k,  d^k   \rangle  +  \frac{L_i\| y_i^k\|}{2} \| d^{k}\|^2.
\end{align*}
	Using this expression and~\eqref{eq:subg_1}, we obtain
	\begin{equation}\label{eq:Phixk}
		\begin{aligned}
			\Phi(\hat{x}^{k},\mathbf{y}^k) - \Phi(x^k,\mathbf{y}^k) =  &   f(\hat{x}^{k})- f(x^{k})  +
			g(\hat{x}^{k})- g(x^k)\\
                &+\sum_{i=1 }^p \langle  \Psi_i(x^k)-\Psi_i(\hat{x}^{k}), y_i^k \rangle   \\
				 \leq  &  f(\hat{x}^{k})- f(x^{k})  +
			g(\hat{x}^{k})- g(x^k)\\& +\sum_{i=1 }^p \left(- \langle \nabla\Psi_i(x^k) y_i^k,  d^k   \rangle  +  \frac{L_i\| y_i^k\|}{2} \| d^{k}\|^2 \right)   \\
			\leq&  f(\hat{x}^{k})- f(x^{k})  -\frac{1}{2\gamma_k}\|d^{k}\|^2  - \langle  v^k, d^{k} \rangle +  \frac{1}{2}\sum_{i=1 }^p L_i\| y_i^k\| \| d^{k}\|^2 \\
		 \leq	& \left( \kappa + \frac{1}{2}\sum_{i=1 }^p L_i\| y_i^k\|    - \frac{1}{2\gamma_k}\right) \|d^k\|^2,
		\end{aligned}
	\end{equation}
	where the last  inequality is due to Proposition~\ref{p:deslemma}. On the other hand, in equation~\eqref{eq:algy} we are computing the proximity operator of the convex  function $h_i^*$, which yields the subgradient inequality
	\begin{equation}\label{eq:subh_2}
		h_i^*(\hat{y}^{k}) + \left\langle \frac{y_i^k-\hat{y}_i^{k}}{\mu_i^k} + \Psi_i(\hat{x}^{k}),y^k_i-\hat{y}_i^{k} \right\rangle \leq h_i^*(y_i^k),   \quad \forall  i=1,\ldots, p.
 	\end{equation}
	Making use of this expression and $\mathbf{e}^k=\hat{\mathbf{y}}^k-\mathbf{y}^k$, we obtain
	\begin{equation}\label{eq:Phiyk}
		\begin{aligned}
			\Phi(\hat{x}^{k},\hat{\mathbf{y}}^{k})-\Phi(\hat{x}^{k},\mathbf{y}^k) & =\sum_{i=1}^p \left( h_i^\ast(\hat{y}_i^{k}) - h_i^*(y^{k})-\langle \Psi_i(\hat{x}^{k}),e_i^{k} \rangle \right)
			 \leq-\sum_{i=1}^p   \frac{1}{\mu_i^k} \|e_i^k\|^2.
		\end{aligned}
	\end{equation}
Then, \eqref{eq:Phiyk} and \eqref{eq:Phixk} give
\begin{equation}\label{eq:hats}
\Phi(\hat{x}^{k},\hat{\mathbf{y}}^{k})\leq \Phi(x^k,\mathbf{y}^k)-\left(\frac{1}{2\gamma_k}- \kappa -\frac{1}{2}\sum_{i=1 }^p L_i\| y_i^k\|\right)\|d^k\|^2 - \sum_{i=1}^p   \frac{1}{\mu_i^k} \|e_i^k\|^2.
\end{equation}
Finally, using the linesearch~\eqref{eq:whilecond} and~\eqref{eq:hats}, we get
\begin{align*}
\Phi(x^{k+1},\textbf{y}^{k+1})&\leq \Phi(\hat{x}^{k}, \hat{\mathbf{y}}^{k})-\alpha \lambda_k^2 \| (d^{k},\mathbf{e}^k)\|^2\\
&\leq \Phi({x}^{k}, {\mathbf{y}}^{k})- \left(\alpha \lambda_k^2+\frac{1}{2\gamma_k}- \kappa -\frac{1}{2}\sum_{i=1 }^p L_i\| y_i^k\|\right) \| d^{k}\|^2 \\
&\quad- \sum_{i=1}^p    \left(\alpha \lambda_k^2+\frac{1}{\mu_i^k}\right)  \| e_i^{k}\|^2\\
&=\Phi({x}^{k}, {\mathbf{y}}^{k})- a_k \| x^{k+1}-x^k\|^2 - \sum_{i=1}^p    b_i^k  \| y_i^{k+1}-y_i^k\|^2,
\end{align*}
where we note that the first inequality trivially holds when the linesearch procedure was not successful, as in that case $\lambda_k=0$ by Step~7. Therefore,~\eqref{eq:des} holds.
\end{proof}
Next, we present the convergence result of Algorithm~\ref{alg:1}.
\begin{theorem}\label{th:conv}
Suppose that Assumption~\ref{assump:1} holds and consider the functions $\varphi$ and $\Phi$ defined in~\eqref{Def_varphi} and~\eqref{Def_PHI}, respectively.   Given $(x^0,\mathbf{y}^0)\in\mathbb{R}^n\times\Rnm$ and $\eta\in{]0,1[}$, consider the pair of sequences $(x^k,\mathbf{y}^k)_{k\in\mathbb{N}}$ generated by Algorithm~\ref{alg:1} with $\gamma_k\in{\left]0,\eta \min{\left\{\gamma^g, \left(2\kappa + \sum_{i=1 }^p L_i\left\| y_i^k\right\|\right)^{-1}\right\}}\right]}$ for all $k\in\N$ and $\sup_{ k \in \N, i=1,\ldots, p} \mu_i^k<+\infty$. Then, either Algorithm~\ref{alg:1} stops at a critical point of~\eqref{eq:P2} after a finite number of iterations or it generates an infinite sequence $(x^k,\mathbf{y}^k)_{k\in\mathbb{N}}$ such that the following assertions hold:
\begin{enumerate}[(i)]
    \item\label{it:th-1} The sequence $\big(\Phi(x^k,\mathbf{y}^k)\big)_{k\in\mathbb{N}}$ monotonically (strictly) decreases and converges. Moreover,  the sequences $(x^k)_{k\in\mathbb{N}}$ and $(\mathbf{y}^k)_{k\in\mathbb{N}}$ verify that
    \begin{equation}\label{eq:Ostrowski}
        \sum_{k=0}^{\infty} \|x^{k+1}-x^k\|^2 < \infty \text{ and } \sum_{k=0}^{\infty} \|\mathbf{y}^{k+1}-\mathbf{y}^k\|^2 < \infty.
    \end{equation}
    \item\label{it:th-2} If the sequence $(x^k,\mathbf{y}^k)_{k\in\mathbb{N}}$ is bounded, the set of its accumulation points is nonempty, closed and connected.
    \item\label{it:th-3} If $(\bar{x},\barbf{y})\in\mathbb{R}^n\times\Rnm$ is an accumulation point of the sequence $(x^k,\mathbf{y}^k)_{k\in\mathbb{N}}$, then there exists $\bar{v}\in\partial f(\bar{x})$ such that~\eqref{eq:critpoint} holds, i.e., $\bar{x}$ is a critical point of~\eqref{eq:P1}. In addition, $\varphi(\bar{x}) = \inf_{k\in\mathbb{N}}\Phi(x^k,y^k)$.
    \item\label{it:th-4} If $(x^k,\mathbf{y}^k)_{k\in\mathbb{N}}$ has at least one isolated accumulation point, then the whole sequence $(x^k,\mathbf{y}^k)_{k\in\mathbb{N}}$ converges to a critical point of~\eqref{eq:P2}. Consequently, $(x^k)_{k\in\mathbb{N}}$ converges to a critical point of problem~\eqref{eq:P1}.
\end{enumerate}
\end{theorem}
\begin{proof}
If  Algorithm~\ref{alg:1} stops at some iteration $k+1$ with $x^*=x^{k+1}=x^k$ and $\mathbf{y}^{k+1}=\mathbf{y}^k$, then $x^*$ is a critical point of~\eqref{eq:P1}, as shown in Remark~\ref{r:alg1}.   Otherwise, Algorithm~\ref{alg:1} generates an infinite sequence $(x^k,\mathbf{y}^k)_{k\in\mathbb{N}}$.

(\ref{it:th-1}) Again, observe that  $(x^k,\mathbf{y}^k) \in  \dom \Phi$ for all $k\geq 1$. By   Proposition~\ref{p:des}, summing~\eqref{eq:des} for all $k\geq 1$, we get
\begin{equation}\label{eq:3}
\begin{aligned}
\Phi(x^1,\mathbf{y}^1) - \inf_{ k\in \N }  \Phi(x^k,\mathbf{y}^k)&\geq\sum_{k=1}^{\infty} a_k\|x^{k+1}-x^k\|^2  + \sum_{k=1}^{\infty} \sum_{i=1}^p b_i^k\|\mathbf{y}^{k+1}-\mathbf{y}^k\|^2\\
&\geq C\left(\sum_{k=1}^{\infty} \|x^{k+1}-x^k\|^2  + \sum_{k=1}^{\infty} \|\mathbf{y}^{k+1}-\mathbf{y}^k\|^2 \right),
\end{aligned}
\end{equation}
where $C:= \inf_{k\in\N, i=1,\ldots, p} \left\{ a_k,b_i^k  \right\}$. Let us see that $C>0$. Indeed, minimizing the value of $a_k$ with respect to $\lambda_k$, we deduce
$$a_k=\frac{2\alpha \lambda_k^2+\gamma_k^{-1}- 2\kappa -\sum_{i=1 }^p    L_i\| y_i^k\|}{2(1+\lambda_k)^2}\geq\frac{(\gamma_k^{-1}- 2\kappa-\sum_{i=1 }^p    L_i\| y_i^k\|)\alpha}{\gamma_k^{-1}- 2\kappa-\sum_{i=1 }^p    L_i\| y_i^k\|+2\alpha},$$
whose right-hand side, as a function of $\gamma_k$, is strictly decreasing in $\left]0,\eta / \left(2\kappa + \sum_{i=1 }^p L_i\left\| y_i^k\right\|\right)\right]$. Hence,
$$a_k\geq \frac{(1-\eta)\alpha}{1-\eta+2\alpha\eta/(2\kappa+\sum_{i=1 }^p    L_i\| y_i^k\|)}\geq\frac{(1-\eta)\alpha}{1-\eta+\alpha\eta/\kappa}>0, \quad \forall k\in\N.$$
Likewise,
\begin{align*}
b^k_i=\frac{1+\alpha \lambda_k^2\mu_i^k}{\mu_i^k(1+\lambda_k)^2}\geq\frac{\alpha}{1+\alpha\mu_i^k}\geq\frac{\alpha}{1+\alpha\sup_{k\in\N,i=1,\ldots,p}{\mu_i^k}}>0,\quad \forall k\in\N.
\end{align*}
Therefore, $C>0$ and we obtain from~\eqref{eq:3} that
\begin{equation*}
\begin{aligned}
  \sum_{k=1}^{\infty} \|x^{k+1}-x^k\|^2  + \sum_{k=1}^{\infty} \|\mathbf{y}^{k+1}-\mathbf{y}^k\|^2  \leq C^{-1} \left(  \Phi(x^1,\mathbf{y}^1) - \inf_{ k\in \N }  \Phi(x^k,\mathbf{y}^k) \right).
   \end{aligned}
\end{equation*}
By assumption,  the right-hand side of the  equation is  bounded from above, so the sums in the left-hand side are finite, which proves~\eqref{eq:Ostrowski}.

(\ref{it:th-2})~Equation~\eqref{eq:Ostrowski} implies that the sequences $(x^k)_{k\in\mathbb{N}}$ and $(\mathbf{y}^k)_{k\in\mathbb{N}}$ verify the so-called \emph{Ostrowski condition}, that is,
\begin{equation}\label{eq:Ostrowski2}
 \lim_{k\to\infty} \|x^{k+1}-x^k\| = 0 \text{ and } \lim_{k\to\infty} \|\mathbf{y}^{k+1}-\mathbf{y}^k\| =0.
\end{equation}
Now, the result directly follows from~\cite[Theorem~8.3.9]{MR1955649}.

(\ref{it:th-3}) Let  $(x^{k_j},\mathbf{y}^{k_j})_{j\in\mathbb{N}}$ be a subsequence of  $(x^k,\mathbf{y}^k)_{k\in\mathbb{N}}$ such that $(x^{k_j},\mathbf{y}^{k_j})\to(\bar{x},\barbf{y})$.   By~\eqref{eq:Ostrowski}, since
$$(x^{k_j+1}-x^{k_j},\mathbf{y}^{k_j+1}-\mathbf{y}^{k_j})=(1+\lambda_{k_j})(\hat{x}^{k_j}-x^{k_j},\hat{\mathbf{y}}^{k_j}-\mathbf{y}^{k_j}),$$
the sequence $(\hat{x}^{k_j},\hat{\mathbf{y}}^{k_j})_{k\in\mathbb{N}}$ also converges to $(\bar{x},\barbf{y})$.  Now, we can assume without loss of generality that $\gamma_{k_j} \to \bar{\gamma} \in {]0,\infty[}$.  Consider the subsequence $(v^{k_j})_{j\in\mathbb{N}}$ of $(v^k)_{k\in\mathbb{N}}$. In particular, $v^{k_j}\in\partial f(x^{k_j})$ for all $j\in\mathbb{N}$. Since $f$ is locally Lipschitz continuous, $v^{k_j}$ is bounded for all sufficiently large $j\in\mathbb{N}$ (see, e.g., \cite[Proposition~9.13]{MR1491362}). Therefore, we can also assume without loss of generality that $(v^{k_j})_{j\in\mathbb{N}}$ converges to some point $\bar{v}\in\partial f(\bar{x})$. Thus, the sequence
$$\left(x^{k_j}+\gamma_{k_j}\sum_{i=1}^{m}\nabla\Psi_i(x^{k_j})y_i^{k_j}-\gamma_{k_j} v^{k_j},\hat{x}^{k_j}\right)_{j\in\mathbb{N}} \subseteq \gra \prox_{\gamma_{k_j} g}$$ converges to $(\bar{x}+\bar{\gamma}\sum_{i=1}^p\nabla\Psi_i(\bar{x})\bar{y}_i-\bar{\gamma} \bar{v},\bar{x})$ as $j\to\infty$. Hence, by  \cite[Theorem 1.25]{MR1491362}), we get that
\begin{equation*}
 \bar{x}\in \prox_{\bar \gamma g} \left(\bar{x}+\bar \gamma\sum_{i=1}^p\nabla\Psi_i(\bar{x})\bar{y}_i-\bar{\gamma} \bar{v}\right),
 \end{equation*}
 which implies that
\begin{equation}\label{eq:thst-1}
   \sum_{i=1}^p \nabla\Psi_i(\bar{x}) \bar{y}_i  \in\bar{v} + \partial g(\bar{x}) \subseteq \partial f(\bar{x})+\partial g(\bar{x}).
\end{equation}
Further, using~\eqref{eq:subdif1}, we get
\begin{equation*}\label{eq:subdiff_j}
    \frac{y^{k_j}_i-\hat{y}^{k_j}_i}{\mu_i^{k_j}} +  \Psi_i(\hat{x}^{k_j})\in \partial h_i^*(\hat{y}^{k_j}_i), \quad \forall i=1,\ldots,p.
\end{equation*}
Again, we can assume without loss of generality that   $\mu_i^{k_j} \to \bar{\mu}_i \in {]0,\infty[}$ as $j\to\infty$ for all $ i=1,\ldots,p$. Taking limits as $j\to\infty$,  the closedness of the subdifferential of convex functions results in
\begin{equation}\label{eq:tiii-cp}
 \begin{aligned}
 \Psi_i (\bar{x} )\in \partial h_i^*(\bar{y}_i),& \quad \forall i=1,\ldots,p.
 \end{aligned}
\end{equation}
Therefore,~\eqref{eq:thst-1} and~\eqref{eq:tiii-cp} imply that $\bar{x}$ is a critical point of~\eqref{eq:P1}.

Let us now prove that $\varphi(\bar{x}) = \inf_{k\in\mathbb{N}}\Phi(x^k,y^k)$. 
On the one  hand, due to~\eqref{eq:subdif1}, for all $i=1,\ldots,p$, we have
\begin{equation}\label{eq:2}
   h_i^*(\hat{y}_i^{k_j}) + \left\langle \frac{y_i^{k_j}-\hat{y}_i^{k_j}}{\mu^{k_j}_i}+ \Psi_i (  \hat{x}^{k_j}), \bar{y}_i - \hat{y}_i^{k_j}\right\rangle \leq h_i^*(\bar{y}_i),
\end{equation}
from the definition of the convex subdifferential.
On the other hand, for all $j\in\mathbb{N}$, we have
{
\begin{equation}\label{eq:gnorm}
\begin{aligned}
g(\hat{x}^{k_j}) & +\frac{1}{2\gamma_{k_j}}  \Bigg\| \hat{x}^{k_j}-\left(x^{k_j}+\gamma_{k_j} \sum_{i=1}^p \nabla\Psi_i(x^{k_j}) y_i^{k_j}-  \gamma_{k_j} v^{k_j}\right)\Bigg\|^2 \\
&  \leq  g(\bar{x}) + \frac{1}{2\gamma_{k_j}} \left\| \bar{x}-\left(x^{k_j}+\gamma_{k_j} \sum_{i=1}^p  \nabla\Psi_i(x^{k_j})  y_i^{k_j}-\gamma_{k_j}v^{k_j}\right)\right\|^2,
\end{aligned}
\end{equation}}%
Thus, we deduce from~\eqref{eq:2} and~\eqref{eq:gnorm} that $\limsup_{j\to\infty} \Phi(\hat{x}^{k_j},\hat{\mathbf{y}}^{k_j}) \leq \Phi(\bar{x},\barbf{y})$.

By~\eqref{eq:Ostrowski2}, we have that $(x^{k_j+1},\mathbf{y}^{k_j+1})_{j\in\N}$ also converges to $(\bar{x},\barbf{y})$, so by lower semicontinuity of the functions defining $\Phi$ in~\eqref{Def_PHI}, we have
\begin{align*}
\liminf_{j\to\infty} \Phi(x^{k_j+1},\textbf{y}^{k_j+1})&\geq \Phi(\bar{x},\bar{\mathbf{y}})\\
&\geq\limsup_{j\to\infty}\Phi(\hat{x}^{k_j},\hat{\mathbf{y}}^{k_j})\\
&\geq\limsup_{j\to\infty} \Phi(x^{k_j+1},\mathbf{y}^{k_j+1}),
\end{align*}
where the last inequality is a consequence of the linesearch~\eqref{eq:whilecond}. Therefore, using Proposition~\ref{p:solphiL}\eqref{it:psolphiL-3} and item~\eqref{it:th-1}, we obtain
$$\varphi(\bar{x})=\Phi(\bar{x},\barbf{y})=\lim_{j\to\infty}\Phi(x^{k_j+1},\mathbf{y}^{k_j+1})=\inf_{k\in\mathbb{N}}\Phi(x^{k},\mathbf{y}^k),$$
which proves the claim.

(\ref{it:th-4}) In this case, by~\cite[Proposition~8.3.10]{MR1955649} the sequence $(x^k,\mathbf{y}^k)_{k\in\mathbb{N}}$ converges to some point $(\bar{x},\barbf{y})$, which is a critical point of problem~\eqref{eq:P2} by~(\ref{it:th-3}).
\end{proof}

\begin{remark}[Possible modification of Algorithm \ref{alg:1}]
It is worth mentioning that it would be possible to replace $\hat{x}^{k}$ in~\eqref{eq:algy} by some interpolation between the points $x^k$ and $\hat{x}^{k}$ of the form $(1-\beta^k_i) \hat{x}^{k} + \beta^k_i x^k$, where $\beta^k_i \in \mathbb{R}$ is arbitrarily chosen at each iteration. In principle, this could allow to improve the overall performance of the algorithm. For example, setting $\beta_i = 1$ for all $i=1,\ldots,p$ would permit to fully run the algorithm in parallel, since only $x^{k}$ and $y_i^k$ would be required to compute $\hat{y}_i^{k}$. This would allow to simultaneously compute $\hat{x}^{k}$ and $\hat{y}_i^{k}$, which could improve the algorithm's overall efficiency when the computation of the proximal mappings is time-consuming. However, it is important to note that the bounds obtained in the accordingly modified version of Proposition~\ref{p:des} would be considerably more technical and difficult to be satisfied in practical applications when $\beta_i\neq 0$. As, in addition, in our experiments we did not find numerical evidence of the benefits to justify the inclusion of such extra linear terms, for conciseness we only consider the case where $\beta_i^k=0$.
\end{remark}

\begin{remark}\label{Remark_sub}
As mentioned in Remark~\ref{r:alg1}, one can allow choosing $v^k \in \conv{\partial f(x^k)}$ in Algorithm~\ref{alg:1}. The only modification in Theorem~\ref{th:conv} is that if $(x^{k_j}, \mathbf{y}^{k_j})_{k\in\N}$ converges to
$(\bar{x},\bar{\mathbf{y}})$,  then there exists $\bar{y}_i \in \partial h_i (\Psi_i( \bar{x}))$ for all $ i=1,\ldots,p$, such that
\begin{equation}\label{general_critical_point}
	\sum_{i=1}^p \nabla\Psi_i(\bar{x}) \bar{y}_i    \in \limsup_{j\to +\infty} \left\{v^{k_j}\right\} +\partial g(\bar{x}),
\end{equation}
where here $\limsup$ refers to the Painlev\'e--Kuratowski upper-limit of the sequence $(v^{k_j})_{k\in  \N }$.	Furthermore, it is easy to prove that $\limsup_{j\to +\infty} \left\{v^{k_j}\right\} \subseteq \conv \partial f(\bar x)$.
\end{remark}

\subsection{Convergence under the Kurdyka--{\L}ojasiewicz property}\label{subsec:KL}
In this subsection, we establish the global convergence of Algorithm~\ref{alg:1} and some convergence rates. In addition to the assumptions required by Theorem~\ref{th:conv}, we assume that the primal-dual function $\Phi$ satisfies the Kurdyka--{\L}ojasiewicz property at some accumulation point of the sequence generated by Algorithm~\ref{alg:1}. Recall that the Kurdyka--{\L}ojasiewicz property holds for $\phi:\mathbb{R}^n\to\mathbb{R}$ at $\bar{x}\in\mathbb{R}^n$ if there exists $\beta>0$ and a continuous concave function $ \theta:[0,\beta] \to [0,+\infty[$ such that $\theta(0)=0$, $\theta$ is $\mathcal{C}^1$-smooth on $]0,\beta[$ with a strictly positive derivative $\theta'$ and
\begin{align}\label{Kur-Loj}
	\theta'\big(\phi(x) - \phi(\bar{x})\big)\,d(0,\partial \phi(x))\geq 1
\end{align}
for all $x\in \mathbb{B}_\beta(\bar{x})$ with $\phi(\bar{x}) < \phi(x) <\phi( \bar{x} ) + \beta$, where $d(\cdot,\Omega)$ stands for the distance function to a set $\Omega$.

\begin{lemma}\label{Lemma_desgrad}
Consider a point $(\bar{x},\bar{\mathbf{y}}) \in \mathbb{R}^n\times \Rnm $. In addition to   the assumptions of Theorem {\rm\ref{th:conv}},  suppose that $f$ is $\mathcal{C}^{1,+}$   around $\bar{x}$.  Then, there exists   $r>0$,  $\rho>0$ and $\hat{k}\in\N$ such that, for all  $(x^k, \mathbf{y}^k) \in \mathbb{B}_r(\bar{x},{\bar{\mathbf y}})$ with $k \geq \hat{k}$, there exists $(u^k,\mathbf{w}^k) \in \partial \Phi(\hat{x}^k,\hat{\mathbf{y}}^k)$ verifying
	\begin{align}\label{eq001_Lemma_desgrad}
		\| (u^k,\mathbf{w}^k)\| \leq \rho \| (x^{k+1},\mathbf{y}^{k+1}) - (x^{k},\mathbf{y}^{k})\|.
		\end{align}
\end{lemma}
\begin{proof}
	Let $L_1>0$ and $r>0$ be such that $f$  is continuously differentiable with $L_1$-Lipschitz gradient on $\mathbb{B}_{2r}(\bar{x},\bar{\mathbf{y}})$. Let $\hat{k}\in \N$ be such that $\| (x^{k+1},\mathbf{y}^{k+1}) - (x^{k},\mathbf{y}^{k})\| \leq r$ for all $k\geq \hat{k}$. Now, consider  $(x^k, \mathbf{y}^k) \in \mathbb{B}_r(\bar{x},{\bar{\mathbf y}})$ with $k \geq \hat{k}$. It follows that $(\hat{x}^k,\hat{\mathbf{y}}^k)$ belongs to $\mathbb{B}_{2r}(\bar{x},{\bar{\mathbf y}})$. Using \eqref{eq:subdif1}, we get that
	\begin{equation*}
		\begin{aligned}
			&\frac{x^{k}-\hat{x}^{k}}{\gamma_{k}} + \sum_{i=1}^p \nabla \Psi_i(x^{k}) y_i^{k} -  \nabla f(x^{k})   \in \partial g(\hat{x}^{k}),
			\\
			&\frac{y_i^{k}-\hat{y}_i^{k}}{\mu_i^{k}} + \Psi_i (  \hat{x}^{k})  \in \partial h_i^*(\hat{y}_i^{k}), \quad \forall i=1, \ldots,p.
		\end{aligned}
	\end{equation*}
Let us define $(u^k,\mathbf{w}^k)=(u^k, w_1^k, \ldots,w_p^k)$ by
\begin{align*}
	u^k &:= \frac{x^{k}-\hat{x}^{k}}{\gamma_{k}} +\sum_{i=1}^p \left(  \nabla \Psi_i(x^{k}) y_i^{k} - \nabla \Psi_i(\hat{x}^{k}) \hat{y}_i^{k} \right)+ \nabla f(\hat{x}^{k})-  \nabla f(x^{k}),\\ 
	w_i^k&:= \frac{y_i^{k}-\hat{y}_i^{k}}{\mu_i^{k}}, \quad \forall i=1,\ldots, p.
\end{align*}
It follows from~\eqref{eq:subdiff_Phi} that $(u^k,\mathbf{w}^k) \in \partial \Phi(\hat{x}^k,\hat{\mathbf{y}}^k)$.  Furthermore, the function $(x,\mathbf{y}) \mapsto \sum_{i=1}^p \nabla\Psi_i(x)y_i$  is Lipschitz continuous on $\mathbb{B}_{2r}(\bar{x})$, let us say $L_2$-Lipschitz continuous, so we can make the following estimations
\begin{align*}
	 \| u^k\| &\leq \left(\frac{1}{\gamma_{k}}+L_1\right)\|x^{k}- \hat{x}^{k}\|  + L_2 \|(x^{k}- \hat{x}^{k}, \mathbf{y}^k  - \hat{\mathbf{y}}^{k})\|, \\
	 \|w_i^k\| &\leq  \frac{1}{\mu_i^{k}}\|y_i^{k}-\hat{y}_i^{k}\|,  \quad \forall i=1,\ldots, p.
\end{align*}
Finally, let us notice that  by Step~8 of Algorithm~\ref{alg:1} we have that
\begin{equation}\label{eq:ineq_hat}
\begin{aligned}
\|(x^{k+1},\mathbf{y}^{k+1}) - (x^{k},\mathbf{y}^{k})\|&=(1+\lambda_k)\|( \hat{x}^{k},\hat{\mathbf{y}}^{k})- (x^{k},\mathbf{y}^{k})\|\\
&\geq \|( \hat{x}^{k},\hat{\mathbf{y}}^{k})- (x^{k},- \mathbf{y}^{k})\| ,
\end{aligned}
\end{equation}
which yields~\eqref{eq001_Lemma_desgrad} taking $\rho>0$ sufficiently large.
\end{proof}
\begin{theorem}\label{t:KL}
In addition  the assumptions of Theorem~{\rm\ref{th:conv}}, suppose that the sequence $(x^k, \mathbf{y}^k)_{k\in\N}$ generated by Algorithm~{\rm\ref{alg:1}} has an accumulation point $(\bar{x},\bar{\mathbf{y}})\in\mathbb{R}^n\times\Rnm$ at which the Kurdyka--{\L}ojasiewicz property \eqref{Kur-Loj} holds, assume that $f$ is $\mathcal{C}^{1,+}$  around $\bar{x}$, and $\sup_{k\in\N}\lambda_k<+\infty$. Then $(x^k, \mathbf{y}^k )_{k\in\N}$ converges to $(\bar{x},\bar{\mathbf{y}})$ as $k\to\infty$.
\end{theorem}

\begin{proof}
	If  Algorithm~{\rm\ref{alg:1}}  stops after a finite number of iterations, then the results clearly holds. Otherwise, Algorithm~{\rm\ref{alg:1}} produces an infinite sequence $(x^k, \mathbf{y}^k)_{k\in\N}$. Let $r$, $\rho$ and $\hat{k}$ be the constants given by Lemma~\ref{Lemma_desgrad}, let $\beta$ and $\theta$ be the constant and function in the definition of the Kurdyka--{\L}ojasiewicz property,  and let   $c_0:=\inf_{ k\in \N,i =1,\ldots, p}\left\{\frac{(1-\eta)\kappa}{\eta}, \frac{1}{\mu_i^k}\right\} $, $\lambda_\infty:=\sup_{k\in\N}\lambda_k$ and $\sigma:= \rho(1+\lambda_\infty)^2/c_0$. Consider an arbitrary $\epsilon \in{] 0, \min\{ r, \beta/2\}]}$ and pick $k_0 \geq \hat{k}$ large enough such that the following conditions hold:
\begin{itemize}
    \item $\|(x^{k_0}, \mathbf{y}^{k_0}) -(\bar{x},\bar{\mathbf{y}})\| \leq\epsilon/4$,
	\item $\|    (x^{k+1}, \mathbf{y}^{k+1}) -	(x^{k}, \mathbf{y}^{k})\|\leq \epsilon/4$, for all $k \geq k_0$,
	\item $\sigma    \theta( \Phi(x^{k_0}, \mathbf{y}^{k_0}) -  \Phi(\bar{x},\bar{\mathbf{y}}))    \leq\epsilon/4$,
	\item $\Phi(\bar{x},\bar{\mathbf{y}}) < \Phi(\hat{x}^k, \hat{\mathbf{y}}^k) <\Phi(\bar{x},\bar{\mathbf{y}}) + \beta$, for all $k \geq k_0$,
\end{itemize}
where in the last assertion we have used the fact that $(\Phi(\hat{x}^k,\hat{\mathbf{y}}^k))_{k\in\N}$ also converges to $\Phi(\bar{x},\bar{\mathbf{y}})$, since
$$\Phi(x^{k+1},\mathbf{y}^{k+1})\leq\Phi(\hat{x}^k,\hat{\mathbf{y}}^k)\leq\Phi(x^k,\mathbf{y}^k),$$
by the linesearch~\eqref{eq:whilecond} and~\eqref{eq:hats}.

The rest of the proof is split into three claims.

 \noindent\textbf{Claim~1:} \emph{Let $k \geq k_0$ be such that  $(x^{k}, \mathbf{y}^{k}) \in  \mathbb{B}_\varepsilon(\bar{x},\bar{\mathbf{y}})$. Then,  the following estimation holds
 	\begin{align}\label{Inq_Claim1}
 	\Delta_{k+1} \leq  \left(\sigma  \Delta s_k  \Delta_{k}\right)^{\frac{1}{2}}, 
 	\end{align}
 	where $\Delta_k:= \|    (x^{k+1}, \mathbf{y}^{k+1}) -	(x^{k}, \mathbf{y}^{k})\|$, $\Delta s_k:= s_k -s_{k+1}$  and
 $s_k:=  \theta( \Phi(\hat{x}^{k}, \hat{\mathbf{y}}^{k}) -  \Phi(\bar{x},\bar{\mathbf{y}}))$. }%
 Indeed, let $(u^{k},\mathbf{w}^{k}) \in \partial \Phi(\hat{x}^{k}, \hat{\mathbf{y}}^{k})$ be the vector given by Lemma \ref{Lemma_desgrad} (recall that $(x^{k}, \mathbf{y}^{k}) \in  \mathbb{B}_r(\bar{x},\bar{\mathbf{y}})$). Observe that $(\hat{x}^k,\hat{\mathbf{y}}^k)\in\ball_\beta(\bar{x},\bar{\mathbf{y}})$, since by~\eqref{eq:ineq_hat}, it holds
\begin{align*}
\|(\hat{x}^k,\hat{\mathbf{y}}^k)-(\bar{x},\bar{\mathbf{y}})\|&\leq \|(\hat{x}^k,\hat{\mathbf{y}}^k)-(x^k,\mathbf{y}^k)\|+\|(x^k,\mathbf{y}^k)-(\bar{x},\bar{\mathbf{y}})\|\\
&\leq\|(x^{k+1},\mathbf{y}^{k+1}) - (x^{k},\mathbf{y}^{k})\|+\epsilon\leq 2\epsilon\leq \beta
\end{align*}
Then,  by the concavity of $\theta$, the Kurdyka--{\L}ojasiewicz property \eqref{Kur-Loj} applied to $(\hat{x}^k,\hat{\mathbf{y}}^k)$, inequality \eqref{eq:hats}, and the linesearch~\eqref{eq:whilecond}, we have that
 \begin{align*}
 	\Delta s_k \|   (u^k,\mathbf{w}^k)\|  &\geq \theta'( \Phi(\hat{x}^{k}, \hat{\mathbf{y}}^{k}) -  \Phi(\bar{x},\bar{\mathbf{y}}))   ) \left(\Phi(\hat{x}^{k}, \hat{\mathbf{y}}^{k}) - \Phi(\hat{x}^{k+1}, \hat{\mathbf{y}}^{k+1})  \right)\|   (u^k,\mathbf{w}^k)\| \\
 	&\geq \Phi(\hat{x}^{k}, \hat{\mathbf{y}}^{k}) - \Phi(\hat{x}^{k+1}, \hat{\mathbf{y}}^{k+1})\\
 &\geq \Phi(\hat{x}^{k}, \hat{\mathbf{y}}^{k}) - \Phi(x^{k+1},\mathbf{y}^{k+1})\\
 &\quad+\frac{1}{2}\left(\frac{1}{\gamma_{k+1}}- 2\kappa -\sum_{i=1 }^p L_i\| y_i^{k+1}\|\right)\|d^{k+1}\|^2 + \sum_{i=1}^p   \frac{1}{\mu_i^{k+1}} \|e_i^{k+1}\|^2\\
 &\geq\frac{(1-\eta)\kappa}{\eta}\|d^{k+1}\|^2 + \sum_{i=1}^p   \frac{1}{\mu_i^{k+1}} \|e_i^{k+1}\|^2\\
 &\geq c_0\|(\hat{x}^{k+1},\hat{\mathbf{y}}^{k+1})-(x^{k+1},\mathbf{y}^{k+1})\|^2\\
 &= \frac{c_0}{(1+\lambda_{k+1})^2}\Delta_{k+1}^2\geq \frac{c_0}{(1+\lambda_\infty)^2}\Delta_{k+1}^2.
 \end{align*}
 Hence,

 \begin{align*}
 \Delta_{k+1}&\leq (1+\lambda_\infty)\sqrt{   \frac{ \Delta s_k \|   (u^k,\mathbf{w}^k)\|  }{c_0}}  \leq (1+\lambda_\infty)\sqrt{  \frac{\rho}{c_0}   \Delta s_k \Delta_{k}} =  \sqrt {\sigma   \Delta s_k \Delta_{k}}       ,
 \end{align*}
 which proves the claim.

 \noindent\textbf{Claim~2:} \emph{ Let $k\geq k_0$ and assume that $(x^{j}, \mathbf{y}^{j}) \in  \mathbb{B}_\varepsilon(\bar{x},\bar{\mathbf{y}})$, for all $j\in\{k_0,\ldots, k\}$. Then
	\begin{align}\label{Inq_Claim2_ind}
	\Delta_{k+1} \leq \sigma\left(  \sum_{j=0}^{k-k_0} \frac{1}{2^{j+1}}   \Delta s_{k-j} \right)  +   \frac{1}{2^{k+1-k_0}}\Delta_{k_0}.
\end{align}
}
 Using \eqref{Inq_Claim1} inductively  for  $j\in\{k_0,\ldots, k\}$,  we get
 \begin{align*}
 	\Delta_{k+1} &\leq \left( \prod\limits_{j=0}^{k-k_0}\left(  \sigma  \Delta s_{k-j}\right)^{\frac{1}{2^{j+1}}} \right)\Delta_{k_0}^{\frac{1}{2^{k+1-k_0}}}.
 \end{align*}
 Now, let us recall the (generalized) inequality   of arithmetic and geometric means (see, e.g.,~\cite[Proposition~3.14]{Aragon2019}), which states that for any nonnegative numbers $b_0, \ldots, b_{\ell+1}$,
 \begin{align*}
 	\prod_{j=0}^{\ell+1}  b_j^{\nu_j  }  \leq \sum_{ j=0}^{\ell+1}  \nu_j b_j ,\text{ whenever } \nu_j\geq 0,\text{ with } \sum_{j=0}^{\ell+1} \nu_j = 1.
 \end{align*}
 Using this inequality with $b_j:= \sigma  \Delta s_{k-j}$ and $\nu_j := \frac{1}{2^{j+1}}$ for $j=0,\ldots, k-k_0=:\ell$, and $b_{\ell+1} :=\Delta_{k_0}$ and $\nu_{\ell+1}:=\frac{1}{2^{k+1-k_0}}$,  we have that
 \begin{align*}
 	\left( \prod\limits_{j=0}^{k-k_0}\left(  \sigma  \Delta s_{k-j}\right)^{\frac{1}{2^{j+1}}} \right)\Delta_{k_0}^{\frac{1}{2^{k+1-k_0}}} \leq \sigma\left(  \sum_{j=0}^{k-k_0} \frac{1}{2^{j+1}}   \Delta s_{k-j} \right)  +   \frac{1}{2^{k+1-k_0}}\Delta_{k_0}
 \end{align*}
which concludes the proof of \eqref{Inq_Claim2_ind}.

 \noindent\textbf{Claim~3:} \emph{For all $k\geq k_0$,  $(x^{k}, \mathbf{y}^{k}) \in  \mathbb{B}_\varepsilon(\bar{x},\bar{\mathbf{y}})$. Therefore, $(x^{k}, \mathbf{y}^{k})_{k\in \N}$ converges to  $(\bar{x},\bar{\mathbf{y}})$.
 }

We prove by induction that $(x^{k}, \mathbf{y}^{k}) \in  \mathbb{B}_\varepsilon(\bar{x},\bar{\mathbf{y}})$ for all $k\geq k_0$. The assertion clearly holds for $k=k_0$ and $k=k_0+1$, so we can assume that there is $k_1> k_0+1$ such that $(x^{k}, \mathbf{y}^{k}) \in  \mathbb{B}_\varepsilon(\bar{x},\bar{\mathbf{y}})$ for all $k\in\{k_0,\ldots,k_1\}$. Then  \eqref{Inq_Claim2_ind} holds for all $k \in  \{k_0,\ldots, k_1\}$, so we get that
\begingroup\allowdisplaybreaks
\begin{align*}
\sum_{k=k_0}^{k_1}\| (x^{k+1}, \mathbf{y}^{k+1}) - (x^{k}, \mathbf{y}^{k})\|&= \Delta_{k_0}+\sum_{k=k_0}^{k_1-1}\Delta_{k+1}\\
 	& \leq \Delta_{k_0} +  \sum_{k=k_0}^{k_1-1}\left(    \sigma\left(  \sum_{j=0}^{k-k_0} \frac{1}{2^{j+1}}   \Delta s_{k-j} \right)  +   \frac{1}{2^{k+1-k_0}}\Delta_{k_0}     \right) \\
 	&\leq \Delta_{k_0} +\sigma \sum_{k=k_0}^{k_1-1}\left(       \sum_{j=0}^{k-k_0} \frac{1}{2^{j+1}}   \Delta s_{k-j} \right)   + \left(\sum_{k=1}^{\infty}  \frac{1}{2^{j}}    \right) \Delta_{k_0}\\
 &\leq 2\Delta_{k_0}  +\sigma \sum_{k=k_0}^{k_1-1}\left(       \sum_{j=0}^{k-k_0} \frac{1}{2^{j+1}}   \Delta s_{k-j} \right) \\
 &= 2\Delta_{k_0} +\sigma \sum_{j=1}^{k_1-k_0}\frac{1}{2^{j}} \left(       \sum_{k=k_0}^{k_1-j}    \Delta s_{k} \right) \\
 &\leq 2\Delta_{k_0}  +\sigma \sum_{j=1}^{k_1-k_0}\frac{1}{2^{j}}s_{k_0} \leq 2\Delta_{k_0}  +   \sigma s_{k_0}.
\end{align*}
\endgroup
Hence,
\begingroup\allowdisplaybreaks
\begin{align*}
    \| (x^{k_1+1}, \mathbf{y}^{k_1+1})  - (\bar{x},\bar{\mathbf{y}})\|   &\leq 	\| (x^{{k_0}}, \mathbf{y}^{{k_0}})  - (\bar{x},\bar{\mathbf{y}})\| + \sum_{k=k_0}^{k_1}\| (x^{k+1}, \mathbf{y}^{k+1}) - (x^{k}, \mathbf{y}^{k})\|\\
&\leq \frac{\epsilon}{4} + 2\Delta_{k_0}  +   \sigma s_{k_0}\leq\frac{\epsilon}{4}+\frac{2\epsilon}{4}+\frac{\epsilon}{4}=\epsilon,
\end{align*}
\endgroup
which demonstrates the assertion for $k=k_1+1$.

Therefore,
 \begin{equation}\label{eq:sumDeltak}
 \sum_{k=k_0}^{\infty}\| (x^{k+1}, \mathbf{y}^{k+1}) - (x^{k}, \mathbf{y}^{k})\|\leq 2\Delta_{k_0}  +   \sigma s_{k_0},
\end{equation}
which proves that $(x^{k}, \mathbf{y}^{k})$ is a Cauchy sequence, so it converges to  $(\bar{x},\bar{\mathbf{y}})$.
\end{proof}

The next theorem allows to deduce convergence rates of the sequence generated by Algorithm~\ref{alg:1} when the Kurdyka--{\L}ojasiewicz property holds for a specific choice of function~$\theta$.
\begin{theorem}[Convergence rates]
In addition to the assumptions of Theorem~\ref{t:KL}, suppose that the function $\theta$ in the definition of the Kurdyka--{\L}ojasiewicz property is given by $\theta(t) := Mt^{1-\vartheta}$ for some $M>0$ and $0\leq \vartheta <1$. Then, we obtain the following convergence rates:
\begin{enumerate}[(i)]
 \item\label{it:convrates-1} If $\vartheta=0$, then the sequence $(x^k,\mathbf{y}^k)_{k\in\mathbb{N}}$ converges in a finite number of steps to $(\bar{x},\barbf{y})$.
 \item\label{it:convrates-2} If $\vartheta\in{\left]0,\frac{1}{2}\right]}$, then the sequence $(x^k,\mathbf{y}^k)_{k\in\mathbb{N}}$ converges linearly to $(\bar{x},\barbf{y})$.
 \item\label{it:convrates-3} If $\vartheta\in{\left]\frac{1}{2},1\right[}$, then there exists a positive constant $\varrho$  such that for all $k$ large enough
 \begin{equation*}
 \|(x^k,\mathbf{y}^k)-(\bar{x},\barbf{y})\| \leq \varrho k^{-\frac{1-\vartheta}{2\vartheta-1}}.
 \end{equation*}
\end{enumerate}
\end{theorem}
\begin{proof}
For proving~\eqref{it:convrates-1}, let $\vartheta=0$. By~\eqref{Kur-Loj}, \eqref{eq001_Lemma_desgrad} and Claim~3 from  previous theorem, we have  that
$$
1 \leq M \|(u^k,\mathbf{w}^k)\| \leq \rho \|(x^{k+1},\mathbf{y}^{k+1})-(x^k,\mathbf{y}^k)\|
$$
for all $k$ sufficiently large. Therefore, Theorem~\ref{th:conv} concludes that $(x^k,\mathbf{y}^k)_{k\in\mathbb{N}}$ stops after a finite number of iterations, or otherwise we would enter into contradiction with~\eqref{eq:Ostrowski}.

For the remaining cases, consider the sequence $S_{k}:=\sum_{\ell=k}^{\infty} \|(x^{\ell+1},\mathbf{y}^{\ell+1})-(x^{\ell},\mathbf{y}^{\ell})\|$, which is finite for any $k\geq0$ due to~\eqref{eq:sumDeltak}. The convergence of $(x^{k},\mathbf{y}^{k})_{k\in\mathbb{N}}$ to $(\bar{x},\barbf{y})$ can be studied by means of $S_k$ since $\|(x^{k},\mathbf{y}^{k})-(\bar{x},\barbf{y})\|\leq S_{k}$.

Recall that $\theta^{\prime}(t) = (1-\vartheta)Mt^{-\vartheta}$. Then, for any $k$ large enough  \eqref{eq:sumDeltak} implies
\begin{equation}\label{eq:S_k}
\begin{aligned}
S_{k}& = \sum_{\ell=k}^{\infty} \Delta_{\ell} \leq 2\Delta_{k} + \sigma s_{k} \\
&= 2 \Delta_{k} + \sigma M  \left(\Phi(\hat{x}^k,\hatbf{y}^k)-\Phi(\bar{x},\barbf{y})\right)^{1-\vartheta} \\
&= 2 \Delta_{k} + \frac{\sigma M^{\frac{1}{\vartheta}}(1-\vartheta)^{\frac{1-\vartheta}{\vartheta}}}{\theta^{\prime}\left(\Phi(\hat{x}^k,\hatbf{y}^k)-\Phi(\bar{x},\barbf{y})\right)^{\frac{1-\vartheta}{\vartheta}}} \\
&\leq  2 \Delta_{k} + \sigma M^{\frac{1}{\vartheta}}(1-\vartheta)^{\frac{1-\vartheta}{\vartheta}}\rho^{\frac{1-\vartheta}{\vartheta}} \Delta_{k}^{\frac{1-\vartheta}{\vartheta}},
\end{aligned}
\end{equation}
where the last inequality is due to~\eqref{Kur-Loj} and~\eqref{eq001_Lemma_desgrad}. If $\vartheta\in{\left]0,\frac{1}{2}\right]}$ , the dominant term in the right hand side of the above equation is the first summand. Therefore, there exists $k_1>0$ and $K_1>0$ such that
$$
S_k \leq K_1 \Delta_k, \quad \text{ for all } k\geq k_1.
$$
This implies~\eqref{it:convrates-2} by resorting to~\cite[Lemma~1(ii)]{aragon2018accelerating}. On the other hand, if $\vartheta\in{\left]\frac{1}{2},1\right[}$, the second term in the right hand side of~\eqref{eq:S_k} would be the dominant one. This yields the existence of some $k_2>0$ and $K_2>0$  such that
$$
S_k^{\frac{\vartheta}{1-\vartheta}} \leq K_2 \Delta_k, \quad \text{ for all } k\geq k_2.
$$
Finally, the conclusion of~\eqref{it:convrates-3} similarly follows from~\cite[Lemma~1(iii)]{aragon2018accelerating}.
\end{proof}

\section{Numerical Experiments}\label{sect:numerical}

In this section we present some computational experiments where we evaluate the performance of Algorithm~\ref{alg:1}. We recall that when $R>0$, so the linesearch in Steps~$6$-$7$ is performed, the resulting algorithm is referred as BDSA; otherwise, Algorithm~\ref{alg:1} without linesearch is named as DSA.

The linesearch of BDSA requires the selection of a number of hyperparameters, namely, the initial stepsize $\overline{\lambda}_k$, the backtracking constant~$\rho$ and the number of trials $R$. This may seem to be a drawback, as each particular problem could require of a specific tuning of all these parameters in order to obtain a good performance of the linesearch. Quite the opposite, the next numerical experiments on very different applications evidence that this is not the case: we ran all instances of BDSA with the same choice of parameters specified below and this general tuning was good enough for BDSA to significantly outperform its counterpart DSA with no linesearch.
\paragraph{Parameter tuning for Algorithm~\ref{alg:1} linesearch}
All the linesearches for BDSA in our numerical experiments were performed with the following choice of parameters: $R=2$, $\rho = 0.5$ and $\alpha=0.1$. The initial stepsize $\overline{\lambda}_k$ was chosen according to the self-adaptive trial stepsize scheme presented in Algorithm~\ref{alg:fortracking} with $\lambda_0=2$ and $\delta=2$.

\begin{algorithm}[H]
 \caption{Self-adaptive trial stepsize}\label{alg:fortracking}
 \begin{algorithmic}[1]
\Require{$\delta>1$ and $\lambda_0>0$. Obtain $\lambda_k$ from $\overline{\lambda}_k$ by Steps $4$-$7$ of BDSA (Algorithm~\ref{alg:1}).}
 \If{$r=0$}
 \State{set $\overline{\lambda}_{k+1}:=\delta\overline{\lambda}_k$;}
 \Else
 \State{set $\overline{\lambda}_{k+1}:=\max{\{ \lambda_0, \rho^{r}\,\overline{\lambda}_k\}}$.}
 \EndIf
\end{algorithmic}
\end{algorithm}
The \emph{self-adaptive} trial stepsize given by Algorithm~\ref{alg:fortracking} is based on the one proposed in~\cite{boostedDCA} for the \emph{Boosted Difference of Convex functions Algorithm} (BDCA). We note that a similar adaptive scheme  for the \emph{gradient descend method} was recently introduced in~\cite{truong2021backtracking}. The procedure works as follows.
Algorithm~\ref{alg:fortracking} determines how to choose the starting stepsize $\overline{\lambda}_{k+1}$ of the next iteration of the method. If in the current iteration a decrease of $\Phi$ was achieved in the first attempt of the linesearch (i.e., when $r=0$),  then the starting stepsize for the next iteration of the main algorithm is increased by setting $\overline{\lambda}_{k+1}:=\delta \lambda_k$, with $\delta>1$. Otherwise, $\overline{\lambda}_{k+1}$ is set as the maximum between the default initial stepsize  and the smallest stepsize accepted in the previous iteration, i.e.,  $\overline{\lambda}_{k+1} :=\max{\{ \lambda_0, \rho^{r}\,\overline{\lambda}_k\}}$ (observe that $\lambda_k$ could be zero if the linesearch was not successful).

This section is divided into three subsections, each containing a different application. The purpose of the experiments in the first subsection is twofold. First, to demonstrate how the linesearch from the boosting step can help reaching better critical points. Second, to show that the assignment of the terms of the objective function of~\eqref{eq:P1} to each of the functions $f$, $g$, $h_i$ and $\Psi_i$ has a big impact in the success of the resulting scheme derived from Algorithm~\ref{alg:1}. In Subsection~\ref{sect:numerical2} we consider an application with real-data for clustering cities in a region and show how the linesearch of BDSA helps finding better solutions in considerably less time than DSA (which, in this context, coincides with GPPA~\cite{an2017convergence}). Lastly, Subsection~\ref{sect:numerical3} contains a nonconvex generalization of Heron's problem that can be addressed with BDSA. In this case, BDSA is not a particular instance of any other known algorithm.

All the experiments were ran in a computer of Intel Core i7-12700H   2.30 GHz with 16GB RAM, under Windows 11 (64-bit).

\subsection{Avoiding Non-Optimal Critical Points}\label{subsect:avoidingcriticalpoints}

Theorems~\ref{th:conv} and~\ref{t:KL} prove the convergence of Algorithm~\ref{alg:1} to some critical point of~\eqref{eq:P1}. We recall that being a critical point is a necessary (but not sufficient) condition for local optimality of problem~\eqref{eq:P1}. In~\cite[Example~3.3]{boostedDCA} it was shown how the linesearch performed by the BDCA helps prevent the algorithm from being trapped by critical points which are not local minima. In this subsection we illustrate the same phenomenon by considering different known algorithms that can be obtained as particular cases of Algorithm~\ref{alg:1}. We show that its \emph{boosted} version, with the additional linesearch, outperforms the basic methods in avoiding these non-desirable critical points.

To this aim, we introduce a new family of functions which entails a challenge to this class of methods. These functions have a large number of critical points where the algorithm can easily get stuck, but a unique global minimum. Specifically, for any $q\in \mathbb{N}$, we define the functions $\varphi_{q}:\mathbb{R}^n\to\mathbb{R}$ as
\begin{equation}\label{eq:func_phi_q}
\varphi_{q}(x) := \|x\|^2 - \|x\|_{1} - \sum_{j=1}^q \left( \|x-je\|_{1} + \|x+je\|_{1}\right) - \|x-(q+1)e\|_{1},
\end{equation}
where $e$ is the vector of ones in $\R^n$.
It is a simple exercise to check that the function $\varphi_q$ possesses $(2q+3)^n$ critical points, which are given by the set ${\{-(q+1),-q,\ldots,0,\ldots,q,q+1\}}^n$, and a unique local minimum at  $x^* := (-(q+1),-(q+1),\ldots,-(q+1))^T$, which corresponds to its global optimum, with optimal value $\varphi_q(x^*)=-n(q^2+3q+2)$.

Note that the function $\varphi_q$  admits different representations as an instance of~\eqref{eq:P1}, and different algorithms are derived from BDSA depending on which terms one assigns to each of the functions $f$, $g$ and $h_i$ (recall Remark~\ref{remark:particularcases}):
\begin{itemize}
\item Setting $f(x):= 0$, $g(x) := \| x\|^2$  and $h_i$, for $i= 1,\ldots, 2q+2$, to be the remaining terms involving the $\ell_1$-norm, the \emph{Double-proximal Gradient Algorithm} (DGA) by Banert--Bo\c{t}~\cite{bot2019doubleprox} is obtained.
\item If we take $f(x):= - \|x\|_{1} - \sum_{j=1}^q \left( \|x-je\|_{1} + \|x+je\|_{1}\right) -\|x-(q+1)e\|_{1}$, $g(x) := \| x\|^2$ and $h(x):= 0$, we recover the particular case of the \emph{Proximal DC Algorithm} (PDCA) discussed in Remark~\ref{remark:particularcases}~\eqref{it:particularcases1}, which would become the \emph{Boosted Proximal DC Algorithm} (BPDCA) from~\cite{alizadeh2022new} when $R=\infty$ (but recall that we take $R=2$ in our experiments). Due to variable separability of the $\ell_1$-norm, it can be proved that the subdifferential of  $f$ coincides with the sum of subdifferentials of the $\ell_1$-norm terms. Therefore, for every $k\geq0$, we take $v^k \in \partial f(x^k)$ as a sum of subgradients of the form
\[
v^k = \sum_{s\in I} v^k_{s} \quad  \text{ where }  \quad v^k_{s}\in\partial\left(-\|\cdot-se\|_1\right) (x^k),
\]
where $I:=\{0,1,-1,\ldots,q,-q,q+1\}$ and every subgradient is componentwise chosen as
\begin{equation*}
(v^k_{s})_i = \left\{
\begin{aligned}
1 & \text{ if } x^k_i\leq s, \\
-1 & \text{ if } x^k_i >s,
\end{aligned}
\right. \quad \text{ for all  }i=1,\ldots,n.
\end{equation*}
\end{itemize}

For different combinations of $n$ and $q$, we performed $10\,000$ runs of DGA, PDCA, and their boosted counterparts with linesearch (abbr. as BDGA and BPDCA) all initialized at the same starting points randomly chosen in the interval $[-q-2,q+2]^{n}$, and $[-1,1]^n$ for the dual variables (when necessary), with $\mu=\gamma=1$. We note that the conjugate of the $\ell_1$-norm is the indicator function of $[-1,1]^n$, so this seems a fair set in which to choose the initial dual variables.
We stopped all the algorithms when the norm of the difference between two consecutive iterates is smaller than $n\times 10^{-6}$  and counted how many times each of the methods converged to the optimal solution $x^*$. The results are summarized in Table~\ref{tab:counterexample}.

\begin{table}[ht!]\centering
\begin{tabular}{rrcccc}
\toprule
n & q& DGA& BDGA& PDCA & BPDCA \tabularnewline
 \midrule
$2$ & $3$ &273  &1\,202& 410 & 10\,000 \tabularnewline
 \midrule
$2$ & $5$ &72  &774&  201 & 10\,000 \tabularnewline
 \midrule
$2$ & $10$ &10  &440&  71 & 10\,000  \tabularnewline
 \midrule
$2$ & $20$ & 0 &253 & 21 & 10\,000 \tabularnewline
 \midrule
$10$& $3$ & 0 &2\,229& 0 & 10\,000 \tabularnewline
\midrule
$20$ & $3$ &0  &2\,076&  0 & 10\,000 \tabularnewline
\bottomrule
\end{tabular}
\caption{For different values of $n$ and $q$, and $10 \, 000$ random starting points, we count the number of instances that each of the algorithms converged to the global minimum $x^*=(-q-1,\ldots,-q-1)^T$ of the function $\varphi_q$ in~\eqref{eq:func_phi_q}. All algorithms are particular instances of BDSA.}\label{tab:counterexample}
\end{table}
The most remarkable fact is that BPDCA converged to the optimal point $x^*$ in every single instance. By contrast, its non accelerated version PDCA very rarely managed to reach the optimum (1.17\% of the overall instances). On the other side, DGA also got trapped very often by non-optimal critical points, only converging to the global minimum in 0.59\% of the instances. Its accelerated version BDGA greatly improved this poor result and converged 11.62\% of the times to the optimal solution.

For illustration, we display in Figure~\ref{fig:counterexample} for $n=2$ and $q=3$ the sequences generated by  DGA and PDCA from the starting point $x^0= (1.5,-0.5)^T$, using $y^0= (0,0)^T$ as the dual starting point for DGA. In addition, we show the iterates obtained after accelerating both methods with the boosted linesearch scheme proposed in Algorithm~\ref{alg:1}. We observe that PDCA got caught by $(0,-1)^T$, which is the nearest critical point that it encountered, while DGA converged to the slightly better critical point $(-1,-1)^T$. On the other hand, both BPDCA and BDGA managed to converge to $x^*=(-4,-4)^T$.

\begin{figure}[ht!]\centering
 \includegraphics[height=.4\textwidth]{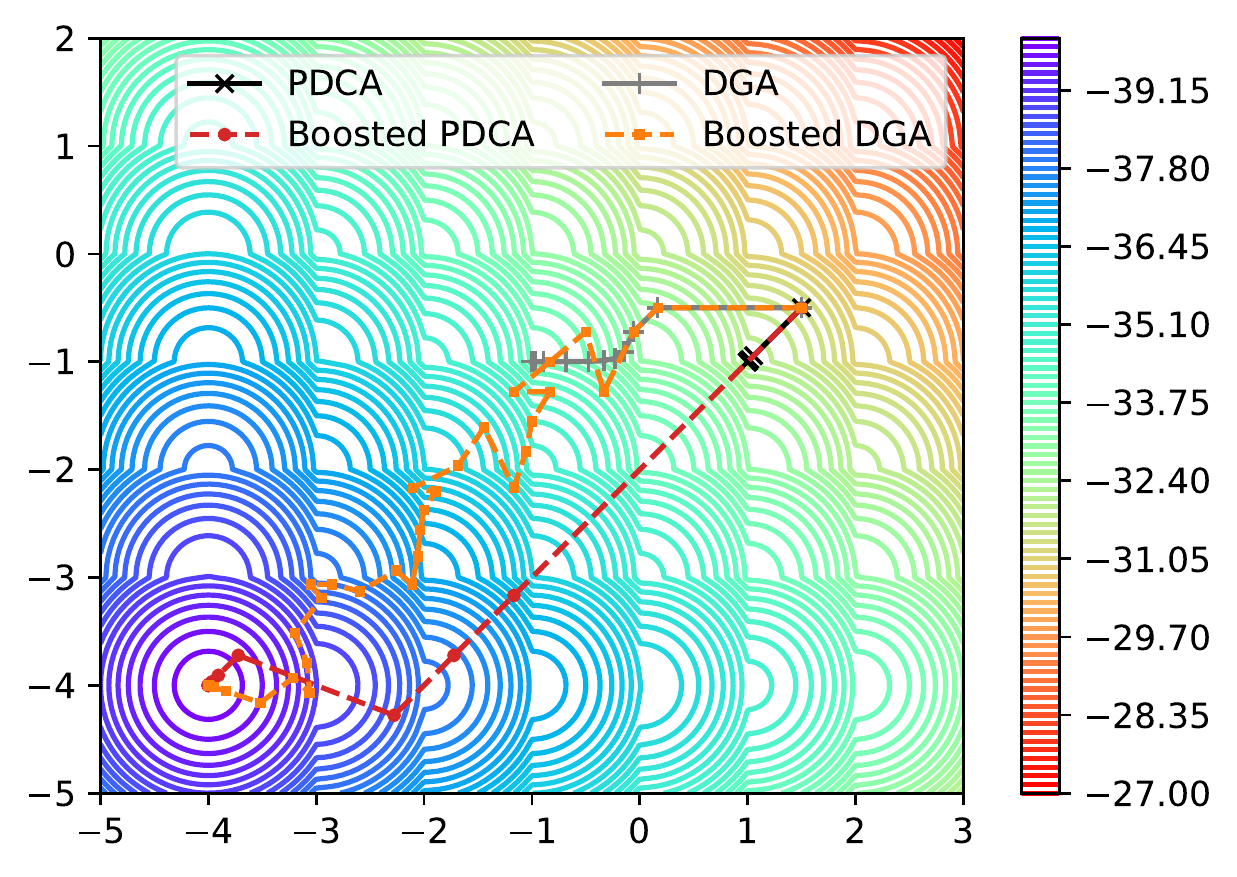}
 \caption{Sequence of iterates generated by PDCA, DGA and their boosted versions for the same starting point when they were applied to the function $\varphi_3$ in $\R^2$. Only the boosted versions manage to converge to the global optimum $x^*=(-4,-4)^T$.}\label{fig:counterexample}
\end{figure}

The fact that PDCA and DGA have such a low rate of success in reaching the global minimum is an indicator of how challenging the proposed family of functions is for this type of algorithms. The advantage of the boosted versions of the algorithms for this family is clear. Even so, it is important to mention that although the linesearch in BDGA only succeeded to improve its success rate up to 11.62\%, it consistently improved the objective values. BDGA converged to a point with lower objective value than DGA in 46.61\% of the instances, while both algorithms attained the same value in the remaining in 53.39\%. DGA did not surpass BDGA in any of the 60\,000 instances.

\subsection{Minimum Sum-of-Squares Constrained Clustering Problem}\label{sect:numerical2}
\emph{Clustering analysis} is a widely-employed technique in data science for classifying a collection of objects into groups, called \emph{clusters}, whose elements share similar characteristics. In order to mathematically describe the clustering problem, we can think of our data as a finite set of points $A=\left\{a^1,\ldots,a^q\right\}$ in $\mathbb{R}^s$. Our goal is to group $A$ into $\ell$ disjoint subsets $A^1,\ldots,A^\ell$, based on the minimization of some clustering measure.

In the \emph{minimum sum-of-squares clustering problem}, the groups are determined by the minimization of the squared  Euclidean distance of each data to the \emph{centroid} of its cluster. In this way, each cluster $A^j$ is identified by its centroid, which we denote by $x^j\in\mathbb{R}^s$, for $j=1,\ldots,\ell$. Letting $X:=(x^1,\ldots,x^{\ell})\in\mathbb{R}^{s \times \ell}$, this clustering problem can be reformulated as the optimization problem.

\begin{equation}\label{eq:DC-clustering}
 \min_{X\in\mathbb{R}^{s\times \ell}} f(X) := \frac{1}{q}\sum_{i=1}^q \omega_i(X),
\end{equation}
where $\omega_i(X):= \min\left\{  \|x^j-a^i\|^2 : j=1,\ldots, \ell  \right\}.$ The function $f$ is $1$-upper-$\cC^2$, since each of the functions $\omega_i$ is $1$-upper-$\cC^2$ (simply by definition).

In~\cite{boostedDCA}, the authors considered the  clustering problem \eqref{eq:DC-clustering} with the aim of grouping the 4001 Spanish cities in the peninsula with more than 500 residents.
In this work, we consider a more challenging version of the above problem in which we add a nonconvex constraint on $X$. This is useful for example when the centroids represent facilities (e.g., hospitals or government administrations). In this case, the centroids cannot be located in the sea, or even in certain areas that should be avoided. Therefore, we are interested in solving the problem
\begin{equation}\label{eq:const-clustering}
	\min_{X\in C} f(X). 
\end{equation}
where $C\subseteq\mathbb{R}^{s\times \ell}$ is the newly introduced (not necessarily convex) constraint. This allows us to make the experiment in~\cite{boostedDCA} more challenging, in the following way:
\begin{itemize}
	\item We consider the cities with more than 500 residents in the Spanish peninsula, but also those in the Balearic Islands, which is an archipelago in the Mediterranean Sea. They sum a total of 4049 cities.
	\item We exclude a region in the center of Spain as a possible location for centroids, which would be useful if decentralization policies were aimed.
	\item We exclude Portugal, which is also contained in the same peninsula as Spain.
\end{itemize}
The resulting closed nonconvex constraint is depicted in Figure~\ref{fig:constraint}.

Now, considering the objective function of problem \eqref{eq:const-clustering} as a large sum of nonsmooth functions, the sum rule for the basic subdifferential only offers an upper estimation rather than an equality. Consequently, it becomes more convenient to compute subgradients of individual functions $\omega_i$ instead of examining the entire function $f$. In this context, the following proposition formally provides the computation of the subdifferential of the functions $\omega_i$, for $i=1,\ldots,q$.

\begin{proposition}
	Given $a\in\R^s$,  consider the function $\omega :\mathbb{R}^{s\times \ell} \to \R$ given by
\begin{equation*}
	\omega(X):= \min\left\{  \|x^j-a\|^2 : j=1,\ldots, \ell  \right\},
\end{equation*}
with $X=(x^1,\ldots,x^{\ell})\in\mathbb{R}^{s \times \ell}$. Then, the following  formula holds
\begin{align}\label{basic_subd_f_i}
	\partial \omega(X)= \left\{    (0,\ldots, 0, \underbrace{2(x^j - a)}_{
		j\text{-th position} }, 0 , \ldots,0)     : \omega(X) = \| x^j - a\|^2         \right\}, \text{ for all } X\in \mathbb{R}^{s\times \ell }.
\end{align}
	\end{proposition}
\begin{proof}
  To prove \eqref{basic_subd_f_i} let us notice that the inclusion $\subseteq$ follows from the calculus rule for the minimum  function (see, e.g., \cite[Proposition 1.113]{MR2191744}). To prove the opposite inclusion, we recall that by \eqref{eq:subinclusions} the following equality holds \begin{equation}\label{eq_conv}
  	\conv	\partial \omega(X) = -\partial (-\omega)(X).
  \end{equation}
 Now, since $-\omega$ is a maximum  of quadratic forms, we can apply \cite[Theorem 3.46]{MR2191744} to $-\omega$ to conclude that
   \begin{align*}
   	-\partial (-\omega)(X) = \conv B(X),
   \end{align*}
where $B(X)$ is the set in the right-hand side of \eqref{basic_subd_f_i}. Finally, since  all the points  in the set $B(X)$ are linearly independent, we get that $\supset$ must hold in \eqref{basic_subd_f_i}, as otherwise it would contradict \eqref{eq_conv}.
\end{proof}

Based on the aforementioned observation, we are motivated to present Algorithm \ref{alg:3:clustering} as a well-suited variant of the Boosted Double-proximal Subgradient Algorithm for effectively addressing the constrained clustering problem \eqref{eq:DC-clustering}. When $C$ is defined by linear inequality constraints, observe that feasibility of the direction $D^k$ defined in Step~3 can be checked as in~\cite[Algorithm~1]{AragonArtacho2022} (see also Lemma~3.1 there), so the boosting in Step~5 is only run when $D^k$ is in the cone of feasible directions. The set-valued \emph{projector} onto the set $C\subseteq\R^m$ is denoted by $P_C:\R^m\rightrightarrows \R^m$.

\begin{algorithm}
	\caption{Boosted  proximal Subgradient Algorithm for constrained clustering}\label{alg:3:clustering}
	\begin{algorithmic}[1]
		\Require{$X^0 \in\mathbb{R}^{s\times \ell}$, $R\geq 0$, $\rho\in{]0,1[}$ and $\alpha\geq 0$. Set  $k:=0$.}
	
	\State{Choose $v_i^{k}\in \partial \omega_i(X^k)$ for $i=1, \ldots,q$ and set $V^k = \frac{1}{q} \sum_{i=1}^q v_i^k$ .} 
\State{Take some positive $\gamma_k<\frac{1}{2}$ and compute
	\begin{equation*}
		\begin{aligned}
			\hat{X}^{k} & \in P_C \left(X^k-\gamma_k V^k\right) .
		\end{aligned}
	\end{equation*}
}

\State{Choose any $\overline{\lambda}_k\geq 0$. Set $\lambda_{k}:=\overline{\lambda}_{k}$, $r:=0$ and $D^{k} :=\hat{X}^{k}  -X^{k}$.}
\State{\textbf{if} $D^{k}=0$ \textbf{then} STOP and return $x^k$.}
\State{\textbf{while} $r<R$ and
	\begin{align*}
	 \hat{X}^{k} + \lambda_{k}D^{k} \notin C \text{ or }	f\big(\hat{X}^{k} + \lambda_{k}D^{k}  \big)   >  		f(\hat{X}^{k})- \alpha \lambda_k^2 \| D^{k}\|^2
	\end{align*}%
		\textbf{do }{$r:=r+1$ and $\lambda_k:=\rho^r\overline{\lambda}_k$.}}
	
	\State{\textbf{if }$r=R$ \textbf{then} $\lambda_k:=0$.}

	\State{Set $X^{k+1}  :=\hat{X}^{k} +  \lambda_{k}D^{k}$, $k:=k+1$ and go to Step 1.}
\end{algorithmic}
\end{algorithm}

Now, in order to present our convergence result for Algorithm \ref{alg:3:clustering} we define a suitable notion of critical point:  we say that $\bar X$ is a \emph{critical point} of the constrained clustering problem \eqref{eq:const-clustering} if
\begin{align*}\label{eq:clustering:critical}
	 0 \in \frac{1}{q} \sum_{ i=1}^q \partial \omega_i(\bar{X}) + N_C(\bar{X}),
\end{align*}
where $N_C$ denotes the (\emph{basic}) \emph{normal cone} to $C$, which coincides with $\partial\iota_C$ (see, e.g.,\cite[Proposition~1.79]{MR2191744}).

\begin{corollary}\label{th:conv:clustering}
Given $ X^0 \in\mathbb{R}^{s\times \ell }$ and $\eta\in{]0,1[}$, consider the   sequence   $(X^k)_{k\in\mathbb{N}}$ generated by Algorithm~\ref{alg:3:clustering} with  $\gamma_k\in{\left]0,\frac{\eta}{2}\right]}$ for all $k\in\N$. Then, either Algorithm~\ref{alg:3:clustering} stops at a critical point of~\eqref{eq:const-clustering} after a finite number of iterations or it generates an infinite sequence $ (X^k)_{k\in\mathbb{N}}$ such that the following assertions hold:
	\begin{enumerate}[(i)]
		\item  The sequence $\big(f(X^k )\big)_{k\in\mathbb{N}}$ monotonically (strictly) decreases and converges, and $X^k \in C$ for all $k \geq 1$. Moreover,  the sequences $(X^k)_{k\in\mathbb{N}}$ verifies that
		\begin{equation*}\label{eq:Ostrowski:clustering}
			\sum_{k=0}^{\infty} \|X^{k+1}-X^k\|^2 < \infty.
		\end{equation*}
		\item  If the sequence $(X^k)_{k\in\mathbb{N}}$ is bounded, the set of its accumulation points is nonempty, closed and connected.
		\item  If $\bar{X}\in\mathbb{R}^{s\times \ell }$ is an accumulation point of the sequence $(X^k)_{k\in\mathbb{N}}$, then $\bar{X}$   is a critical point of~\eqref{eq:const-clustering}. In addition, $f(\bar{X}) = \inf_{k\in\mathbb{N}}f(X^k)$.
		\item If $(X^k)_{k\in\mathbb{N}}$ has at least one isolated accumulation point, then the whole sequence $(X^k)_{k\in\mathbb{N}}$ converges to a critical point of~\eqref{eq:const-clustering}.
	\end{enumerate}
\end{corollary}
\begin{proof}
	To prove this corollary,  let us first notice that every subgradient $V^k $ belongs to $\conv \partial f(X^k)$. Indeed,  using \cite[Theorem 3.46(ii)]{MR2191744} we have that each function $-\omega_i$ is lower regular at any point. Hence, using \eqref{eq:subinclusions} and the sum rule for the basic subdifferential we get that
	\begin{align*}
	\conv 	\partial f (X) = -	\partial (-f) (X) = -\frac{1}{q}\sum_{ i=1}^q  \partial (-\omega_i)(X) =\frac{1}{q}\sum_{ i=1}^q \conv \partial \omega_i(X) \supset  \frac{1}{q}\sum_{ i=1}^q  \partial \omega_i(X).
	\end{align*}
Therefore, using Remark~\ref{Remark_sub} we get the result, where the justification that $\bar{X}$ is a critical point of~\eqref{eq:const-clustering} follows from \eqref{general_critical_point}.
\end{proof}

\begin{figure}[ht]\centering
	\includegraphics[height=.5\textwidth]{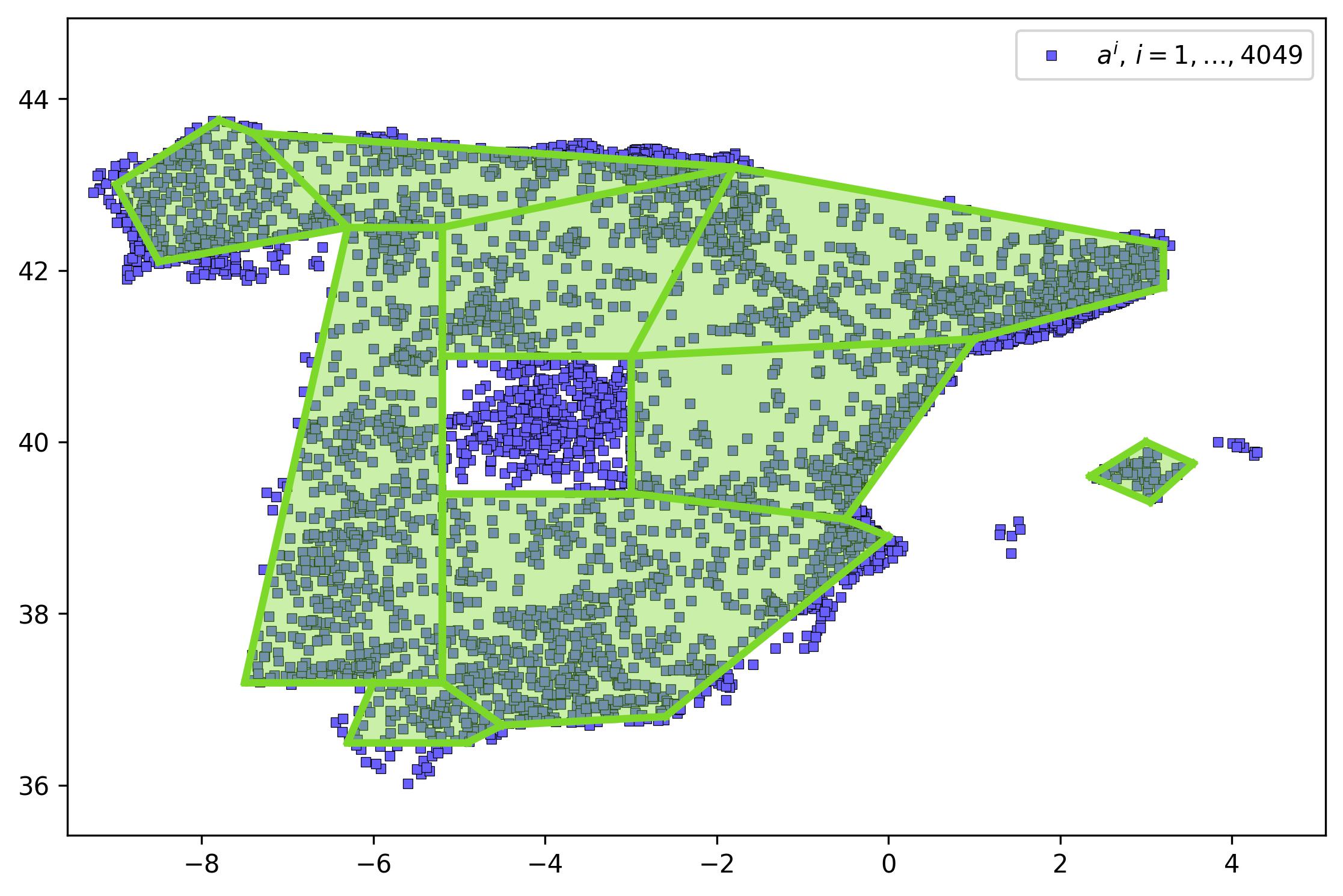}
	\caption{The blue squares represent the 4049 cities of Spain peninsula and Balearic Islands with more than 500 inhabitants. In order to accurately gather all the area of Spain including these cities, we build our constraint $C$ as the union of a finite number of shaded polyhedral sets. Note that the rectangle in the center of Spain is excluded.}\label{fig:constraint}
\end{figure}
\begin{figure}[ht]\centering
 \includegraphics[height=.375\textwidth]{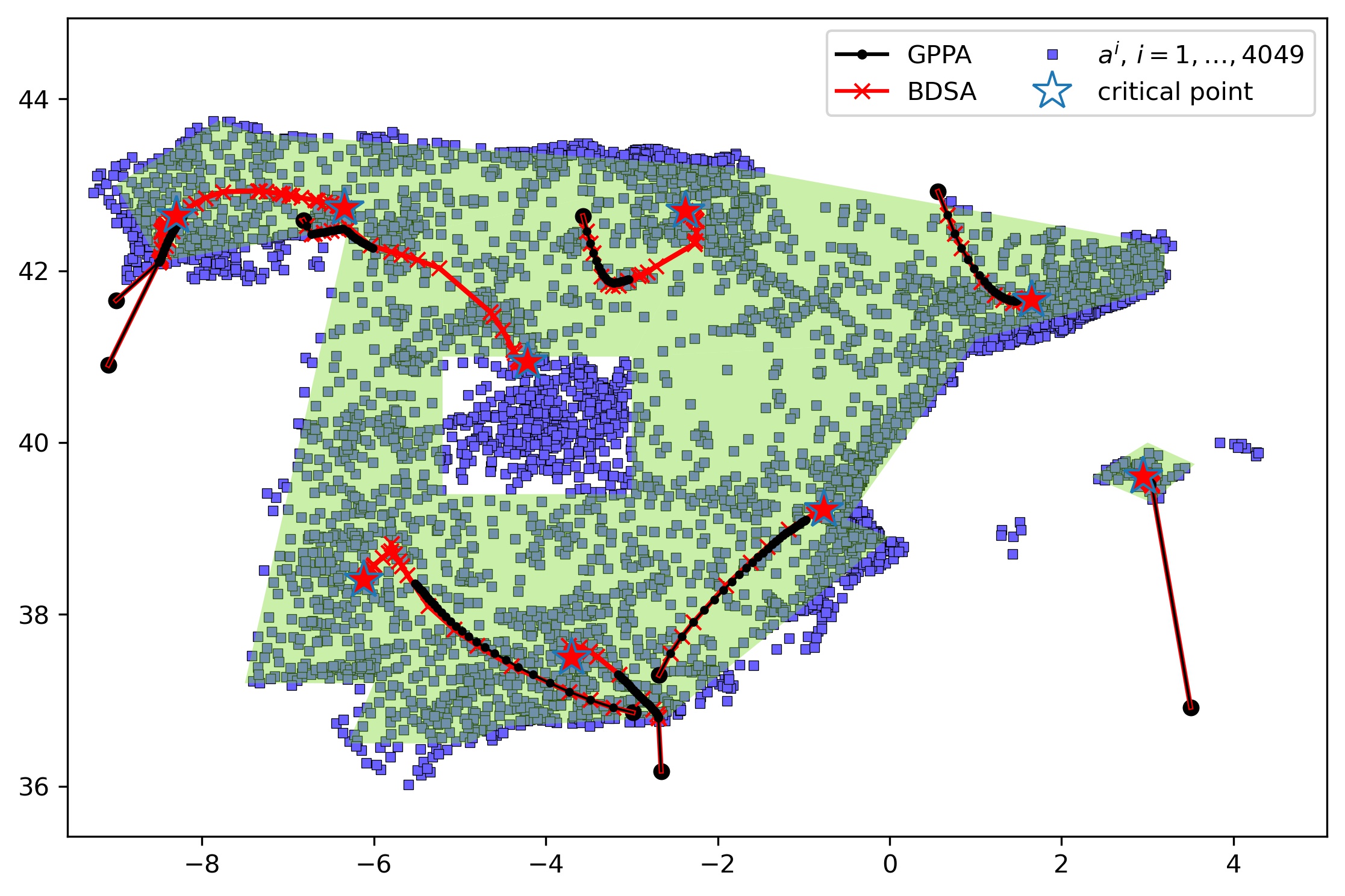}
 \includegraphics[height=.375\textwidth]{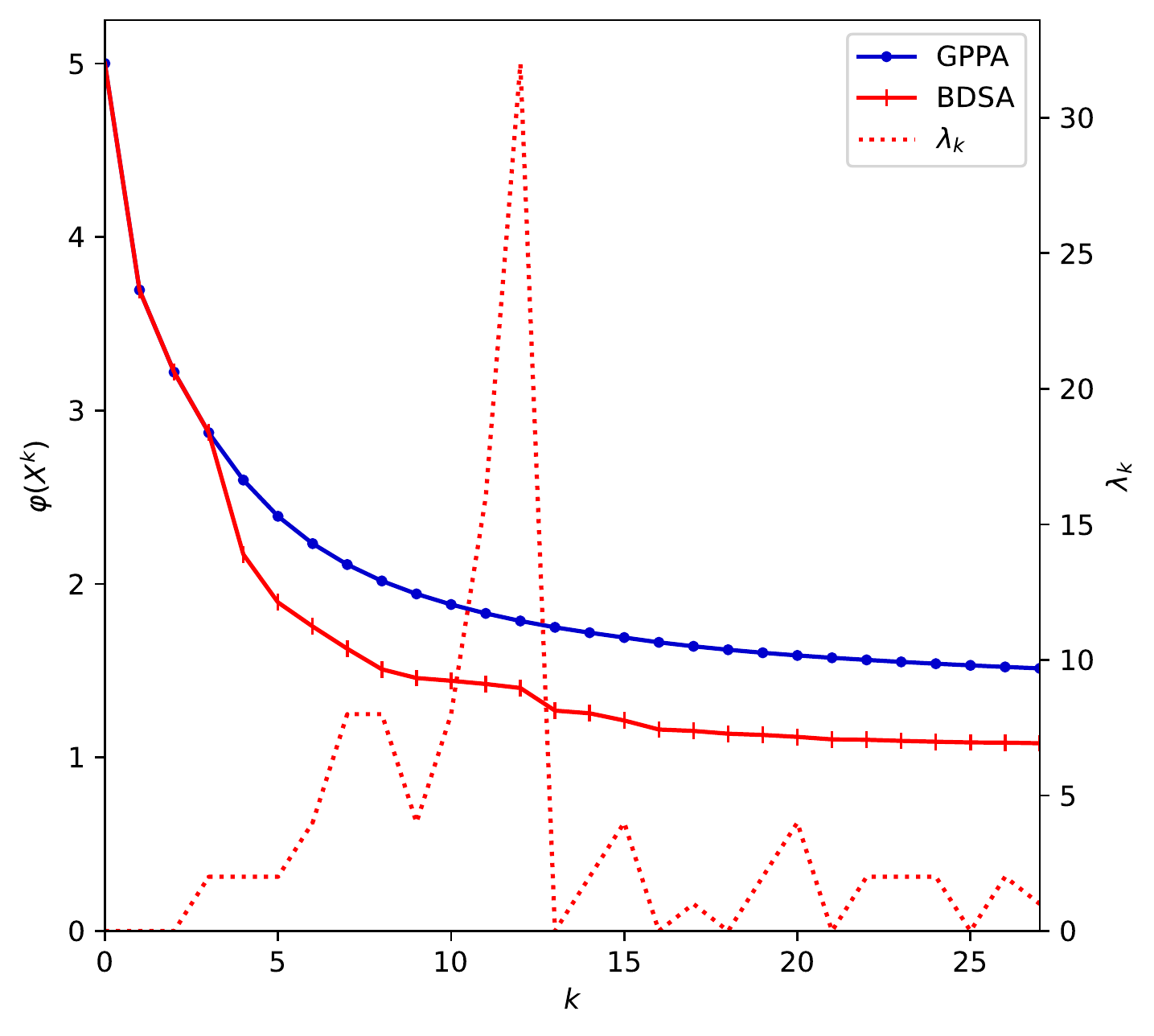}
\caption{
We ran GPPA and BDSA from the same initial random point, for finding $9$ centroids satisfying the constraints shown in Figure~\ref{fig:constraint}. Both algorithms converged to the same critical point, but BDSA only needed 27 iterations. On the left, 27 iterations are drawn on the map. On the right, we compare the value of $\varphi$ for GPPA and BDSA along the iterations. The figure has two vertical axes, the right one corresponds with the stepsize taken by BDSA in every iteration, represented with a dotted line.} \label{fig:clustering9}
\end{figure}

In our first experiment, we aim to find a partition into $9$ clusters of the $4049$ cities in consideration. For this experiment, we set $\gamma := 0.9 \times \left( 1/2\right)$ and ran both GPPA and BDSA, starting from the same initial random point, until BDSA reached a relative error in the objective function (i.e., $|f(X^{k+1})-f(X^k)|/f(X^{k+1})$)  smaller than $10^{-3}$. This stopping criterion was achieved after $27$ iterations, represented in Figure~\ref{fig:clustering9}, with a value of the objective function of $1.0822$. Simultaneously, GPPA returned an objective value of $1.5138$ for the same number of iterations. In particular, Figure~\ref{fig:clustering9} (left) demonstrates significant progress made by BDSA in moving towards regions with a higher concentration of cities. Meanwhile, GPPA required $116$ iterations before reaching the same relative error. Figure~\ref{fig:clustering_decrease} presents an illustrative example where both algorithms converge to distinct critical points. In this occasion, GPPA converged to a critical point with an objective function value of $3.2670$ after $30$ iterations, while BDSA only needed $16$ iterations to converge to a point with a superior value of $2.1827$. The  plots on both Figure~\ref{fig:clustering9} and~\ref{fig:clustering_decrease} give more insight into how the linesearch helps the boosted algorithm to achieve a more significant decrease in the objective value.

\begin{figure}[ht]\centering
 \includegraphics[height=.38\textwidth]{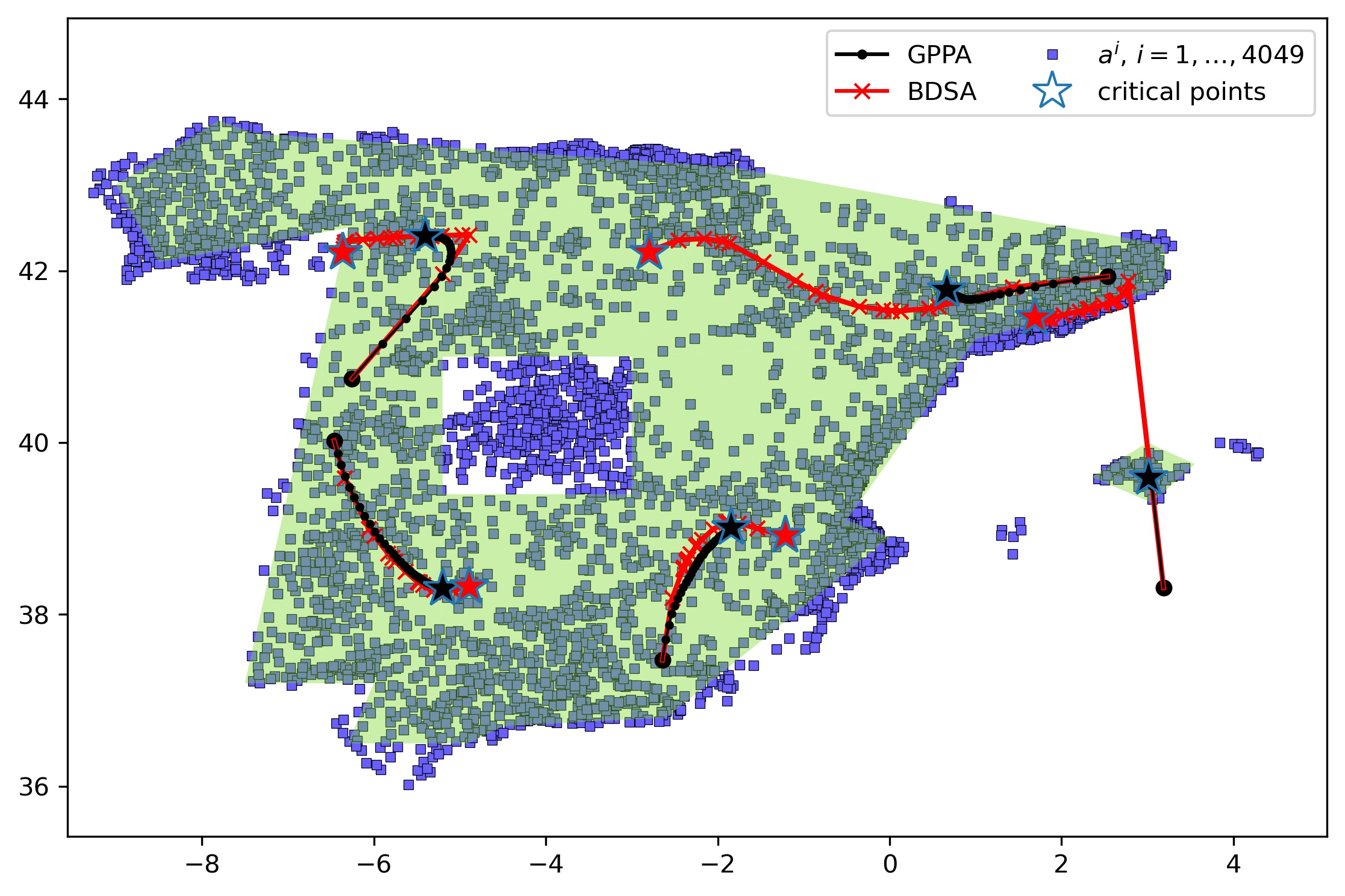}
 \includegraphics[height=.38\textwidth]{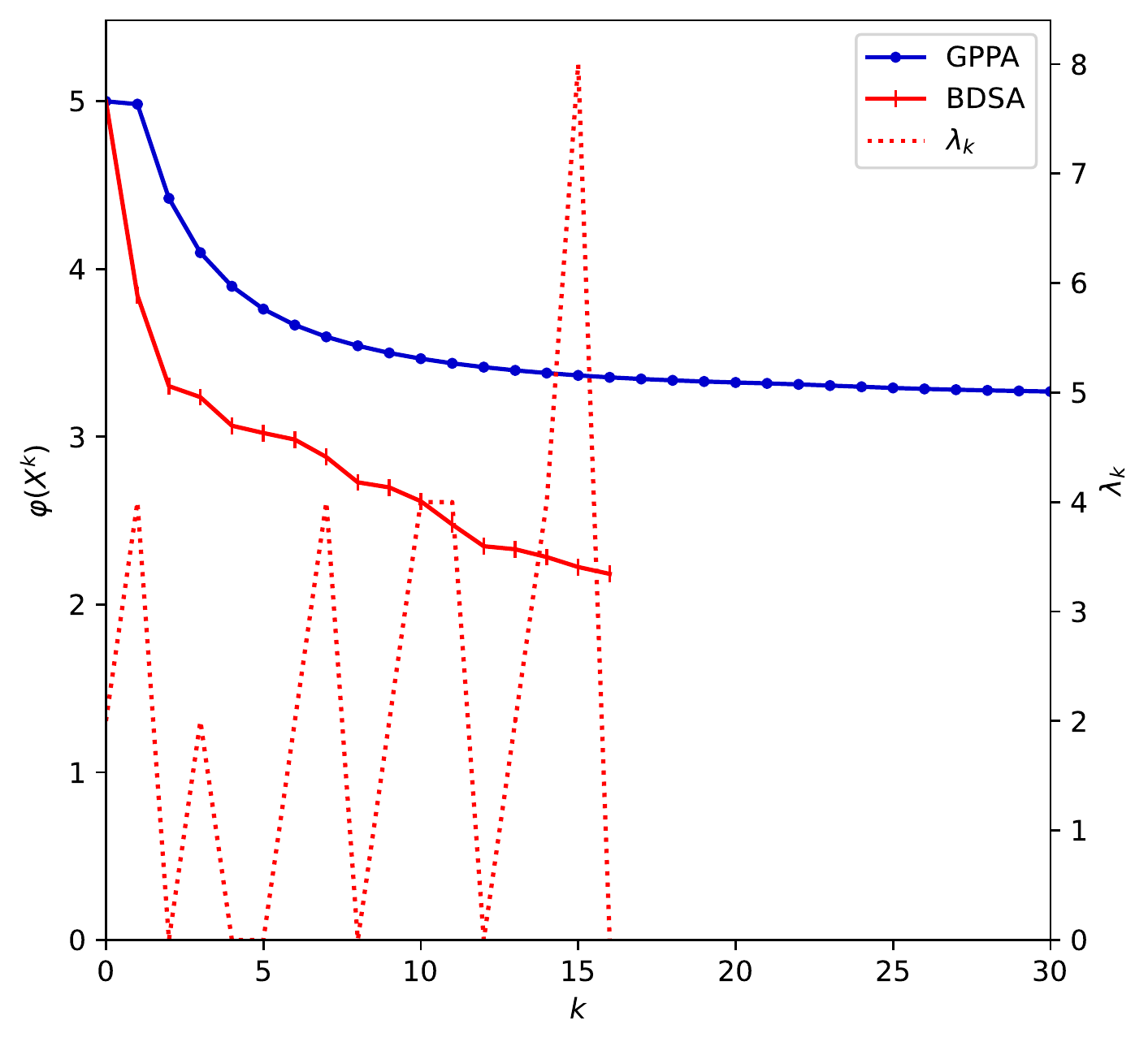}
 \caption{We ran GPPA and BDSA from the same initial random point, for finding $5$ centroids satisfying the constraints shown in Figure~\ref{fig:constraint}. The algorithms converged to different critical points, BDSA in 16 iterations and GPPA in 30. On the left, the iterations are drawn on the map. On the right, we compare the value of $\varphi$ for GPPA and BDSA along the iterations. The figure has two vertical axes, the right one corresponds with the stepsize taken by BDSA in every iteration, represented with a dotted line.}\label{fig:clustering_decrease}
\end{figure}

To demonstrate that this is the general trend, we solved the same problem with the Spanish cities for a different number of clusters $\ell\in{\{3,5, 10, 15, 20, 30, 40, 50\}}$. The results are summarized in Figure~\ref{fig:clusteringF}. For each of these values and for $10$ different random starting points, we ran BDSA until the relative error in the objective function was smaller than $10^{-3}$. Then, GPPA was ran from the same starting point until it reached the same objective value than BDSA or until the relative error in the objective function was smaller than $10^{-3}$. In particular, GPPA failed to reach the same value than BDSA in $68$ out of the $80$ runs. In Figure~\ref{fig:clusteringF} we present the iterations ratio (left) and  time ratio (right) of both algorithms for all the instances. On average, BDSA was approximately $3$ times faster with respect to the number of iterations than GPPA, and  more than $2$  times faster in time. There was only one instance where GPPA was faster than BDSA (both in time and iterations), for $\ell=5$, but the value of $\varphi$ at the stopping point for GPPA was $1.43$ times larger than that of BDSA. Therefore, the increase in the running time and iterations of BDSA was a consequence of converging to a better critical point than GPPA.

\begin{figure}[ht!]\centering
 \includegraphics[height=.3\textwidth]{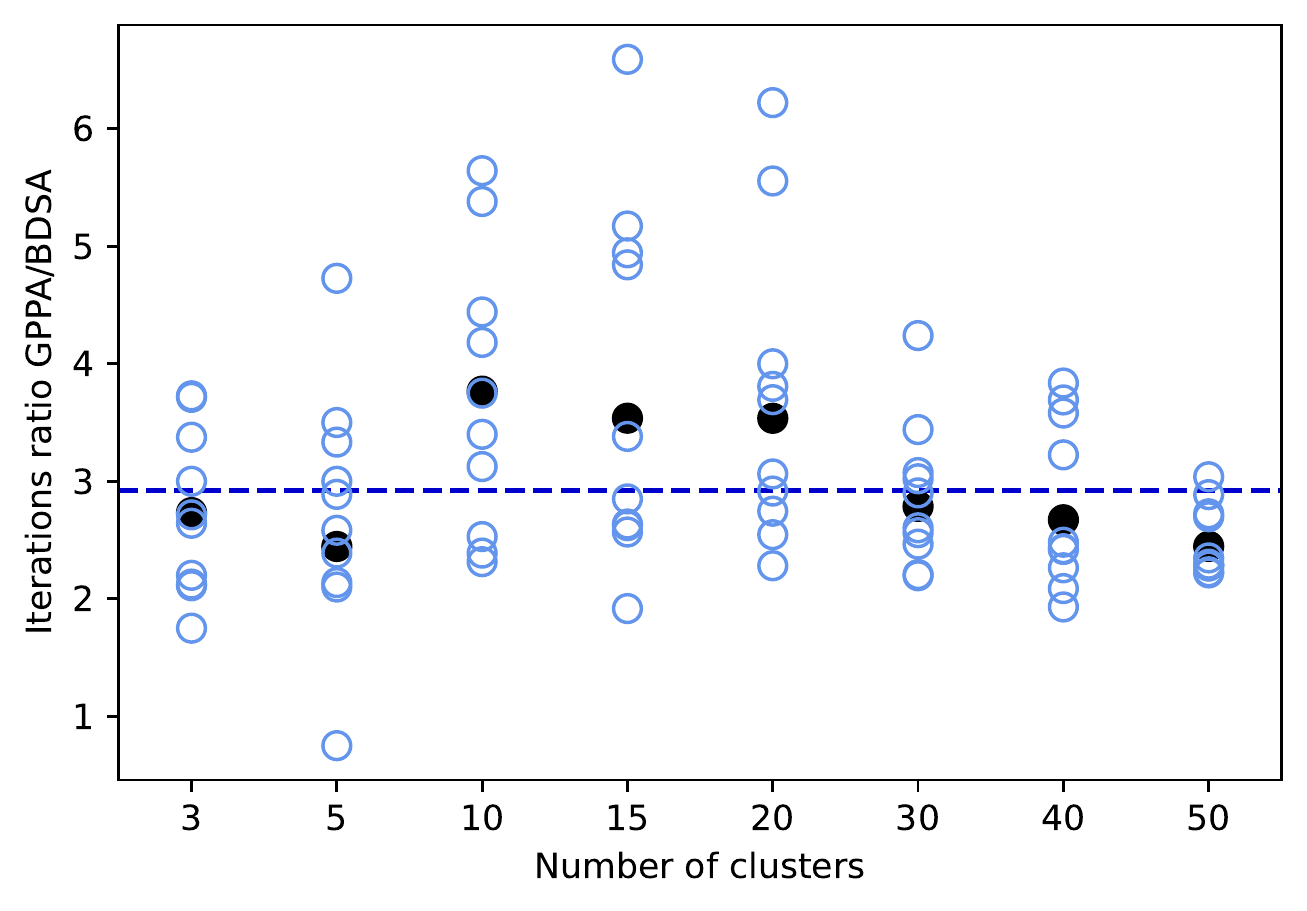}\hfill
 \includegraphics[height=.3\textwidth]{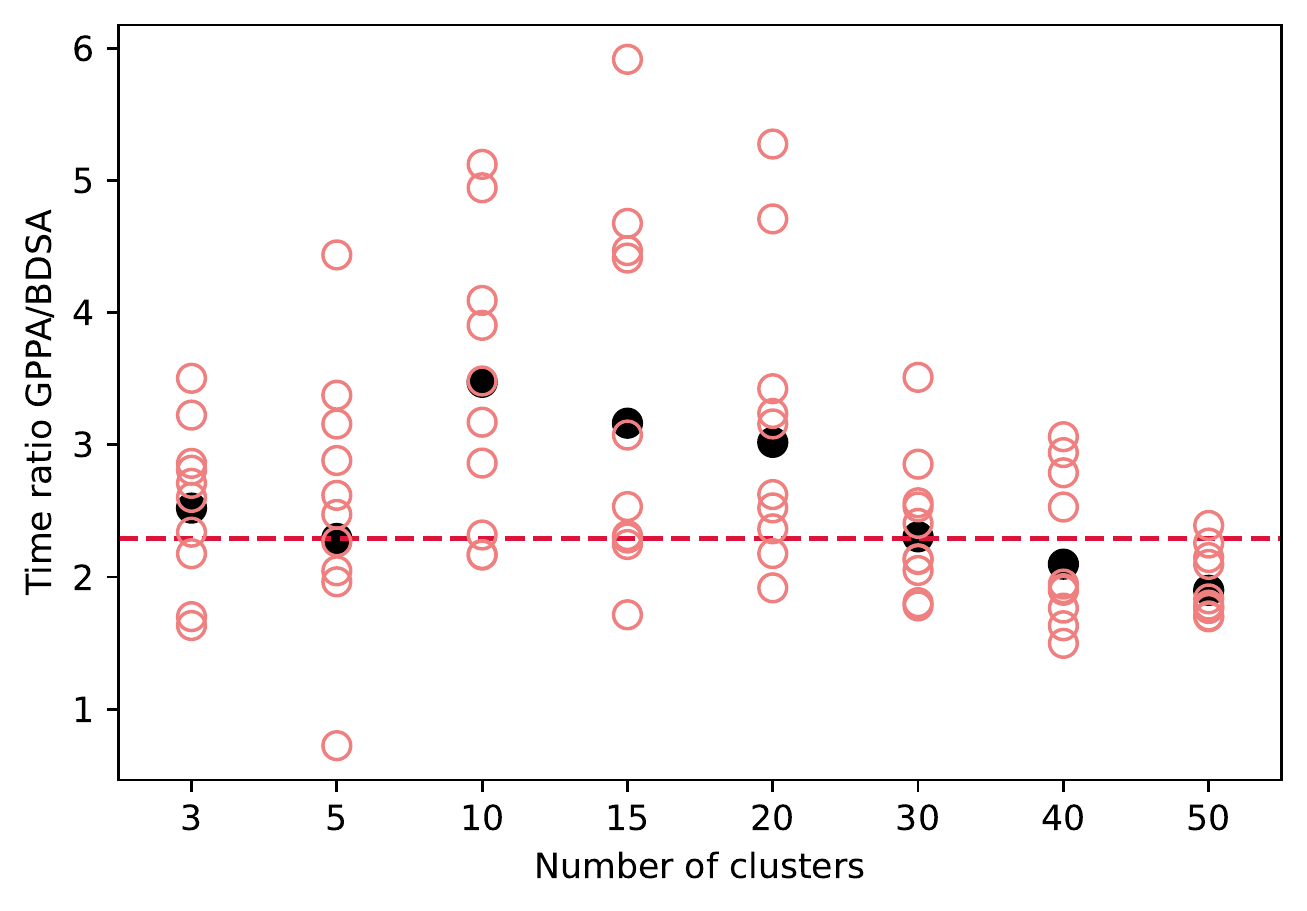}
 \caption{Iteration and time ratios between GPPA and BDSA for solving problem~\eqref{eq:const-clustering} with different number of clusters. Both algorithms were run until  they reached  the same relative error, or until GPPA reached the objective value obtained by BDSA. The unfilled markers show the ratio for every particular instance, the black dots represent the ratio average for a fixed  number of clusters  and the dashed line the overall ratio average.} \label{fig:clusteringF}
\end{figure}

	\subsection{A Nonconvex Generalization of Heron's Problem}\label{sect:numerical3}

	The original formulation of Heron's problem consists in the following: given a straight line in the plane, find a point $x$ in it such that the sum of the distances from $x$ to two given points is minimal. A generalization of Heron's problem to a Euclidean space of arbitrary dimension $\mathbb{R}^n$ was introduced in~\cite{mordukhovich2012solving},  where the line was substituted by a  closed convex set $C_0$ and the two given points by a finite family of closed convex sets $\{C_i\}_{i=1}^p$. This convex problem was then solved by means of a projected subgradient algorithm. Lately, different splitting methods have also been employed to tackle this generalization of Heron's problem~\cite{bot2013douglas,campoy2022product}.
	
	In this subsection, we go one step ahead and consider a more general version of the problem. Specifically, we seek to minimize a weighted sum of the squared distance of the images of $x$ by certain differentiable functions $\Psi_i:\mathbb{R}^n\to\mathbb{R}^{m_i}$ with Lipschitz continuous gradients, for $i=1,\ldots,p$. Namely, given some closed (but not necessarily convex) sets $C_0\subseteq \R^n$ and $C_i\subseteq\R^{m_i}$, for $i=1,\ldots,p$, we are interested in solving the following nonconvex generalization of Heron's problem
\begin{align}\label{eq:GQHP}
	\min_{x \in C_0} \sum_{i=1}^p \frac{w_i}{2}d^2 (\Psi_i(x),C_i), 
\end{align}
where $w_i > 0$ represents a weight associated to the $i$-th constraint.

Problem~\eqref{eq:GQHP} can be easily reformulated as an instance of~\eqref{eq:P1}. Indeed, as shown in Example~\ref{prop:Asplund}, the  squared distance function admits the following decomposition:
\begin{align*}
	 \frac{1}{2} d^2 (\Psi_i(x),C_i)= \frac{1}{2} {\| \Psi_i(x)\|^2 } - \Asp{C_i}{}(\Psi_i(x)).
\end{align*}
Hence, problem~\eqref{eq:GQHP} is equivalent to the unconstrained problem
\begin{equation*}
 \min_{x\in\mathbb{R}^n} \iota_{C_0}(x) + \sum_{i=1}^p w_i \left( \frac{1}{2}\|\Psi_i(x)\|^2 - \Asp{C_i}(\Psi_i(x))\right).
\end{equation*}
It is clear that this problem can be expressed in the form of~\eqref{eq:P1} by choosing $f:= \sum_{i=1}^p  \frac{w_i}{2}\|\Psi_i(\cdot)\|^2$, $g=\iota_{C_0}$ and $h_i = w_i\Asp{C_i}$, for $i=1,\ldots,p$. Note that although the Asplund function $\Asp{C_i}$ is always convex, $\Asp{C_i} \circ \Psi_i$ may not be convex if $\Psi_i$ is not linear. Therefore,  $\frac{w_i}{2} d^2 (\Psi_i(x),C_i)$ is not necessarily upper-$\cC^2$.

In the following we analyze the performance of DSA and BDSA for the particular instance of the problem in which $w_i:=1$ for all $i=1,\ldots,p$, $C_0$ is the closed ball of radius $r_{C_0} := 5$ in $\mathbb{R}^n$, and the soft constraints $C_i$ are hypercubes of edge length $2$. To avoid intersections with $C_0$, the centroids of the hypercubes were randomly generated with norm between $7$ and $10$. We set all $\Psi_i := \Psi:\mathbb{R}^n\to \mathbb{R}^m$, with
\begin{equation*}
 \Psi(x) := \left( x^TQ_1x, x^TQ_2x, \ldots, x^TQ_mx \right)^T,
\end{equation*}
where, for simplicity, we chose $Q_1,\ldots,Q_m$ as diagonal matrices with randomly generated entries in $]-1,1[$. Note, that the gradient of  $\Psi$ is the linear transformation given by
\[
 \nabla \Psi(x) = 2(Q_1x,Q_2x,\ldots,Q_mx),
\]
which is Lipschitz continuous with constant $L_{\Psi}:=2 \rho(Q) $, where $\rho(Q)$ denotes the spectral radius of $Q:=(Q_1,Q_2,\ldots,Q_m)$. On the other hand, note that $\nabla f$ is also $L_f$-Lipschitz in the ball $C_0$ for $L_f:=6pr_{C_0}^2\rho(Q)\sqrt{\sum_{i=1}^m\rho(Q_i)^2}$, since
\begin{align*}
\|\nabla f(x)-\nabla f(y)\|&=\|p\left(\nabla\Psi(x)\Psi(x)-\nabla\Psi(y)\Psi(y)\right)\|\\
&\leq p\|\nabla\Psi(x)\Psi(x)-\nabla\Psi(x)\Psi(y)\|+p\|\nabla\Psi(x)\Psi(y)-\nabla\Psi(y)\Psi(y)\|\\
&\leq p\|\nabla\Psi(x)\|\sqrt{\sum_{i=1}^m(x^TQ_ix-y^TQ_iy)^2}+p\|\Psi(y)\|L_{\Psi}\|x-y\|\\
&\leq 2\rho(Q)p\left(\|x\|\sqrt{\sum_{i=1}^m4r_{C_0}^2\rho(Q_i)^2}+\max_{z\in C_0}\|\Psi(z)\|\right)\|x-y\|\\
&\leq 2\rho(Q)p\left(\|x\|2r_{C_0}\sqrt{\sum_{i=1}^m\rho(Q_i)^2}+\max_{z\in C_0}\sqrt{\sum_{i=1}^m\|z\|^4\ \rho(Q_i)^2}\right)\|x-y\|\\
&\leq 6\rho(Q)pr_{C_0}^2\sqrt{\sum_{i=1}^m\rho(Q_i)^2}\|x-y\|=L_f\|x-y\|.
\end{align*}
Therefore, according to~\eqref{eq:descentlemmaineq} and Proposition~\ref{p:deslemma}, the function $f$ is $L_f /2$-upper-$\cC^2$.

In order to fairly illustrate the advantages of the linesearch step in BDSA, we initially perform an experiment to find some adequate performing parameters for DSA (i.e., when no linesearch was performed).
In all the experiments in this subsection, we stopped the algorithms when
\begin{equation}\label{eq:stop}
|\Phi(x^{k+1},y^{k+1})-\Phi(x^k,y^k)|< 10^{-6}.
\end{equation}

\paragraph{Tuning the parameters for DSA}

We set $n=3$, $m=4$ and $p=3$, and ran  Algorithm~\ref{alg:1}  with different choices of parameters for $5$ randomly generated problems and $5$ different starting points for each problem (i.e., 25 instances in total). Having in mind the bounds for the parameters given in Theorem~\ref{th:conv}, we tested the algorithm for every combination of the following choices:
\begin{align*}
\gamma_k&:=\eta\left( L_f + L_{\Psi} \sum_{i=1}^p \|y_i^k\| \right)^{-1},\text{ with }
 \eta \in {\{0.1, 0.3, 0.5, 0.7, 0.9, 0.99\}},\\
\mu^k_i &:=\mu\in{\{0.5, 1, 5\}},\text{ for all } k\geq 0.
 \end{align*}
For every combination of stepsize parameters the algorithm obtained the same value of $\Phi$ in the last iterate. However, there is a considerable variability in the number of iterations  needed for reaching the stopping criterion, which we show in Table~\ref{fig:ParamSplitting}. Note that the parameter~$\gamma_k$ is the only one providing significant differences, being $\eta= 0.99$ the best performing value. The parameter $\mu$ does not seem to have much influence in the results obtained.
\begin{table}[ht!]\centering
\begin{tabular}{lcccccc}
\toprule
 & $\eta=0.1$ & $\eta=0.3$ & $\eta=0.5$ & $\eta=0.7$ & $\eta=0.9$ &  $\eta=0.99$\tabularnewline
\midrule
$\mu= 0.5$ & 8\,013.0 & 3\,104.8 & 1\,989.0& 1\,486.6 & 1\,201.5&  1\,104.9\tabularnewline
  $\mu = 1$ &   8\,012.4 & 3\,103.7 & 1\,987.6 &1\,485.0&1\,199.9 &1\,103.4\tabularnewline
  $\mu = 5$ &  8\,011.7 &3\,102.9&1\,986.5& 1\,483.9 & 1\,198.6& 1\,102.1\tabularnewline
\bottomrule
\end{tabular}\caption{Average number of iterations of DSA for $5$ random problems~\eqref{eq:GQHP} and $5$ random starting points for each problem, with $n=3$, $m=4$, $p=3$ and different values of the parameters.}\label{fig:ParamSplitting}
\end{table}

\paragraph{DSA vs. BDSA: Benefit of linesearches}

Now we compare both versions of Algorithm~\ref{alg:1} with the  best choice parameter $\eta=0.99$. Since $\mu$ does not have much effect, we set a small value $\mu^k=0.5$, which is more likely to satisfy the bound in Proposition~\ref{l:descentdirection}(iii).
We tested DSA and BDSA for different values of $n$, $m$ and~$p$. Both algorithms obtained similar results regarding the objective values, showing  only differences after the second decimal, in favor of BDSA in all but one instance, so we only present the results regarding number of iterations (without counting those needed for the linesearch) and the running time.
The results are summarized in Figures~\ref{fig:SplitVSBoosted} and~\ref{fig:SplitVSBoosted_r}, where we observe that BDSA clearly outperformed DSA in each of the $120$ instances. In particular, BDSA was on average more than $2.5$ times faster than DSA.

\begin{figure}[ht!]\centering
 \includegraphics[height=.32\textwidth]{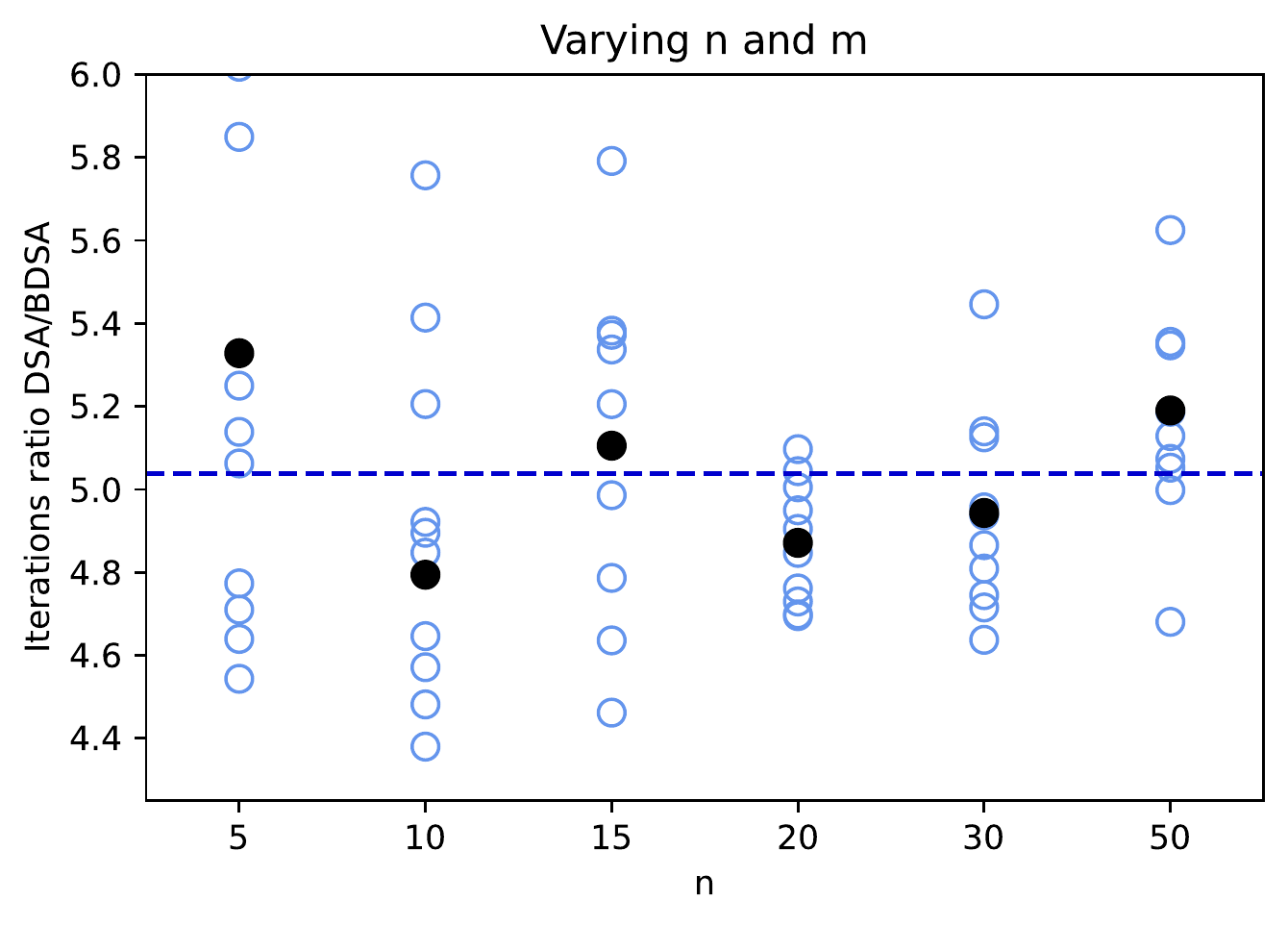}\hfill
 \includegraphics[height=.32\textwidth]{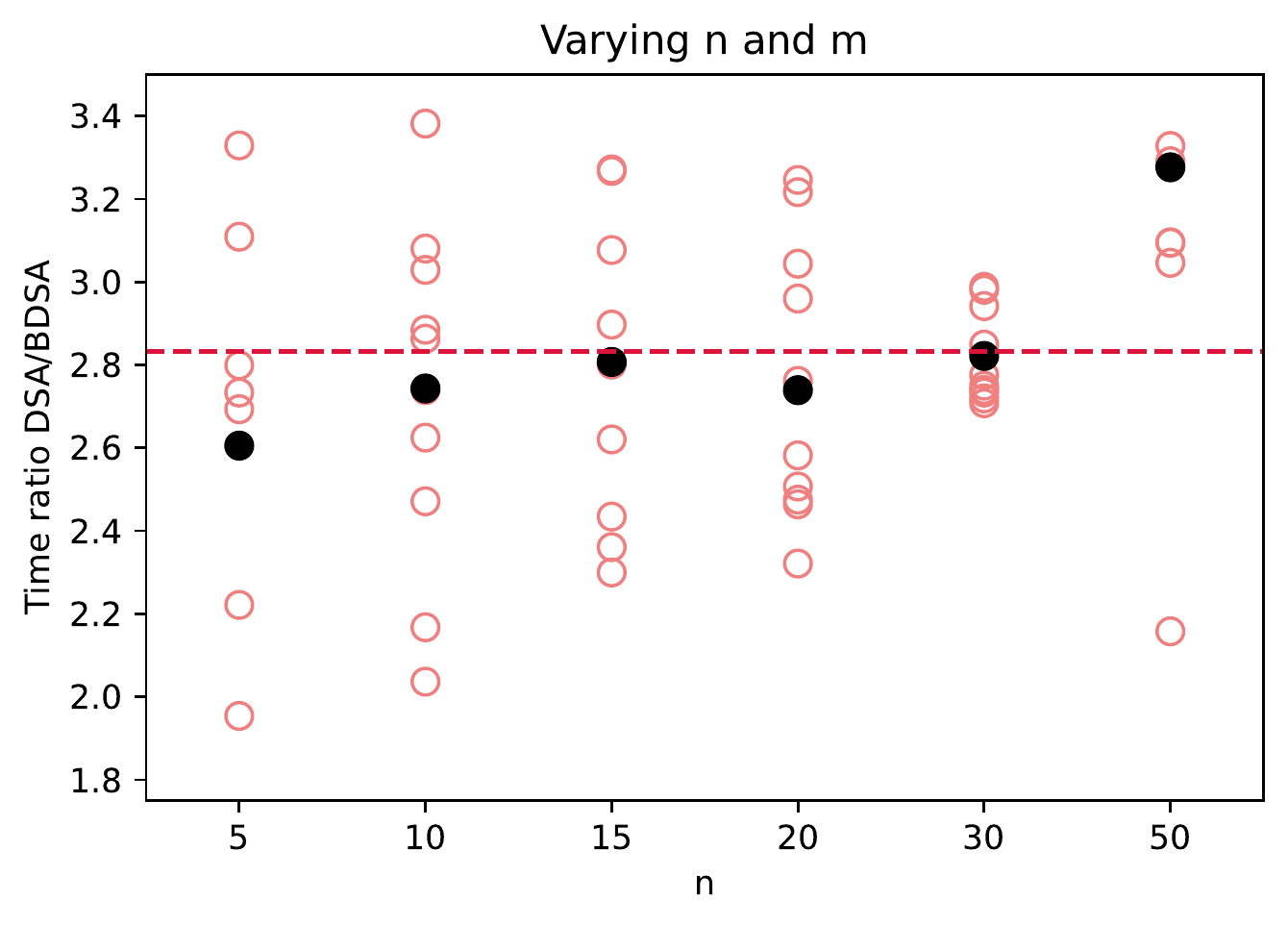}
 \caption{For $p=3$, each $n\in{\{5,10,15,20,30,50\}}$ and $m= 1.2\,n$, we randomly generated $10$ different problems and ran DSA and BDSA initialized at the same random starting point. We plot with circles the ratio of the number of iterations (left) and the running time (right). The black dots represent the average ratio for a fixed  $n$  and the dashed line the overall average ratio. }\label{fig:SplitVSBoosted}
\end{figure}

\begin{figure}[ht!]\centering
 \includegraphics[height=.32\textwidth]{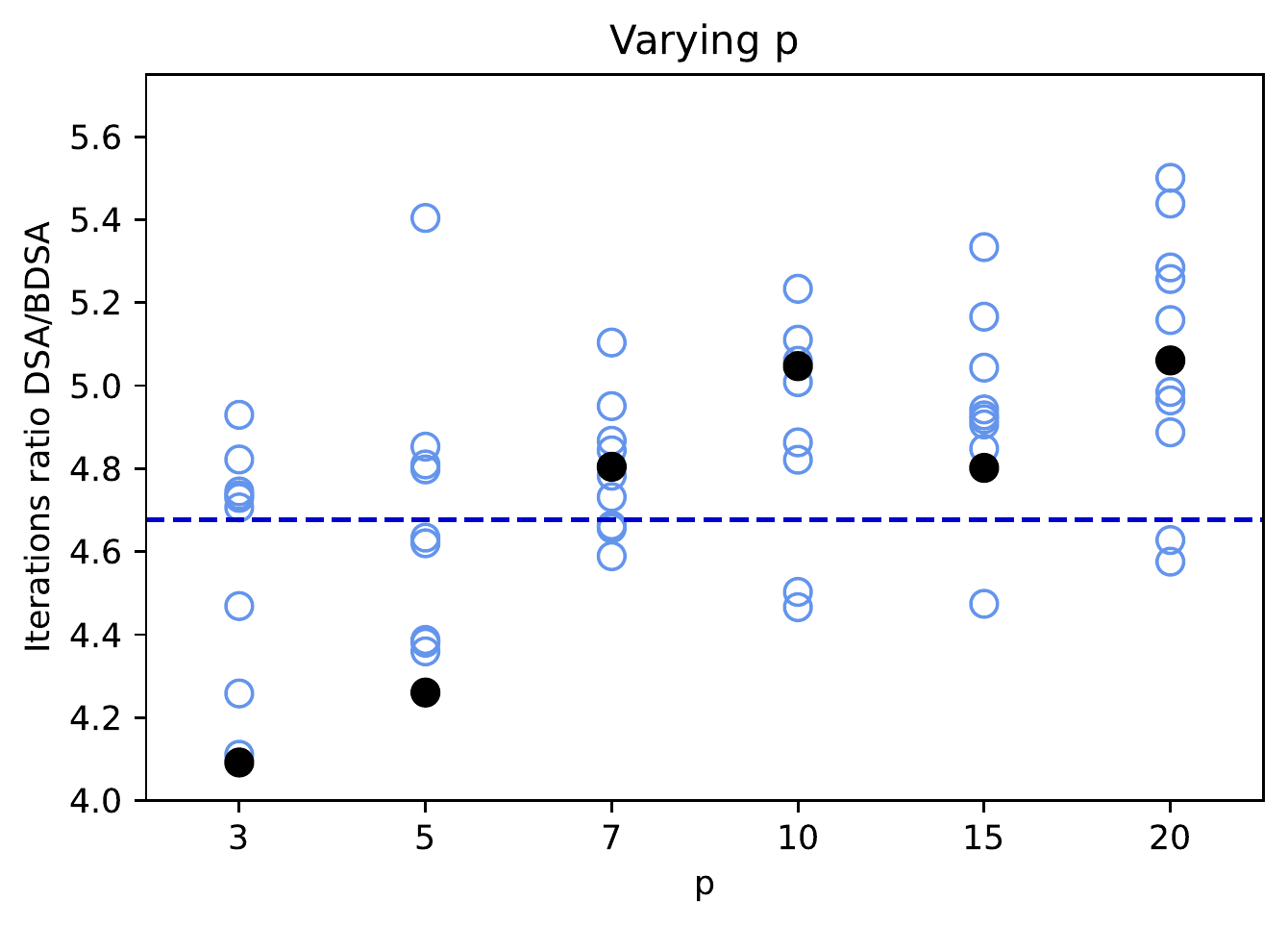}\hfill
 \includegraphics[height=.32\textwidth]{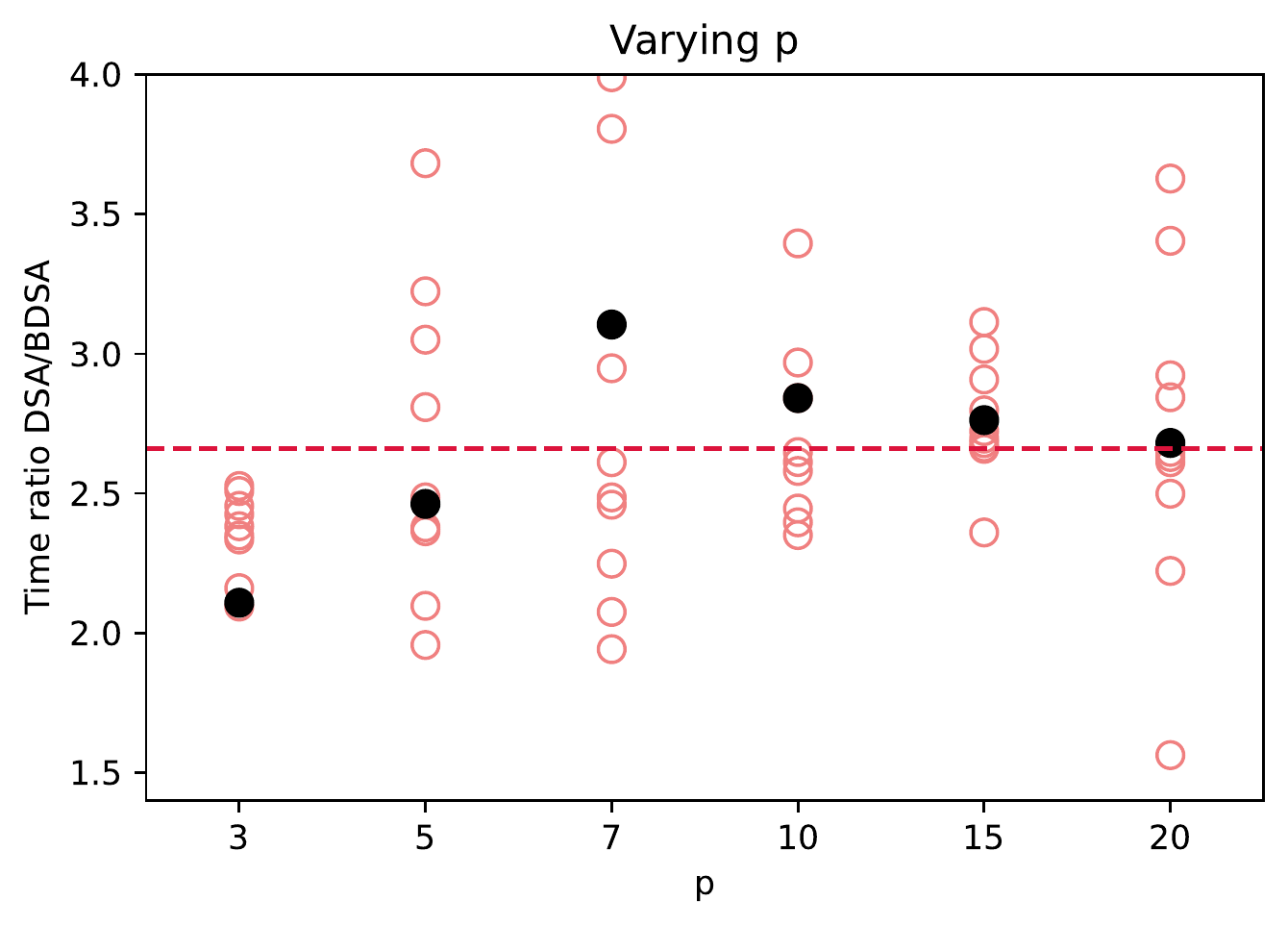}
 \caption{For $n=20$, $m=16$ and each $p\in{\{3, 5, 7, 10, 15, 20\}}$, we randomly generated $10$ different problems and ran DSA and BDSA initialized at the same random starting point. We plot with circles the ratio of the number of iterations (left) and the running time (right). The black dots represent the average ratio for a fixed  $n$  and the dashed line the overall average ratio. }\label{fig:SplitVSBoosted_r}
\end{figure}

\paragraph{The generalized Heron Problem with nonconvex sets} In our last experiment, we consider examples of~\eqref{eq:GQHP} in which the sets $C_i$ are not necessarily convex. Instances of the generalized Heron problem with nonconvex sets have already been studied for example in \cite{mordukhovich2012applications}.  In this experiment, we let  $\Psi_i$ be a linear mapping of the form $\Psi_i(x)=Qx$, with $Q\in\R^{m\times n}$, for all $i=1,\ldots,p$. We showed in Example~\ref{prop:Asplund} that $x\mapsto\frac{1}{2}d^2(Qx,C)$ is a $\kappa$-upper-$\cC^2$ function  with $\kappa =  \rho(Q)^2/2$. Moreover, note that by \cite[Theorem~5.3~(iii)]{aragonartacho2023coderivativebased} its subdifferential at a point $x\in\mathbb{R}^n$ is given by
\[
\partial \left(\frac{1}{2}d^2(Q(\cdot),C)\right)(x) = Q^T \bigg(Qx-P_{C}(Qx)\bigg).
\]
These two facts allow us to tackle~\eqref{eq:GQHP} when $C_i$  are not necessarily convex  as an instance of problem~\eqref{eq:P1}  by setting $f:= \sum_{i=1}^p \frac{w_i}{2}d^2(Q(\cdot),C_i)$, $g= \iota_{C_0}$ and $h=0$.

We work with the particular instance of~\eqref{eq:GQHP}  in which $p=1$, $w_1= 1$ and $C_1$ is given as the union of $5$ hypercubes which were generated in the same way  as in the previous experiment. The entries of $Q$ were randomly generated in the interval $]-1,1[$. As in the previous experiment, up to the authors' knowledge, in this setting BDSA does not recover any method already proposed in the literature. In Figure~\ref{fig:SplitVSBoosted_NC} we show the results of running both DSA and BDSA for $5$ different dimensions and $10$ different randomly generated problems for each dimension. In this case, BDSA reached a  better value of the objective function than DSA in  $49$ out of $50$ instances. Regarding the comparison in iterations and time, BDSA was again significantly faster: on average, DSA needed around $5$ times more iterations and $2.5$ more running time than BDSA to satisfy the stopping criterion~\eqref{eq:stop}.

\begin{figure}[ht!]\centering
 \includegraphics[height=.32\textwidth]{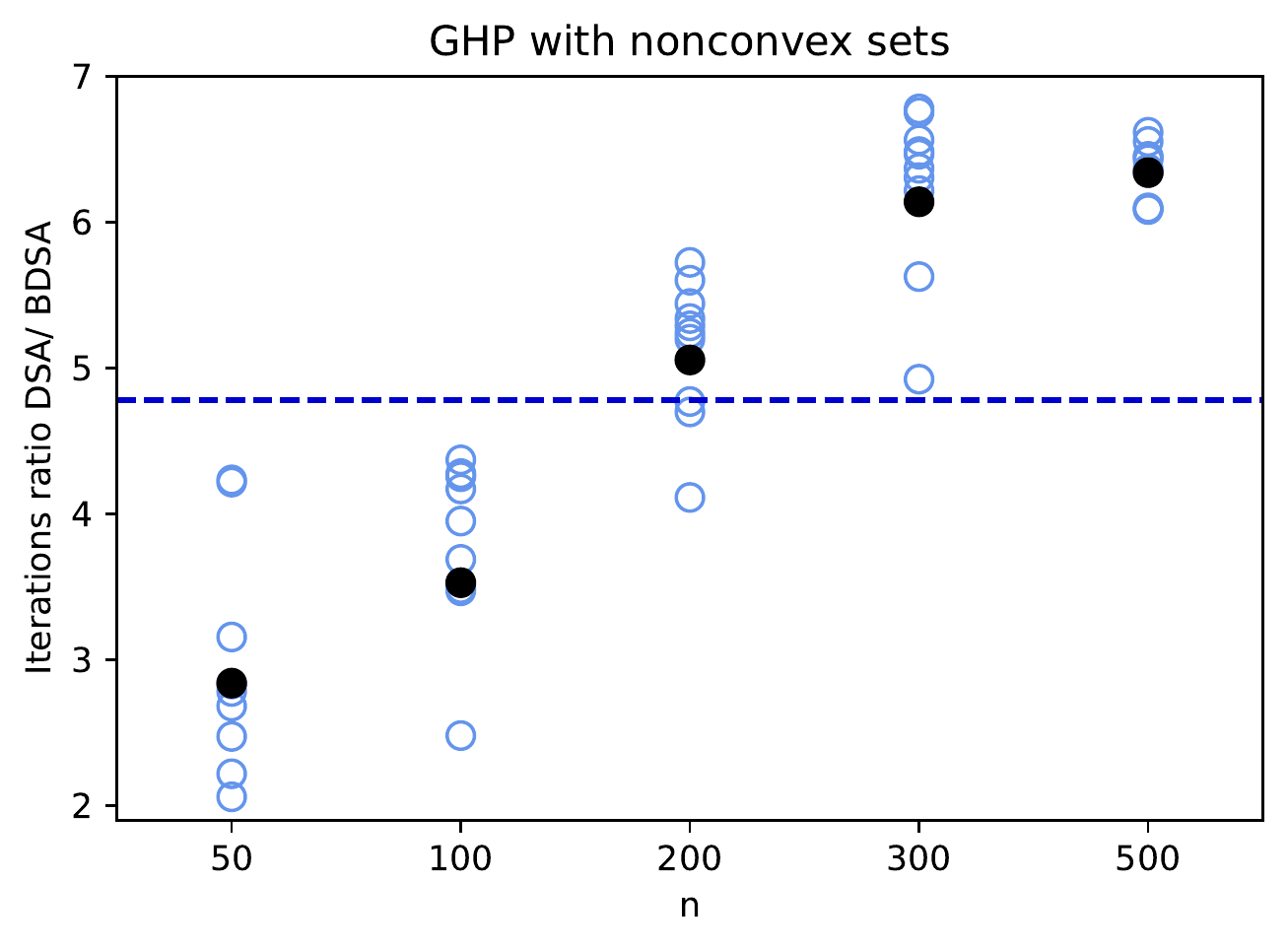}\hfill
 \includegraphics[height=.32\textwidth]{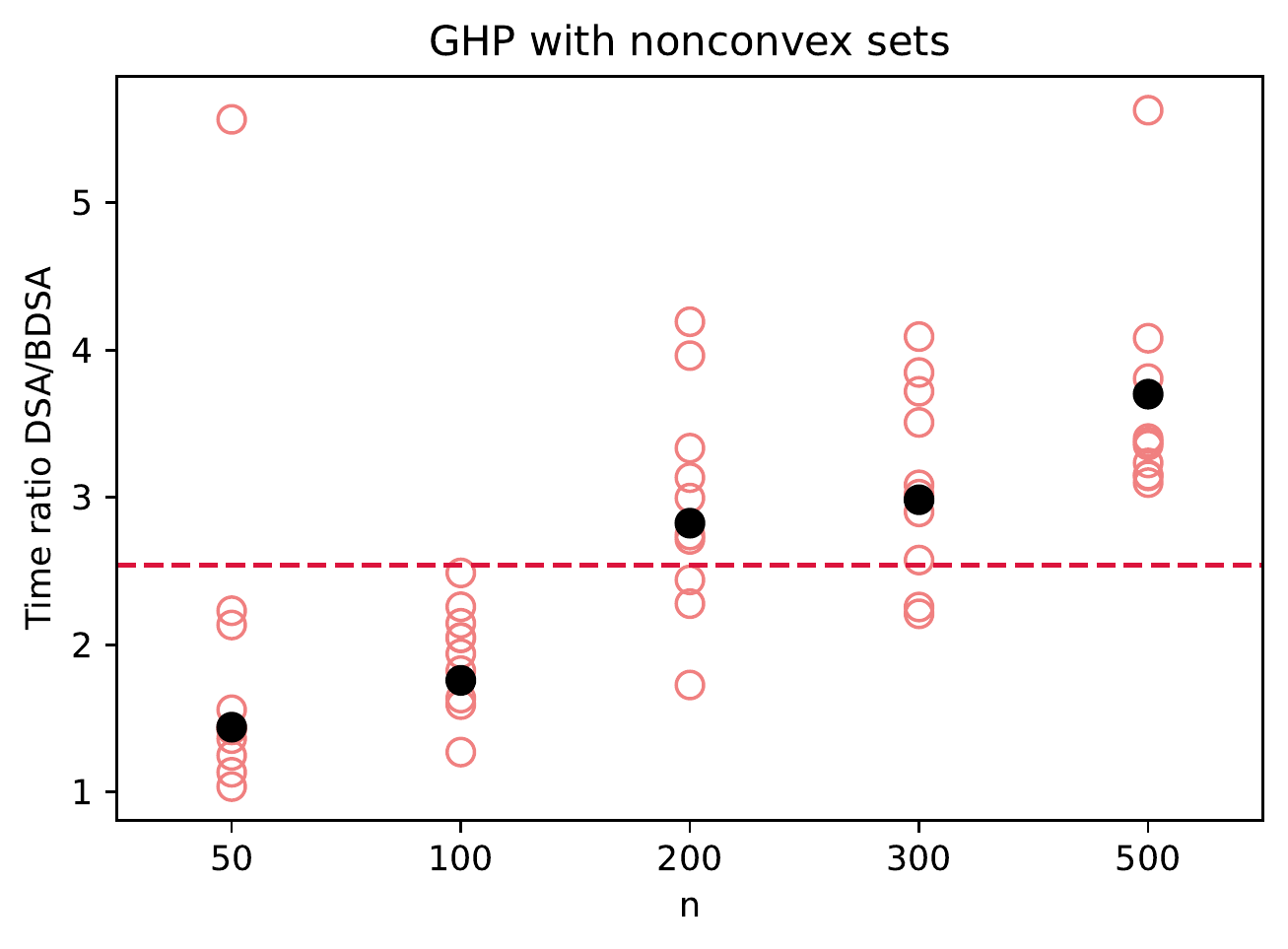}
 \caption{Let  $n\in{\{50, 100,200,300,500\}}$ and $m= 1.2\,n$. For every choice of dimensions we run both algorithms for $10$ different randomly generated problems and initialized in the same random starting point. The unfilled markers show  the ratio in iterations and  CPU time of DSA over BDSA for every particular instance, the black dots represent the ratio average for a fixed  $n$  and the dashed line the overall ratio average. }\label{fig:SplitVSBoosted_NC}
\end{figure}

\section{Conclusions and Future Work}\label{sec:conc}

In this paper, we developed a new algorithm for structured nonconvex optimization problems, which we named as \emph{Boosted Double-proximal Subgradient Algorithm} (BDSA). One of the main features of our method is the inclusion of a linesearch at the end of every iteration. If the stepsizes in every iteration of the linesearch are set to be $0$, then algorithms such as the \emph{Proximal Difference of Convex functions Algorithm} (PDCA)~\cite{sun2003proximal}, the \emph{Generalized Proximal Point Algorithm} (GPPA)~\cite{an2017convergence} and the \emph{Double-proximal Gradient Algorithm} (DGA)~\cite{bot2019doubleprox} can be recovered as particular cases of BDSA (see Remark~\ref{remark:particularcases}). Nevertheless, BDSA can also be applied to far more general problems.

The convergence of the sequence generated by BDSA is guaranteed under the usual assumptions required for this class of nonconvex problems. In addition, when the Kurdyka--{\L}ojasiewicz property holds, global convergence and convergence rates can be derived.

We demonstrated the advantages of the additional linesearch included in BDSA with multiple experiments. By recurring to a new family of test functions, we showed that BDSA remarkably improves  ``non-boosted'' methods in avoiding non-optimal critical points and achieving convergence to minima. Indeed, in our test problems BDSA had a $100\%$ rate of success, while the PDCA only reached the global minimum in $1.17\%$ of the instances.  Further, the boosting step significantly reduced the running time and the number of iterations employed by the algorithm. For example, the GPPA needed approximately $2.5$ more iterations and time for reaching the same accuracy than BDSA for an application of the minimum sum-of-squares clustering problem. For  different generalizations of the Heron problem, BDSA also managed to be much faster than its non-accelerated version, both in time and number of iterations.

To conclude, we point out two possible directions for future research.
\paragraph{Considering a nonmonotone linesearch} The authors in~\cite{ferreira2021boosted} have recently proposed a nonmonotone modification of the BDCA~\cite{aragon2018accelerating,AragonArtacho2022}. Their scheme is suitable for handling DC problems where both functions are nondifferentiable  by allowing the linesearch  to accept steps leading to some controlled growth in the objective function. It would be interesting to study whether a similar linesearch can be considered in Step $6$ of Algorithm~\ref{alg:1}.

\paragraph{Incorporating second-order information} The importance of including second-order information in algorithms for improved performance is widely well-recognized. For some applications, it can be crucial to extend our results by incorporating Hessian information into the data. This becomes particularly significant when dealing with data that is not twice continuously differentiable. Recent studies propose the integration of generalized Hessians (see, e.g., \cite{aragonartacho2023coderivativebased} and the references therein) to enable linesearches in Newton-like methods.

\bibliographystyle{siam}
\bibliography{references}
\end{document}

%% file: ex_shared.tex
\usepackage{lipsum}
\usepackage{amsfonts}
\usepackage{graphicx}
\usepackage{epstopdf}
\ifpdf
  \DeclareGraphicsExtensions{.eps,.pdf,.png,.jpg}
\else
  \DeclareGraphicsExtensions{.eps}
\fi

\usepackage{amsmath,amssymb,amsfonts,stmaryrd}
\usepackage{enumerate}
\usepackage{cite}
\usepackage[margin=3cm]{geometry}
\usepackage{hyperref}
\usepackage{siunitx}
\usepackage{graphicx}
\usepackage{mathtools}
\usepackage{array}
\usepackage{subcaption}
\usepackage{todonotes}
\usepackage{algorithm}
\usepackage{algorithmicx}
\usepackage{algpseudocode}

\providecommand{\tabularnewline}{\\}

\usepackage{booktabs}
\usetikzlibrary{calc}

\usepackage{multirow}
\usepackage{array}
\usepackage{subcaption}
\usepackage{todonotes}
\usepackage{dcolumn}

\DeclareMathOperator{\dom}{dom}

\DeclareMathOperator{\gra}{graph}

\DeclareMathOperator{\prox}{prox}

\DeclareMathOperator{\conv}{co}

\DeclareMathOperator*{\argmin}{argmin}

\newcommand{\Rex}{\overline{\mathbb{R}}}

\newcommand{\R}{\mathbb{R}}
\newcommand{\N}{\mathbb{N}}
\newcommand{\ball}{\mathbb{B}}

\newcommand{\cC}{\mathcal{C}}

\newcommand{\Rnm}{\mathbb{R}^{m}}

\newcommand*\barbf[1]{\bar{\mathbf{#1}}}
\newcommand*\hatbf[1]{\hat{\mathbf{#1}}}

\newcommand{\Asp}[1]{{{\mathtt{A}}_{#1 }}}

\let\epsilon\varepsilon
\let\subseteq\subset

\newsiamremark{remark}{Remark}
\newsiamremark{hypothesis}{Hypothesis}
\newsiamremark{assumption}{Assumption}
\newsiamremark{example}{Example}
\crefname{hypothesis}{Hypothesis}{Hypotheses}
\newsiamthm{claim}{Claim}

\headers{The Boosted Double-Proximal Subgradient Algorithm}{F.J. Arag\'on-Artacho, P. P\'erez-Aros, and D. Torregrosa-Bel\'en}

\title{The Boosted Double-Proximal Subgradient Algorithm \\ for Nonconvex Optimization\thanks{Submitted to the editors \today.
\funding{F. J. Arag\'on-Artacho and D. Torregrosa-Bel\'en were partially supported by the Ministry of Science, Innovation and Universities of Spain and the European Regional Development Fund (ERDF) of the European Commission, Grant PGC2018-097960-B-C22, and by the Generalitat Valenciana (AICO/2021/165).
P. P\'erez-Aros was supported by  Centro de Modelamiento Matem\'atico
(CMM), ACE210010 and FB210005, BASAL
funds for center of excellence and ANID-Chile grant: Fondecyt Regular 1200283 and Fondecyt Regular 1220886 and Fondecyt Exploración 13220097.
D. Torregrosa-Bel\'en was supported by MINECO and European Social Fund (PRE2019-090751) under the program ``Ayudas para contratos predoctorales para la formaci\'{o}n de doctores''
2019.}}}

\author{Francisco J. Arag\'on-Artacho \thanks{Department of Mathematics,
                             University of Alicante,
                             San Vicente del Raspeig Alicante, 03690, Spain
  (\email{francisco.aragon@ua.es, david.torregrosa@ua.es}).}
\and Pedro P\'erez-Aros\thanks{Instituto de Ciencias de la Ingeniería,
       			 Universidad de O'Higgins,
       			 Rancagua, 2820000, Chile
  (\email{pedro.perez@uoh.cl}).}
\and David Torregrosa-Bel\'en\footnotemark[2]}

\usepackage{amsopn}
